\documentclass{amsart}
\usepackage{amssymb,amsmath,amsthm,amscd}
\usepackage{mathtools}
\usepackage{esint}
\usepackage{mathrsfs}
\usepackage{mathabx}
\usepackage[T1]{fontenc}
\usepackage{hyperref}
\hypersetup{colorlinks,linkcolor={red},citecolor={blue}} 
\usepackage{enumerate}
\usepackage[final,notref,notcite]{showkeys}

\newtheorem{theorem}{Theorem}
\newtheorem{lemma}[theorem]{Lemma}

\newtheorem{proposition}[theorem]{Proposition}
\newtheorem{condition}[theorem]{Condition}

\numberwithin{theorem}{section}
\numberwithin{equation}{section}

\def\N {{\mathbb N}}
\def\Z {{\mathbb Z}}

\def\R {{\mathbb R}}

\DeclareMathOperator{\diff}{d\!}
\DeclareMathOperator{\Id}{Id}
\DeclareMathOperator{\loc}{loc}

\DeclarePairedDelimiter\abs{\lvert}{\rvert}
\DeclarePairedDelimiter\norm{\lVert}{\rVert}

\title{Duality for double iterated outer $L^p$ spaces}

\author{Marco Fraccaroli}
\address{Mathematisches Institut, Universit\"at Bonn, Endenicher Allee 60, 53115 Bonn, Germany}
\email{mfraccar@math.uni-bonn.de}


\begin{document}
	
\date{\today}
	
\begin{abstract}	
	We study the double iterated outer $L^p$ spaces, namely the outer $L^p$ spaces associated with three exponents and defined on sets endowed with a measure and two outer measures. We prove that in the case of finite sets, under certain conditions between the outer measures, the double iterated outer $L^p$ spaces are isomorphic to Banach spaces uniformly in the cardinality of the set. We achieve this by showing the expected duality properties between them. We also provide counterexamples demonstrating that the uniformity does not hold in any arbitrary setting on finite sets, at least in a certain range of exponents. We prove the isomorphism to Banach spaces and the duality properties between the double iterated outer $L^p$ spaces also in the upper half $3$-space infinite setting described by Uraltsev, going beyond the case of finite sets.	
\end{abstract}
	
\maketitle

\section{Introduction}

The theory of $L^p$ spaces for outer measures, or outer $L^p$ spaces, was introduced by Do and Thiele in \cite{MR3312633} in the context of time-frequency analysis. It provides a framework to encode the boundedness of linear and multilinear operators satisfying certain symmetries in a two-step programme. The programme consists of a version of H\"{o}lder's inequality for outer $L^p$ spaces together with the boundedness of certain embedding maps between classical and outer $L^p$ spaces associated with wave packet decompositions. This scheme of proof turns out to be applicable not only in time-frequency analysis, see for example \cite{MR4115625},\cite{MR4163524},\cite{2020arXiv200302742A},\cite{MR3873113},\cite{MR3829751},\cite{MR3596720},\cite{2016arXiv161007657U},\cite{Uraltsev},\cite{Warchalski}, but in other contexts too, see for example \cite{MR3841536},\cite{MR3875242},\cite{MR3312633},\cite{2020arXiv200105903F},\cite{mirek2017local},\cite{MR3406523}.

Although the theory of outer $L^p$ spaces comes in a broad generality of settings, the outer $L^p$ spaces used in \cite{MR3312633} are specifically defined by quasi-norms reminiscent in nature of iterated Lebesgue norms. In particular, the two Lebesgue norms involved in the definition of outer $L^p$ quasi-norms are associated with the two structures on a set provided by a measure and an outer measure. We recall that an outer measure $\mu$ on a set $X$ is a monotone, subadditive function from $\mathcal{P}(X)$, the power set of $X$, to the extended positive half line, attaining the value $0$ on the empty set. Similarly, in \cite{2016arXiv161007657U} Uraltsev considered outer $L^p$ spaces associated with three structures on a set, namely a measure and two outer measures, once again in the context of time-frequency analysis and in the spirit of the aforementioned two-step programme. Outer $L^p$ spaces associated with three structures where used in \cite{MR4115625},\cite{MR4163524},\cite{2020arXiv200302742A},\cite{MR3829751},\cite{2016arXiv161007657U},\cite{Uraltsev},\cite{Warchalski}.

As a matter of fact, one can define outer $L^p$ spaces associated with arbitrary $(n+1)$ structures on a set, namely a measure and $n$ outer measures. We refer to these spaces as \emph{iterated outer $L^p$ spaces}, and we provide a definition in detail. We start recalling the classical product of $L^p$ spaces on a set with a Cartesian product structure. Given a collection of couples of finite sets with strictly positive weights $(X_i,\omega_i)$, we define recursively the product $L^p$ quasi-norms for functions on their Cartesian product as follows. For any $n \in \N$, let
\begin{equation*}
Y^n = \prod_{i=1}^n X_i,
\end{equation*}
where, for $n=0$, the empty Cartesian product is intended to be $\{ \varnothing \}$. Note that, for any $x \in X_n$, a function $f$ on $Y^n$ defines a function $f(\cdot,x)$ on $Y^{n-1}$. Given a collection of exponents $p_i \in (0,\infty]$, we define the classical product $ \mathbb{L}_n $ quasi-norm of a function $f$ on $Y^n$, where
\begin{equation*}
\mathbb{L}_n = L^{p_n}_{\omega_n}(L^{p_{n-1}}_{\omega_{n-1}}(\dots L^{p_1}_{\omega_1})),
\end{equation*}
by the recursion
\begin{align} \label{eq:classical_recursion_base}
& \norm{f(x)}_{\mathbb{L}_0} = \abs{f(x)}, \\
\label{eq:classical_recursion}
& \norm{f}_{\mathbb{L}_n} = \norm{ \norm{f (\cdot,x) }_{\mathbb{L}_{n-1}} }_{L^{p_n}(X_n,\omega_n)}.
\end{align}

The theory of outer $L^p$ spaces allows for a generalization of this definition to settings where the underlying set has no Cartesian product structure. For the purpose of this paper, we provide the definition of the iterated outer $L^p$ quasi-norms in the form of a recursion analogous to that in \eqref{eq:classical_recursion_base}, \eqref{eq:classical_recursion}. 

Let $X$ be a finite set together with a collection of outer measures $\mu_i$ on it. To avoid cumbersome details, we make the harmless assumption that every $\mu_i$ is finite and strictly positive on every nonempty element of $\mathcal{P}(X)$. In fact, it is reasonable that subsets of $X$ on which either of the outer measures is $0$ or $\infty$ should contribute only trivially to the iterated outer $L^p$ spaces on $X$, and we ignore them altogether. Throughout the paper, we avoid recalling this assumption, but the reader should always consider it implicitly stated whenever we refer to outer measures.

Given a collection of exponents $p_i \in (0,\infty]$, we define the iterated outer $\mathbf{L}_n$ quasi-norm of a function $f$ on $X$, where
\begin{equation*}
\mathbf{L}_n = L^{p_n}_{\mu_n}(\ell^{p_{n-1}}_{\mu_{n-1}}(\dots \ell^{p_1}_{\mu_1})),
\end{equation*}
by the recursion
\begin{align} \label{eq:outer_recursion_base}
& \norm{f}_{\mathbf{L}_0} = \sup_{x \in X} \abs{f(x)}, \\
\label{eq:outer_recursion_size}
& \mathbf{I}_n(f) = \sup_{\varnothing \neq A \subseteq X} \mu_{n}(A)^{-(p_{n-1})^{-1}} \norm{f 1_A}_{\mathbf{L}_{n-1}}, \\
\label{eq:outer_recursion}
& \norm{f}_{\mathbf{L}_n} = 
\begin{dcases}
\mathbf{I}_n(f), \qquad & \textrm{if $p_n=\infty$,} \\ 
\Big( \int_{0}^{\infty} p_n \lambda^{p_n} \inf \{ \mu_n(B) \colon \mathbf{I}_n(f 1_{B^c}) \leq \lambda \} \frac{\diff \lambda}{\lambda} \Big)^{\frac{1}{p_n}}, 
\qquad & \textrm{if $p_n \neq \infty$,}
\end{dcases}
\end{align}
where $p_0 = \infty$, and the exponent $\infty^{-1}$ is intended to be $0$. We refer to the space defined by the quantity in \eqref{eq:outer_recursion} as the \emph{iterated outer $L^p$ space} $\mathbf{L}_n$ or $L^{p_n}_{\mu_n}(\ell^{p_{n-1}}_{\mu_{n-1}}(\dots \ell^{p_1}_{\mu_1}))$, where we denote the argument of the supremum in \eqref{eq:outer_recursion_size} as
\begin{equation} \label{eq:size}
\ell^{p_{n-1}}_{\mu_{n-1}}(\dots \ell^{p_1}_{\mu_1}) (f)(A) =  \mu_{n}(A)^{-(p_{n-1})^{-1}} \norm{f 1_A}_{\mathbf{L}_{n-1}},
\end{equation} 
and the infimum in \eqref{eq:outer_recursion} as
\begin{equation} \label{eq:super_level_measure}
\mu_{n}( \ell^{p_{n-1}}_{\mu_{n-1}}(\dots \ell^{p_1}_{\mu_1}) (f) > \lambda ) = \inf \{ \mu_n(B) \colon \mathbf{I}_n(f 1_{B^c}) \leq \lambda \}.
\end{equation} 
In the language of the $L^p$ theory for outer measure spaces, the quantity in \eqref{eq:size} defines a \emph{size}, and that in \eqref{eq:super_level_measure} defines the \emph{super level measure} of a function $f$ at level $\lambda$ with respect to the size. 

If the outer measure $\mu_1$ is a measure $\omega$, then we have, for every $ p_1 \in (0,\infty]$,
\begin{equation*}
\norm{f}_{\mathbf{L}_{1}} = \norm{f}_{L^{p_1}(X, \omega)},
\end{equation*}
hence we can begin the recursion in \eqref{eq:outer_recursion_base}, \eqref{eq:outer_recursion_size}, \eqref{eq:outer_recursion} from $\mathbf{L}_{1}$. In fact, already the general case has this form. The quasi-norm defined by the collections of outer measures $\mu_i$ and exponents $p_i$ is the same one defined by the collections of outer measures $\widetilde{\mu}_{i}$ and exponents $\widetilde{p}_i$, where $\widetilde{\mu}_1$ is the counting measure, $\widetilde{p}_1 = \infty$, and $\widetilde{\mu}_{i+1} = \mu_{i}, \widetilde{p}_{i+1} = p_{i}$ for every $i \in \N$. Therefore, without loss of generality, we always assume that $\mu_1$ is a measure $\omega$ associated with a finite and strictly positive weight that we denote by $\omega$ as well, with a slight abuse of notation. As before, throughout the paper, we avoid recalling this assumption, but the reader should always consider it implicitly stated whenever we refer to measures.

The classical product $\mathbb{L}_n$ quasi-norms defined in \eqref{eq:classical_recursion} are a special case of the iterated outer $\mathbf{L}_n$ ones defined in \eqref{eq:outer_recursion}, with the same collection of exponents and the following collection of outer measures $\mu_i$. For any $1 \leq j \leq n$, we define
\begin{equation*}
Y^n_j= \prod_{i=j}^n X_i,
\end{equation*}
and we observe that the set $Y^n$ has a canonical partition $\mathcal{Z}_j$, namely
\begin{equation*}
\mathcal{Z}_j = \{  Y_1^{j-1} \times z \colon z \in Y^n_j \}.
\end{equation*}
where the set $Y_1^{0} \times z$ is intended to be the singleton $\{z\}$. For every $A \subseteq Y^n$, let
\begin{equation} \label{eq:cartesian_product_outer_measure}
\mu_i(A) = \inf_{Z} \{ \sum_{z \in Z} \prod_{j=i}^n \omega_j \big(\pi_j(z) \big) \},
\end{equation}
where $\pi_j$ is the projection in the coordinate in $X_j$, and the infimum is taken over all subsets $Z$ of $Y^n_{i}$ such that $A$ is covered by the elements of $\mathcal{Z}_i$ associated with $Z$.

The theory of classical product of $L^p$ spaces is well-developed, see for example \cite{MR126155}. In the range of exponents $p_i \in [1, \infty]$, the quantities defined in \eqref{eq:classical_recursion} are norms, and they satisfy the expected duality properties. On the other hand, the theory of outer $L^p$ spaces is a theory of quasi-norms, mainly developed in \cite{MR3312633} towards their real interpolation features like Radon-Nikodym results, H\"{o}lder's inequality and Marcinkiewicz interpolation, due to the aforementioned two-step programme.
 
However, as showed in \cite{MR3312633}, the iterated outer $L^p$ spaces satisfy some properties analogous to those of the iterated classical ones. In particular, a one-direction "collapsing effect" and a version of H\"older's inequality up to a uniform constant, namely
\begin{gather}
\label{eq:collapsing}
\norm{f}_{L^1(X, \omega)} \leq C \norm{f}_{L^1_{\mu_n} ( \ell^1_{\mu_{n-1}}( \dots \ell^1_\omega)) }, \\
\label{eq:outer_Holder}
\sup_{g} \{  \norm{fg}_{L^1_{\mu_n} ( \dots \ell^1_\omega) } \colon \norm{g}_{L^{p'_n}_{\mu_n} ( \dots \ell^{p'_1}_\omega)} \leq 1 \} \leq C \norm{f}_{L^{p_n}_{\mu_n} ( \dots \ell^{p_1}_\omega)},
\end{gather}
where, for every $1 \leq i \leq n$, the exponent $p_i'$ is the H\"older conjugate of $p_i$, satisfying 
\begin{equation*}
\frac{1}{p_i} + \frac{1}{p_i'} =1.
\end{equation*}

In \cite{2020arXiv200105903F}, we studied the opposite inequalities in \eqref{eq:collapsing} and in \eqref{eq:outer_Holder} in the single iterated case, namely when $n=2$. We proved the equivalence in both cases up to constants depending on $p_i \in (1,\infty)$ but uniform in the cardinality of $X$, as long as it is finite. These in turn imply the equivalence of the outer $L^{p_2}_\mu(\ell^{p_1}_{\omega})$ quasi-norms to the norms defined by the supremum in \eqref{eq:outer_Holder}. The endpoint cases $p_1 = \infty$ and $p_2 = 1$ exhibit a different behaviour, and we refer to \cite{2020arXiv200105903F} for more details.

In the present paper, we focus on the analogous opposite inequalities in \eqref{eq:collapsing} and in \eqref{eq:outer_Holder} in the double iterated case, namely when $n=3$. Already in this case, the study of the opposite inequalities becomes substantially more difficult due to the interplay between the subadditivity of the two outer measures and the exponents. We start recalling the setting. Let $X$ be a finite set, $\mu,\nu$ outer measures, and $\omega$ a measure. Given three exponents $ p,q,r \in (0,\infty]$, we define the double iterated outer $L^p$ space $L^p_\mu(\ell^q_\nu(\ell^r_\omega))$ through the quasi-norm in \eqref{eq:outer_recursion}, with $\mu_1=\omega$, $\mu_2 = \nu$, $\mu_3 = \mu$, and $p_1=r$, $p_2 = q$, $p_3 = p$.

Before stating our main results, we introduce some auxiliary definitions. They depend on parameters $\Phi,K \geq 1$ that we are going to avoid recalling every time.

Given a subset $A$ of $X$, we say that a subset $B$ of $X$ is a \emph{$\mu$-parent set of $A$ (with parameter $\Phi$)} if $A \subseteq B$ and we have
\begin{equation}
\label{eq:parent_optimality}
\mu(B) \leq \Phi \mu(A).
\end{equation}
A \emph{$\mu$-parent function $\mathbf{B}$ (with parameter $\Phi$)} is then a monotone function from $\mathcal{P}(X)$ to itself, associating every subset $A$ of $X$ with a $\mu$-parent set (with parameter $\Phi$) $\mathbf{B}(A)$. 

Moreover, given a collection $\mathcal{E}$ of subsets of $X$, we say that a function $\mathcal{C}$ from $\mathcal{P}(X)$ to the set of subcollections of pairwise disjoint elements in $\mathcal{E}$ is a \emph{$\mu$-covering function (with parameter $\Phi$)} if the function $\mathbf{B}_{\mathcal{C}}$ from $\mathcal{P}(X)$ to itself defined by
\begin{equation*}
\mathbf{B}_{\mathcal{C}} (A) = \bigcup_{E \in \mathcal{C}(A)} E,
\end{equation*}
is a $\mu$-parent function (with parameter $\Phi$). 

Next, we say that a collection $\mathcal{A}$ of pairwise disjoint subsets of $X$ is \emph{$\nu$-Carath\'{e}odory (with parameter $K$)} if, for every subset $U$ of $X$, we have
\begin{equation} \label{eq:K_Carath\'{e}odory}
\sum_{A \in \mathcal{A}} \nu ( U \cap A ) \leq K \nu \big( U \cap \bigcup_{A \in \mathcal{A}} A \big).
\end{equation}

Finally, we define two conditions for the quadruple $(X,\mu,\nu,\mathcal{C})$.

\begin{condition} [Canopy] \label{def:parent_collection_nesting}
	We say that $(X,\mu,\nu,\mathcal{C})$ satisfies the \emph{canopy condition (with parameters $\Phi,K$)} if $\mathcal{C}$ is a $\mu$-covering function (with parameter $\Phi$), and for every $\nu$-Carath\'{e}odory collection (with parameter $K$) $\mathcal{A}$, for every subset $D$ of $X$ disjoint from $\mathbf{B}_{\mathcal{C}} \big(\bigcup_{A \in \mathcal{A}} A\big)$, the collection $\mathcal{A} \cup \{ D \}$ is still $\nu$-Carath\'{e}odory (with parameter $K$).
\end{condition}

\begin{condition} [Crop] \label{def:anti_stacking_compatibility}
	We say that $(X,\mu,\nu,\mathcal{C})$ satisfies the \emph{crop condition (with parameters $\Phi,K$)} if $\mathcal{C}$ is a $\mu$-covering function (with parameter $\Phi$), and for every collection $\mathcal{A}$ in $\mathcal{E}$, there exists a $\nu$-Carath\'{e}odory subcollection (with parameter $K$) $\mathcal{D}$ of $\mathcal{A}$ such that, for every subset $F$ of $X$ disjoint from $\bigcup_{D \in \mathcal{D}} D$, we have
	\begin{equation*}
	\mathbf{B}_{\mathcal{C}}(F) = \mathbf{B}_{\widetilde{\mathcal{C}}}(F),
	\end{equation*}
	where
	\begin{equation*}
	\widetilde{\mathcal{C}}(F) = \mathcal{C}(F) \setminus \mathcal{A}.
	\end{equation*}
\end{condition}

We are now ready to state our main results. 
\begin{theorem} \label{thm:collapsing_2_step_iteration_finite}
	For every $q,r \in (0,\infty]$, $\Phi,K \geq 1$, there exist constants $C_1=C_1(q,r,\Phi,K) , C_2=C_2(q,r,\Phi,K)$ such that the following property holds true. 
	
	Let $X$ be a finite set, $\mu,\nu$ outer measures, $\omega$ a measure, and $\mathcal{C}$ a $\mu$-covering function such that $(X,\mu,\nu,\mathcal{C})$ satisfies the canopy condition \ref{def:parent_collection_nesting}. Then, for every function $f \in L^q_\mu(\ell^q_\nu(\ell^r_\omega))$ on $X$, we have
		\begin{equation} \label{eq:collapsing_2_iterated}
		C_1^{-1} \norm{f}_{L^q_\nu(\ell^r_\omega)} \leq \norm{f}_{L^q_\mu(\ell^q_\nu(\ell^r_\omega))} \leq C_2 \norm{f}_{L^q_\nu(\ell^r_\omega)}.
		\end{equation}
		
		If $q \leq r$ or $q=\infty$, the constant $C_1$ does not depend on $\Phi,K$. 
		
		If $q \geq r$, the constant $C_2$ does not depend on $\Phi,K$. 
\end{theorem}

\begin{theorem} \label{thm:Holder_triangular_2_step_iteration_finite}
	For every $p,q \in (1,\infty)$, $r \in [q, \infty)$, $\Phi,K \geq 1$, there exists a constant $C=C(p,q,r,\Phi,K)$ such that the following property holds true. 
	
	Let $X$ be a finite set, $\mu,\nu$ outer measures, $\omega$ a measure, and $\mathcal{C}$ a $\mu$-covering function such that $(X,\mu,\nu,\mathcal{C})$ satisfies the canopy condition \ref{def:parent_collection_nesting}. Then
	\begin{enumerate} [(i)]
		\item For every function $f \in L^p_\mu(\ell^q_\nu(\ell^r_\omega))$ on $X$, we have
		\begin{equation} \label{eq:Holder_2_iterated}
		C^{-1} \norm{f}_{L^{p}_\mu(\ell^{q}_\nu(\ell^{r}_\omega))} \leq \sup_{\norm{g}_{L^{p'}_\mu(\ell^{q'}_\nu(\ell^{r'}_\omega))} = 1} \norm{fg}_{L^1(X,\omega)} \leq C \norm{f}_{L^{p}_\mu(\ell^{q}_\nu(\ell^{r}_\omega))} .
		\end{equation}
		\item For every collection of functions $\{ f_n \colon n \in \N \} \subseteq {L^{p}_\mu(\ell^{q}_\nu(\ell^{r}_\omega))}$ on $X$, we have
		\begin{equation} \label{eq:triangle_2_iterated}
		\norm{\sum_{n \in \N} f_n}_{L^{p}_\mu(\ell^{q}_\nu(\ell^{r}_\omega))} \leq C \sum_{n \in \N} \norm{f_n}_{L^{p}_\mu(\ell^{q}_\nu(\ell^{r}_\omega))}.
		\end{equation}
	\end{enumerate}
	
	For every $p,q \in (1,\infty), r \in (1,q]$, $\Phi,K \geq 1$, there exists a constant $C=C(p,q,r,\Phi,K)$ such that the analogous property holds true for every finite set $X$, outer measures $\mu,\nu$, measure $\omega$, and $\mu$-covering function $\mathcal{C}$ such that $(X,\mu,\nu,\mathcal{C})$ satisfies the crop condition \ref{def:anti_stacking_compatibility}.
	
	If $q = r$, the constant $C$ does not depend on $\Phi,K$.
\end{theorem}
The first result describes one instance of the "collapsing effect". When we have two consecutive outer $L^p$ space structures associated with the same exponent, under certain conditions, the "exterior" one can be disregarded. We recall that, as a consequence of the "collapsing effect" in the single iterated case, property $(i)$ of Theorem 1.1 in \cite{2020arXiv200105903F}, for every $p,r \in (0,\infty]$, we have
\begin{equation*}
C^{-1} \norm{f}_{L^p_\mu(\ell^r_\omega)} \leq \norm{f}_{L^p_\mu(\ell^r_\nu(\ell^r_\omega))} \leq C \norm{f}_{L^p_\mu(\ell^r_\omega)},
\end{equation*}
where the constant $C=C(p,r)$ does not depend on $\Phi,K$, and it is uniform in $X,\mu,\nu,\omega$. Hence, the double iterated outer $L^p$ spaces are reduced to single iterated ones. In particular, when $p=q=r \in (0,\infty]$, we have the full "collapsing effect"
\begin{equation} \label{eq:collapsing_pairing}
C^{-1} \norm{f}_{L^r(X,\omega)} \leq \norm{f}_{L^r_\mu(\ell^r_\nu(\ell^r_\omega))} \leq C \norm{f}_{L^r(X,\omega)},
\end{equation}
with constant $C=C(r)$ uniform in $X,\mu,\nu,\omega$.

The second result yields the sharpness of outer H\"{o}lder's inequality. As a consequence, the iterated outer $L^p_\mu(\ell^q_\nu(\ell^r_\omega))$ quasi-norm inherits from the $L^1(X,\omega)$-pairing a quasi-triangle inequality up to a constant uniform in the number of the summands, which is stated in the second property. Therefore, in the prescribed range of exponents, the double iterated outer $L^{p}$ space is uniformly isomorphic to a Banach space with norm defined by the supremum in \eqref{eq:Holder_2_iterated}. Moreover, it is the K\"{o}the dual space of the outer $L^{p'}_\mu(\ell^{q'}_\nu(\ell^{r'}_\omega))$ space, and we refer to \cite{2020arXiv200105903F} for an explanation of the use of the term K\"{o}the duality in this context. 

The main focus of both of the theorems is on the dependence of the constants in \eqref{eq:collapsing_2_iterated}, \eqref{eq:Holder_2_iterated}, and \eqref{eq:triangle_2_iterated}. A class of counterexamples we exhibit in Subsection \ref{subsec:counterexamples} shows that these constants are not uniform in $\Phi,K$, at least in a certain range of exponents $p,q,r \in (0, \infty]$. Therefore, in these cases it is necessary to require some conditions on the setting $(X,\mu,\nu,\omega)$ or the exponents $p,q,r$ to recover, under these additional assumptions, the uniformity of the constant in the cardinality of $X$. In particular, this observation points out a substantial difference with the single iterated case, where pathological behaviours of the outer $L^p$ spaces appear only in the endpoint cases corresponding to the exponents $p_1 = \infty$ or $p_2 = 1$. As a matter of fact, it would be interesting  to identify necessary and sufficient conditions on the setting $(X,\mu,\nu,\omega)$ in order to obtain the results stated in Theorem \ref{thm:collapsing_2_step_iteration_finite} and Theorem \ref{thm:Holder_triangular_2_step_iteration_finite} uniformly in the cardinality of $X$. Finally, we mention the dichotomy between the cases $q > r$ and $q < r$ in the statement of the two theorems, in particular in view of the reduction to the single iterated outer $L^p$ spaces in the case $q=r$. While in Theorem \ref{thm:collapsing_2_step_iteration_finite} this phenomenon is in part explained by the class of counterexamples, it would be interesting to clarify whether in Theorem \ref{thm:Holder_triangular_2_step_iteration_finite} it is an intrinsic feature of the problem or it is just an artefact of the argument used in the proof. If the former case were true, it would be interesting to clarify how the dichotomy between the cases $q > r$ and $q < r$ was reflected in the necessary and sufficient conditions to recover the uniformity of the constant in the cardinality of $X$. 

Before moving on, we briefly comment on the definition of $\nu$-Carath\'{e}odory collections and the conditions we stated before the results. We start observing that the Carath\'{e}odory measurability test with respect to an outer measure $\mu^{*}$ corresponds to checking that the collection $\{ E, E^c \}$ is $\mu^{*}$-Carath\'{e}odory with parameter $1$. In particular, when $\nu$ is a measure, every collection of pairwise disjoint measurable subsets of $X$ is $\nu$-Carath\'{e}odory with parameter $K=1$. This fact implies that, in the single iterated case, we can always deal with $\nu$-Carath\'{e}odory collections, which come with desirable properties. In particular, for every set $X$, outer measure $\mu$, measure $\omega$, the quadruple $(X,\mu,\omega,\Id)$ satisfies both the canopy condition \ref{def:parent_collection_nesting} and the crop condition \ref{def:anti_stacking_compatibility} with parameters $\Phi=K=1$.

The extension of the results stated in Theorem \ref{thm:collapsing_2_step_iteration_finite} and Theorem \ref{thm:Holder_triangular_2_step_iteration_finite} to infinite settings under reasonable assumptions should not be a surprise. However, this level of generality is beyond the scope of the paper. We concern ourselves only with two specific infinite settings, namely the one described by Uraltsev in \cite{2016arXiv161007657U} and a slight variation of it, both of them defined on the upper half $3$-space. Although not equivalent, these settings exhibit similar geometric properties. We focus mainly on the latter, which allows for a better exploitation of them. 

We briefly recall the setting that we describe in detail in Subsection \ref{subsec:3spacedyadic}. Let $X$ be the upper half $3$-space $ \R \times (0,\infty) \times \R$, and $\omega$ the measure induced on it by the Lebesgue measure $\diff y \diff t \diff \eta$ on $\R^3$. On $X$, we define two outer measures by means of the following covering construction. Given a collection $\mathcal{S}$ of subsets of $X$ and a pre-measure $\sigma \colon \mathcal{S} \to (0, \infty)$, we define the outer measure $\mu \colon \mathcal{P}(X) \to [0, \infty]$ on an arbitrary subset $A$ of $X$ by 
\begin{equation} \label{eq:outer_measure_from_pre_measure}
\mu(A) = \inf\{ \sum_{S \in \mathcal{S}'} \sigma(S) \colon \mathcal{S}' \subseteq \mathcal{S}, A \subseteq \bigcup_{S \in \mathcal{S}'} S \}.
\end{equation}

First, for any dyadic interval $I \subseteq \R$, let $D(I)$ be the dyadic strip given by the Cartesian product between $I$, the interval $(0,\abs{I})$, and $\R$. Let $\mathcal{D}$ be the collection of all the dyadic strips, and, for every $D(I) \in \mathcal{D}$, let $\sigma$ be the length of the base $I$.

Second, for any couple of dyadic intervals $I,\widetilde{I} \subseteq \R$ with inverse lengths, let $T(I,\widetilde{I})$ be the dyadic tree given by the union of the Cartesian products between a dyadic interval $J \subseteq I$, the interval $(0,\abs{J})$, and the dyadic interval $\widetilde{J} \supseteq \widetilde{I}$ with inverse length of $J$. Let $\mathcal{T}$ be the collection of all the dyadic trees, and, for every $T(I,\widetilde{I}) \in \mathcal{T}$, let $\tau$ be the length of the base $I$. 

Now, let $\mu,\nu$ be the outer measures on $X$ associated with $(\mathcal{D},\sigma),(\mathcal{T},\tau)$ respectively as in \eqref{eq:outer_measure_from_pre_measure}. As we will see in Appendix \ref{sec:dyadic_geometry}, for every dyadic strip $D$ in $\mathcal{D}$ and every dyadic tree $T$ in $\mathcal{T}$, we have
\begin{equation*}
\mu(D) = \sigma(D), \qquad \qquad \nu(T) = \tau(T).
\end{equation*}
We define the double iterated outer $L^p$ space $L^p_\mu(\ell^q_\nu(\ell^r_\omega))$ in the upper half $3$-space setting through the quasi-norm in \eqref{eq:outer_recursion} for $\omega$-measurable functions. We use $\mu_1=\omega$, $\mu_2 = \nu$, $\mu_3 = \mu$, and we restrict the supremum in $\mathbf{I}_1$ to the $\omega$-measurable sets, that in $\mathbf{I}_2$ to the dyadic trees in $\mathcal{T}$, and that in $\mathbf{I}_3$ to the dyadic strips in $\mathcal{D}$.

In this setting, we have both the "collapsing effect" and the sharpness of outer H\"{o}lder's inequality described in the finite setting in the previous theorems.  
\begin{theorem} \label{thm:collapsing_Holder_triangular_upper_3_space}
Let $(X,\mu,\nu,\omega)$ be the dyadic upper half $3$-space setting just described, $p,q,r \in (0,\infty]$. There exists a constant $C=C(p,q,r)$ such that the following properties hold true.
\begin{enumerate} [(i)]
	\item For every $q,r \in (0,\infty)$, for every function $f \in L^q_\mu(\ell^q_\nu(\ell^r_\omega))$ on $X$, we have
	\begin{equation} \label{eq:collapsing_2_iterated_upper_3_space}
	C^{-1} \norm{f}_{L^q_\nu(\ell^r_\omega)} \leq \norm{f}_{L^q_\mu(\ell^q_\nu(\ell^r_\omega))} \leq C \norm{f}_{L^q_\nu(\ell^r_\omega)}.
	\end{equation}
	\item For every $p,q,r \in (1,\infty)$, for every function $f \in L^p_\mu(\ell^q_\nu(\ell^r_\omega))$ on $X$, we have
	\begin{equation} \label{eq:Holder_2_iterated_upper_3_space}
	C^{-1} \norm{f}_{L^{p}_\mu(\ell^{q}_\nu(\ell^{r}_\omega))} \leq \sup_{\norm{g}_{L^{p'}_\mu(\ell^{q'}_\nu(\ell^{r'}_\omega))} = 1} \norm{fg}_{L^1(X,\omega)} \leq C \norm{f}_{L^{p}_\mu(\ell^{q}_\nu(\ell^{r}_\omega))} .
	\end{equation}
	\item For every $p,q,r \in (1,\infty)$, for every collection of functions $\{ f_n \colon n \in \N \} \subseteq {L^{p}_\mu(\ell^{q}_\nu(\ell^{r}_\omega))}$ on $X$, we have
	\begin{equation} \label{eq:triangle_2_iterated_upper_3_space}
	\norm{\sum_{n \in \N} f_n}_{L^{p}_\mu(\ell^{q}_\nu(\ell^{r}_\omega))} \leq C \sum_{n \in \N} \norm{f_n}_{L^{p}_\mu(\ell^{q}_\nu(\ell^{r}_\omega))}.
	\end{equation}
\end{enumerate}
\end{theorem}
The result analogous to Theorem \ref{thm:collapsing_Holder_triangular_upper_3_space} holds true even in the upper half $3$-space setting with arbitrary strips and trees originally considered in \cite{2016arXiv161007657U} that we describe in detail in Subsection \ref{subsec:3spacearbitrary}. 

We conclude pointing out that the outer $L^p$ spaces used by Uraltsev are different from those defined in \eqref{eq:outer_recursion}. In \cite{2016arXiv161007657U}, the innermost size, namely the quantity in \eqref{eq:size} for $n=2$, is not defined by a single Lebesgue norm with respect to the measure $\omega$, but by a sum of an $L^2$ and an $L^\infty$ norms, making it more complicated to treat. The first step in the study of these spaces would be to extend the results stated in Theorem \ref{thm:collapsing_Holder_triangular_upper_3_space} to the case $r=\infty$. This is likely to be achieved exploiting the geometric properties of the strips and trees in the upper half $3$-space in the same fashion of the boxes in the upper half space in \cite{2020arXiv200105903F}. The second step, the one requiring new considerations, would be to address the issue of the variable exponent Lebesgue norm.

\subsection*{Guide to the paper}

In Section \ref{sec:preliminaries}, we review some preliminaries about outer $L^p$ quasi-norms and, more specifically, single iterated outer $L^p$ ones from \cite{2020arXiv200105903F}. 
In Section \ref{sec:equivalence}, we prove Theorem \ref{thm:collapsing_2_step_iteration_finite} and Theorem \ref{thm:Holder_triangular_2_step_iteration_finite}. Moreover, we exhibit a class of counterexamples to the unconditional uniformity in the cardinality of $X$ of the constants appearing in the statements of these theorems at least in a certain range of exponents $p,q,r \in (0,\infty]$. 
In Section \ref{sec:examples}, we describe some settings in which we define a $\mu$-covering function satisfying the canopy condition \ref{def:parent_collection_nesting} and the crop condition \ref{def:anti_stacking_compatibility}.
In Section \ref{sec:dyadic_upper_space}, we prove Theorem \ref{thm:collapsing_Holder_triangular_upper_3_space} in the dyadic upper half $3$-space setting reducing the problem to an equivalent one in a finite setting via an approximation argument. 
The proof relies on the geometric properties of the outer measures and the approximation properties of functions in iterated outer $L^p$ spaces that we will prove in Appendix \ref{sec:dyadic_geometry} and Appendix \ref{sec:approximation}, respectively.

\section*{Acknowledgements}
The author gratefully acknowledges financial support by the CRC 1060 \emph{The Mathematics of Emergent Effects} at the University of Bonn, funded through the Deutsche Forschungsgemeinschaft. 

Part of the work was initiated during the visit to the MFO in Oberwolfach in July 2020 for the workshop \emph{Real Analysis, Harmonic Analysis and Applications}. 

The author is thankful to Christoph Thiele for helpful discussions, comments and suggestions that greatly improved the exposition of the material, and for his continuous support.

\section{Preliminaries} \label{sec:preliminaries}

In this section, we make some observations about the outer $L^p$ quasi-norms. Moreover, we review the decomposition result for functions in a single iterated outer $L^p$ space, which is the main ingredient in proving the results corresponding to Theorem \ref{thm:collapsing_2_step_iteration_finite} and Theorem \ref{thm:Holder_triangular_2_step_iteration_finite} in \cite{2020arXiv200105903F}. It provides a model for the decomposition in the case of double iterated outer $L^p$ spaces that we perform in Section \ref{sec:equivalence}.  

First, for every $p \in (0, \infty)$, we observe that we can replace the integral defining the outer $L^p$ quasi-norm in \eqref{eq:outer_recursion} by a discrete version of it. For every $\Psi > 1$, we have
\begin{equation} \label{eq:discretized_norm}
\norm{f}_{L^p_\mu(S)}^p \sim_{\Psi,p} \sum_{k \in \Z} \Psi^{kp} \mu(S(f) > \Psi^k) \sim_{\Psi,p} \sum_{k \in \Z} \Psi^{kp} \sum_{l \geq k} \mu(S(f) > \Psi^l),
\end{equation}
where $S$ is a size of the form $\ell^r_\omega$ or $ \ell^q_\nu(\ell^r_\omega)$, and more generally an arbitrary size in the definition in \cite{MR3312633}. The equivalences in \eqref{eq:discretized_norm} follow by the monotonicity of the super level measure, Fubini and the bounds on the geometric series. 

Next, let $X$ be a finite set, $\mu, \nu$ outer measures, and $\omega$ a measure. Since $\mu,\nu$ are finite and strictly positive on every nonempty subset of $X$, by outer H\"{o}lder's inequality, Proposition 3.4 in \cite{MR3312633}, we have
\begin{equation}
\begin{gathered} \label{eq:Holder_containment}
L^q_\nu(\ell^r_\omega) \subseteq L^\infty_\nu(\ell^r_\omega), \\
L^p_\mu(\ell^q_\nu(\ell^r_\omega)) \subseteq L^\infty_\mu(\ell^q_\nu(\ell^r_\omega)) \cap L^\infty_\mu(\ell^\infty_\nu(\ell^r_\omega)).
\end{gathered}
\end{equation}

Finally, we recall two results for single iterated outer $L^p$ spaces already appearing, explicitly or implicitly stated, in Proposition 2.1 in \cite{2020arXiv200105903F}, with their proofs.

\begin{lemma} \label{thm:atomic_super_level_set_interior_level}
	For every $r \in (0, \infty)$, $N \geq 1$, there exist constants $C=C(r,N)$, $c=c(N)$ such that the following property holds true.
	
	Let $X$ be a set, $\nu$ an outer measure, and $\omega$ a measure.	Let $f \in L^\infty_\nu(\ell^r_\omega)$ be a function on $X$, let $k \in \Z$ satisfy
	\begin{equation} \label{eq:bound_L_infty_interior_level}
	\norm{f}_{L^\infty_\nu(\ell_\omega^r)} \in (2^k, 2^{k+1}],
	\end{equation}
	and let $A$ be a subset of $X$ such that
	\begin{equation} \label{eq:lower_bound_condition_interior_level}
	\norm{f 1_{A}}_{L^r(X,\omega)}^r > 2^{(k-N)r} \nu(A).
	\end{equation}
	Then we have
	\begin{equation} \label{eq:atomic_super_level_set_interior_level}
	\nu(A) \leq C \nu( \ell^r_\omega(f) > c 2^k).
	\end{equation}
\end{lemma}

\begin{proof}
	Let $\varepsilon > 0$. Let $V(c2^{k},\varepsilon)$ be an optimal set associated with the super level measure $\nu(\ell^r_\omega(f)> c2^{k})$ up to the multiplicative constant $(1+\varepsilon)$, namely
	\begin{gather} \label{eq:condition_super_level_interior_level_size}
	\norm{ f 1_{V(c2^{k}, \varepsilon)^c}}_{L^\infty_\nu(\ell^r_\omega)} \leq c2^{k}, \\ \label{eq:condition_super_level_interior_level_measure}
	(1+\varepsilon) \nu(\ell^r_\omega(f)> c2^{k}) \geq \nu\big(V(c2^{k}, \varepsilon) \big),
	\end{gather}  
	where $c$ will be fixed later. We have
	\begin{align*}
	\nu\big(V(c2^{k},\varepsilon) \big) & \geq  2^{-(k+1)r}  \norm{ f 1_{V(c 2^{k},\varepsilon)} 1_{A}}^r_{L^r(X,\omega)} \\
	& \geq  2^{-(k+1)r} \big( \norm{f 1_{A}}^r_{L^r(X,\omega)} - \norm{f 1_{A \setminus V(c 2^{k},\varepsilon)}}^r_{L^r(X,\omega)} \big) \\
	& \geq  2^{-(k+1)r} \big( 2^{(k-N)r} - c^r 2^{kr} \big) \nu(A),
	\end{align*}
	where we used the monotonicity of $\nu$ and \eqref{eq:bound_L_infty_interior_level} in the first inequality, the $r$-orthogonality of the classical $L^r$ quasi-norms of functions supported on disjoint sets in the second, \eqref{eq:lower_bound_condition_interior_level} and \eqref{eq:condition_super_level_interior_level_size} in the third. By choosing 
	\begin{equation*}
	c= 2^{-N-1},
	\end{equation*}
	and taking $\varepsilon$ arbitrarily small, the previous chain of inequalities together with \eqref{eq:condition_super_level_interior_level_measure} yields the desired inequality in \eqref{eq:atomic_super_level_set_interior_level}.
\end{proof}

\begin{proposition} \label{thm:atomic_decomposition_interior_level}
	For every $q,r \in (0, \infty)$, there exist constants $C=C(q,r)$, $c=c(q,r)$ such that the following decomposition properties hold true.
	
	Let $X$ be a finite set, $\nu$ an outer measure, $\omega$ a measure. For every function $f \in L^q_\nu(\ell^r_\omega)$ on $X$, there exists a collection $\{ U_{j} \colon j \in \Z \}$ of pairwise disjoint subsets of $X$ such that, if we set
	\begin{equation*}
	V_j = \bigcup_{l \geq j} U_l,
	\end{equation*}
	then, for every $j \in \Z$, we have
	\begin{gather} 
	\label{eq:superlevel_interior_level}
	\ell^r_\omega (f 1_{V_{j+1}^c})(U_{j}) > 2^j, \qquad \qquad \text{when $U_{j} \neq \varnothing$,} \\
	\label{eq:maximal_choice_interior_level}
	\norm{f 1_{V_j^c}}_{L^\infty_\nu(\ell^r_\omega)} \leq 2^{j}, \\
	\label{eq:covering_interior_level}
	\nu(\ell^r_\omega (f) > 2^{j}) \leq \nu(V_j) , \\
	\label{eq:optimal_covering_interior_level}
	\nu (U_{j}) \leq C \nu(\ell^r_\omega (f) > c 2^{j}).
	\end{gather}
	In particular, we have
	\begin{equation} \label{eq:outer_norm_as_size_and_measure_interior_level}
	\norm{f}^q_{L^q_\nu(\ell^r_\omega)}  \sim_{r,q}  \sum_{j \in \Z} 2^{jq} \nu(U_{j}) \sim_{r,q} \sum_{j \in \Z} 2^{jq} \sum_{l \geq j} \nu(U_{l}).
	\end{equation}
\end{proposition}

\begin{proof}
	The first four statements and their proof appeared already in Proposition 2.1 in \cite{2020arXiv200105903F}. The equivalences in \eqref{eq:outer_norm_as_size_and_measure_interior_level} follow by \eqref{eq:discretized_norm} \eqref{eq:covering_interior_level}, the definition of $V_j$, \eqref{eq:optimal_covering_interior_level}, Fubini, and the bounds for the geometric series.
\end{proof}

Throughout the paper, we use the observations made in this section without necessarily further referring to them. For example, the reader should always have in mind the equivalences in \eqref{eq:discretized_norm} whenever we consider an outer $L^p$ quasi-norm, and the list of properties \eqref{eq:superlevel_interior_level}--\eqref{eq:outer_norm_as_size_and_measure_interior_level} whenever we perform such a decomposition.

\section{Equivalence with norms} \label{sec:equivalence}

In this section, we study the equivalence of double iterated outer $L^p$ quasi-norms with norms uniformly in the finite setting.

First, for every $q,r \in (0,\infty)$, we study the $q$-orthogonality behaviour of the outer $L^q_\nu(\ell^r_\omega)$ quasi-norms of functions supported on disjoint sets. Accordingly, we show decomposition results for functions in the double iterated outer $L^p$ space with respect to a size of the form $\ell^q_\nu(\ell^r_\omega)$. We use them to prove Theorem \ref{thm:collapsing_2_step_iteration_finite}.

After that, for every $p,q,r \in (1,\infty)$, we produce a function $g$ for which we have a good lower bound on the $L^1(X,\omega)$-pairing with $f$ and a good upper bound on the $L^{p'}_\mu(\ell^{q'}_\nu(\ell^{r'}_\omega))$ quasi-norm of $g$. We use it to prove Theorem \ref{thm:Holder_triangular_2_step_iteration_finite}. 

Finally, we conclude the section with the promised class of counterexamples.

\subsection{$q$-orthogonality of the $L^q_\nu(\ell^r_\omega)$ quasi-norm}
We start with a result about the sub- and $q$-superorthogonality of the $L^q_\nu(\ell^r_\omega)$ quasi-norms of functions supported on arbitrary disjoint sets according to the case distinction $q \geq r$ or $q \leq r$. We exhibit counterexamples to the validity of the inequality in the opposite directions in both cases $q > r$ or $q < r$ in Subsection \ref{subsec:counterexamples}.

\begin{lemma} \label{thm:L_q_subsuperorthogonality}
	For every $q \in (0,\infty)$, $r \in (0, \infty]$, there exists a constant $C = C(q,r)$ such that the following properties hold true. 
	
	Let $X$ be a finite set, $\nu$ an outer measure, $\omega$ a measure. Let $\mathcal{A}$ be a collection of pairwise disjoint subsets of $X$. Then, for every function $f$ on $X$, we have
	\begin{gather} \label{eq:L_q_suborthogonality}
	\sum_{A \in \mathcal{A}} \norm{f 1_{A}}^q_{L^q_\nu(\ell^r_\omega)} \leq C \norm{f 1_{B}}^q_{L^q_\nu(\ell^r_\omega)}, \qquad \qquad \textrm{for $q \geq r$,} \\
	\label{eq:L_q_superorthogonality}
	\norm{f 1_{B}}^q_{L^q_\nu(\ell^r_\omega)} \leq C \sum_{A \in \mathcal{A}} \norm{f 1_{A}}^q_{L^q_\nu(\ell^r_\omega)}, \qquad \qquad \textrm{for $q \leq r$,}
	\end{gather}
	where $	B = \bigcup_{A \in \mathcal{A}} A$.
\end{lemma}

\begin{proof}
	Without loss of generality, we assume $q=1$. In fact, for $ \frac{r}{q} \in (0, \infty]$, we have
	\begin{equation*}
	\norm{f}^q_{L^q_\nu(\ell^{r/q}_\omega)} = \norm{f^q}_{L^1_\nu(\ell^{r/q}_\omega)}.
	\end{equation*}	
	
	{\textbf{Case I: $ q= 1, r= \infty$.}} We have
	\begin{equation} \label{eq:L_infty_size}
	\nu(\ell^\infty_\omega(f) > \lambda) = \nu (\{ x \in X \colon f(x) > \lambda \}). 
	\end{equation}
	Together with the subadditivity of $\nu$, this yields
	\begin{equation*}
	\nu(\ell^\infty_\omega(f 1_B) > \lambda) \leq \sum_{A \in \mathcal{A}} \nu (\ell^\infty_\omega(f 1_A) > \lambda). 
	\end{equation*}
	By integrating in $(0, \infty)$ on both sides, we obtain the desired inequality in \eqref{eq:L_q_superorthogonality}.
	
	{\textbf{Case II: $ q= 1$, $ r \in (0,1]$.}} We start with the following observation. Let $\mathcal{E}$ be a collection of pairwise disjoint sets such that, for every $E \in \mathcal{E}$, we have
	\begin{equation} \label{eq:condition}
	\ell^r_\omega(f)(E) \in (2^j,2^{j+1}].
	\end{equation}
	Together with the $r$-orthogonality of the classical $L^r$ quasi-norms of functions supported on disjoint sets and the subadditivity of $\nu$, this yields
	\begin{equation} \label{eq:size_on_union} 
	\ell^r_\omega(f) \big(\bigcup_{E \in \mathcal{E}} E \big) \geq \big(\nu \big(\bigcup_{E \in \mathcal{E}} E \big)^{-1} \sum_{E \in \mathcal{E}} 2^{jr} \nu(E) \big)^{\frac{1}{r}} > 2^j.
	\end{equation}
	
	Next, by the subadditivity of $\nu$ and by $r \leq 1$, we have
	\begin{equation*}
	\sum_{E \in \mathcal{E}} \nu(E) \leq \big( \nu \big(\bigcup_{E \in \mathcal{E}} E \big)^{-1} \sum_{E \in \mathcal{E}} \nu(E) \big)^{\frac{1}{r}} \nu\big(\bigcup_{E \in \mathcal{E}} E \big).
	\end{equation*}
	Together with \eqref{eq:condition}, this yields
	\begin{equation} \label{eq:sublinearity_size_outer_measure}
	\begin{split}
	\sum_{E \in \mathcal{E}} \ell^r_\omega(f)(E) \nu(E) & \leq 2^{j+1} \sum_{E \in \mathcal{E}} \nu(E) \\
	& \leq C \big( \nu \big( \bigcup_{E \in \mathcal{E}} E \big)^{-1} \sum_{E \in \mathcal{E}} 2^{jr} \nu(E) \big)^{\frac{1}{r}} \nu \big(\bigcup_{E \in \mathcal{E}} E \big) \\
	& \leq C \ell^r_\omega(f) \big(\bigcup_{E \in \mathcal{E}} E \big) \nu \big(\bigcup_{E \in \mathcal{E}} E \big).
	\end{split}
	\end{equation}
	
	Now, let $\{ A_j \colon j \in \Z \}$, $\{ B_j \colon j \in \Z \}$ be the collections associated with the decomposition in Proposition \ref{thm:atomic_decomposition_interior_level} of the functions $f 1_A$, $f 1_B$, respectively. By \eqref{eq:size_on_union} and \eqref{eq:sublinearity_size_outer_measure}, we can pass from the collection $\{ A_j \colon A \in \mathcal{A}, j \in \Z \}$ of pairwise disjoint subsets of $X$ to a collection $\mathcal{E} = \{ E_l \colon l \in \Z \}$ with strictly fewer elements such that
	\begin{gather} \label{eq:control_E_l_size}
	\ell^r_\omega(f)(E_l) \in (2^{l},2^{l+1}], \\ \label{eq:control_E_l_measure}
	\sum_{A \in \mathcal{A}} \sum_{j \in \Z} 2^{j} \nu(A_j) \leq C \sum_{l \in \Z} 2^l \nu(E_l).
	\end{gather}
	
	By the monotonicity of $\nu$, we have
	\begin{equation*}
	\norm{f 1_{E_l\cap \big(\bigcup_{k \geq l - 1} B_k \big)^c}}_{L^r(X,\omega)}^r \leq 2^{(l-1)r} \nu \big(E_l\cap \big(\bigcup_{k \geq l-1} B_k \big)^c \big) \leq 2^{(l-1)r} \nu(E_l).
	\end{equation*}
	Together with \eqref{eq:control_E_l_size}, this yields
	\begin{equation} \label{eq:condition_1}
	\begin{split}
	\sum_{k \geq l-1} \norm{f 1_{E_l\cap B_k}}_{L^r(X,\omega)}^r 
	& = \norm{f 1_{E_l\cap \bigcup_{k \geq l - 1} B_k}}_{L^r(X,\omega)}^r \\
	& = \norm{f 1_{E_l}}_{L^r(X,\omega)}^r - \norm{f 1_{E_l\cap \big(\bigcup_{k \geq l - 1} B_k\big)^c}}_{L^r(X,\omega)}^r \\
	& \geq c 2^{lr} \nu(E_l).
	\end{split}
	\end{equation}
	Therefore, we have
	\begin{align*}
	\sum_{l \in \Z} 2^{l} \nu(E_l)	& \leq C \sum_{l \in \Z} 2^{l(1-r)} \sum_{k \geq l-1} \norm{f 1_{E_l\cap B_k}}_{L^r(X,\omega)}^r \\
	& \leq C \sum_{k \in \Z} 2^{k(1-r)} \sum_{l \leq k+1} \norm{f 1_{E_l\cap B_k}}_{L^r(X,\omega)}^r \\
	& \leq C \sum_{k \in \Z} 2^{k(1-r)} \norm{f 1_{B_k}}_{L^r(X,\omega)}^r \\
	& \leq C \sum_{k \in \Z} 2^{k} \nu(B_k), 
	\end{align*}
	where we used \eqref{eq:condition_1} in the first inequality, $r \leq 1 $ in the second,  and the $r$-orthogonality of the classical $L^r$ quasi-norms of functions supported on disjoint sets in the third. Together with \eqref{eq:outer_norm_as_size_and_measure_interior_level} for the collections $\{ A_j \colon j \in \Z \}$, $\{ B_j \colon j \in \Z \}$, and \eqref{eq:control_E_l_measure}, the previous chain of inequalities yields the desired inequality in \eqref{eq:L_q_suborthogonality}.
	
	{\textbf{Case III: $ q = 1$, $ r \in [1,\infty)$.}} Let $A_j$, $B_j $ be defined as before. We have
	\begin{equation*} 
	\begin{split}
	\sum_{j \in \Z} 2^{j} & \nu(B_j) \leq \sum_{j \in \Z} 2^{j(1-r)} \norm{f 1_{B_j}}_{L^r(X,\omega)}^r \\
	& \leq \sum_{A \in \mathcal{A}} \sum_{k \in \Z} \sum_{j \in \Z} 2^{j(1-r)} \norm{f 1_{A_k \cap B_j}}_{L^r(X,\omega)}^r \\
	& \leq \sum_{A \in \mathcal{A}} \sum_{k \in \Z} \big( 2^{k(1-r)} \sum_{j \geq k}  \norm{ f 1_{A_k \cap B_j}}_{L^r(X,\omega)}^r + \sum_{j < k} 2^{j(1-r)} \norm{f 1_{A_k \cap B_j}}_{L^r(X,\omega)}^r \big) \\
	& \leq C \sum_{A \in \mathcal{A}} \sum_{k \in \Z} \big( 2^{k(1-r)} \norm{f 1_{A_k}}_{L^r(X,\omega)}^r + \sum_{j < k} 2^{j(1-r)} 2^{jr} \nu( A_k \cap B_j) \big) \\
	& \leq C \sum_{A \in \mathcal{A}} \sum_{k \in \Z} \big( 2^{k} \nu(A_k) + \sum_{j < k} 2^{j} \nu( A_k ) \big),
	\end{split}
	\end{equation*}
	where we used the $r$-orthogonality of the classical $L^r$ quasi-norms for functions with disjoint supports in the second and in the fourth inequality, and $r \geq 1 $ in the third. Together with \eqref{eq:outer_norm_as_size_and_measure_interior_level} for the collections $\{ A_j \colon j \in \Z \}$, $\{ B_j \colon j \in \Z \}$, the previous chain of inequalities yields the desired inequality in \eqref{eq:L_q_superorthogonality}.
\end{proof}

We continue with a result about the full $q$-orthogonality of the $L^q_\nu(\ell^r_\omega)$ quasi-norms of functions supported on disjoint sets forming a $\nu$-Carath\'{e}odory collection.

\begin{lemma} \label{thm:L_q_orthogonality}
	For every $q \in (0, \infty)$, $r \in (0, \infty]$, $K \geq 1$, there exist constants $C_1=C_1(q,r,K)$, $C_2=C_2(q,r,K)$ such that the following property holds true. 
	
	Let $X$ be a set, $\nu$ an outer measure, $\omega$ a measure. Let $\mathcal{A}$ be a $\nu$-Carath\'{e}odory collection of pairwise disjoint subsets of $X$. Then, for every function $f$ on $X$, we have
	\begin{equation} \label{eq:L_q_orthogonality}
	C_1^{-1} \norm{f 1_{B}}^q_{L^q_\nu(\ell^r)} \leq \sum_{A \in \mathcal{A}} \norm{f 1_{A}}^q_{L^q_\nu(\ell^r)} \leq C_2 \norm{f 1_{B}}^q_{L^q_\nu(\ell^r)},
	\end{equation}
	where $	B = \bigcup_{A \in \mathcal{A}} A$.
	
\end{lemma}
\begin{proof}
	As before, without loss of generality, we assume $q=1$.
	
	Expanding the definition of the outer $L^1_\nu(\ell^r_\omega)$ quasi-norms in \eqref{eq:L_q_orthogonality}, we have
	\begin{align*}
	\norm{f 1_{B}}_{L^1_\nu(\ell^r_\omega)} & = \int_{0}^\infty \nu(\ell^r_\omega(f 1_{B}) > \lambda ) \diff \lambda, \\
	\sum_{A \in \mathcal{A}} \norm{f 1_{ A}}_{L^1_\nu(\ell^r_\omega)} & = \int_{0}^\infty \sum_{A \in \mathcal{A}} \nu(\ell^r_\omega(f 1_{A}) > \lambda ) \diff \lambda.
	\end{align*}
	To show the desired inequalities, it is enough to prove that there exist constants $c=c(r,K)$, $C=C(r,K)$ such that, for every $\lambda > 0$, we have
	\begin{equation} \label{eq:L_q_orthogonality_level_set}
	\nu(\ell^r_\omega(f 1_{B}) > c \lambda ) \leq \sum_{A \in \mathcal{A}} \nu(\ell^r_\omega(f 1_{A}) > \lambda ) \leq C \nu(\ell^r_\omega(f 1_{B}) > \lambda ).
	\end{equation}
	By integrating in $(0, \infty)$ on both sides, we obtain the desired inequalities in \eqref{eq:L_q_orthogonality}.
	
	{\textbf{Case I: $ r = \infty$.}} By the subadditivity of $\nu$ and the $\nu$-Carath\'{e}odory condition \eqref{eq:K_Carath\'{e}odory}, we have
	\begin{equation*}
	\nu (\{ x \in B \colon f(x) > \lambda \}) \leq \sum_{A \in \mathcal{A}} \nu (\{ x \in A \colon f(x) > \lambda \}) \leq K \nu (\{ x \in B \colon f(x) > \lambda \}).
	\end{equation*}
	Together with \eqref{eq:L_infty_size}, this yields the desired inequalities in \eqref{eq:L_q_orthogonality_level_set}.
	
	{\textbf{Case II: $ r \in (0,\infty)$.}}	We start with the first inequality in \eqref{eq:L_q_orthogonality_level_set}.	Let $\varepsilon > 0$. For every $A \in \mathcal{A}$, let $V(A,\lambda,\varepsilon)$ be an optimal set associated with the super level measure $\nu(\ell^r_\omega(f 1_{A}) > \lambda )$ up to the multiplicative constant $(1+\varepsilon)$, namely
	\begin{gather} \label{eq:condition_super_level_1_size}
	\norm{ f 1_A 1_{V(A,\lambda,\varepsilon)^c}}_{L^\infty_\nu(\ell^r_\omega)} \leq \lambda, \\ \label{eq:condition_super_level_1_measure}
	(1 + \varepsilon) \nu(\ell^r_\omega(f 1_{A}) > \lambda ) \geq \nu \big( V(A,\lambda,\varepsilon)\big),
	\end{gather}
	and set 
	\begin{equation*}
	V = \bigcup_{A \in \mathcal{A}} V(A,\lambda,\varepsilon).
	\end{equation*}
	For every $U \subseteq X$, we have
	\begin{align*}
	\big(\ell^r_\omega ( f 1_{B} 1_{V^c}) (U)\big)^r & \leq \nu(U)^{-1} \sum_{A \in \mathcal{A}} \norm{f 1_{A} 1_{V(A,\lambda,\varepsilon)^c} 1_U}^r_{L^r(X,\omega)} \\
	& \leq \nu(U)^{-1} \sum_{A \in \mathcal{A}} \lambda^{r} \nu(U \cap A ) \\
	& \leq K \lambda^r,
	\end{align*}
	where we used the $r$-orthogonality of the classical $L^r$ quasi-norms of functions with disjoint support in the first inequality, \eqref{eq:condition_super_level_1_size} in the second, and the $\nu$-Carath\'{e}odory condition \eqref{eq:K_Carath\'{e}odory} in the third. Together with the subadditivity of $\nu$ and \eqref{eq:condition_super_level_1_measure}, the previous chain of inequalities yields
	\begin{equation*} 
	\nu(\ell^r_\omega(f 1_{B}) > K^{1/r} \lambda ) \leq  (1+\varepsilon) \sum_{A \in \mathcal{A}} \nu(\ell^r_\omega(f 1_{A}) > \lambda ).
	\end{equation*}
	By taking $\varepsilon$ arbitrarily small, we obtain the desired first inequality in \eqref{eq:L_q_orthogonality_level_set}.
	
	We turn to the second inequality in \eqref{eq:L_q_orthogonality_level_set}. Let $\varepsilon > 0$. Let $V(\lambda,\varepsilon)$ be an optimal set associated with the super level measure $\nu(\ell^r_\omega(f 1_{B}) > \lambda )$ up to the multiplicative constant $(1+\varepsilon)$, namely
	\begin{gather} \label{eq:condition_super_level_2_size}
	\norm{ f 1_{V(\lambda,\varepsilon)^c}}_{L^\infty_\nu(\ell^r_\omega)} \leq \lambda, \\ \label{eq:condition_super_level_2_measure}
	(1+ \varepsilon) \nu(\ell^r_\omega(f 1_{B}) > \lambda ) \geq \nu \big(V(\lambda,\varepsilon)\big).
	\end{gather}
	For every $U \subseteq X$, we have
	\begin{equation*}
	\big (\ell^r_\omega ( f 1_{A } 1_{V(\lambda,\varepsilon)^c}) (U) \big)^r \leq \nu(U)^{-1} \norm{f 1_{B} 1_{V(\lambda,\varepsilon)^c} 1_U}^r_{L^r(X,\omega)} \leq \lambda^r,
	\end{equation*}
	where we used the monotonicity of the classical $L^r$ quasi-norms in the first inequality, and \eqref{eq:condition_super_level_2_size} in the second. Together with the $\nu$-Carath\'{e}odory condition \eqref{eq:K_Carath\'{e}odory} and \eqref{eq:condition_super_level_2_measure}, the previous chain of inequalities yields
	\begin{equation*}
	\sum_{A \in \mathcal{A}} \nu(\ell^r_\omega(f 1_{A}) > \lambda ) \leq \sum_{A \in \mathcal{A}} \nu(V(\lambda) \cap A) \leq K \nu(\ell^r_\omega(f 1_{B}) > \lambda ).
	\end{equation*}
	By taking $\varepsilon$ arbitrarily small, we obtain the desired second inequality in \eqref{eq:L_q_orthogonality_level_set}.
\end{proof}

\subsection{Decomposition for double iterated outer $L^p$ spaces.}
We start with the result corresponding to Lemma \ref{thm:atomic_super_level_set_interior_level} in the case of a size given by a single iterated outer $L^p$ quasi-norm. The proof relies on the $q$-suborthogonality of the $L^q_\nu(\ell^r_\omega)$ quasi-norms of functions with disjoint supports as stated in \eqref{eq:L_q_suborthogonality} or in the second inequality in \eqref{eq:L_q_orthogonality}. Therefore, according to the relation between the exponents $q,r$, we allow the constants to depend on the parameter associated with the $\nu$-Carath\'{e}odory collection formed by the disjoint sets.   

\begin{lemma} \label{thm:atomic_super_level_set_exterior_level}
	For every $q \in (0, \infty)$, $r \in (0, \infty]$, $K \geq 1$, $N \geq 1$, there exist constants $C=C(q,r,K,N)$, $c=c(q,r,K,N)$ such that the following property holds true. 
		
	Let $X$ be a set, $\mu, \nu$ outer measures, and $\omega$ a measure. Let $f \in L^\infty_\mu(\ell_\nu^q(\ell_\omega^r))$ be a function on $X$, let $k \in \Z$ satisfy
	\begin{equation} \label{eq:bound_L_infty}
	\norm{f}_{L^\infty_\mu(\ell_\nu^q(\ell_\omega^r))} \in (2^k, 2^{k+1}],
	\end{equation}
	and let $\mathcal{A}$ be a $\nu$-Carath\'{e}odory collection of subsets of $X$ such that, for every $A \in \mathcal{A}$,
	\begin{equation} \label{eq:lower_bound_condition}
	\norm{f 1_{A}}_{L_\nu^q(\ell_\omega^r)}^q > 2^{(k-N)q} \mu(A).
	\end{equation}
	Then we have
	\begin{equation} \label{eq:atomic_super_level_set_exterior_level}
	\sum_{A \in \mathcal{A}} \mu(A) \leq C \mu(\ell^q_\nu(\ell^r_\omega)(f) > c 2^k).
	\end{equation}
	
	If $q \geq r$ and $X$ is finite, the constants $C,c$ do not depend on $K$.
\end{lemma}

\begin{proof}	
	{\textbf{Case I: arbitrary $q,r$.}} Let $\varepsilon > 0$. Let $F(c2^{k},\varepsilon)$ be an optimal set associated with the super level measure $\mu(\ell^q_\nu(\ell^r_\omega)(f)> c2^{k})$ up to the multiplicative constant $(1+\varepsilon)$, namely
	\begin{gather} \label{eq:condition_super_level_exterior_level_size}
	\norm{ f 1_{F(c2^{k}, \varepsilon)^c}}_{L^\infty_\mu(\ell^q_\nu(\ell^r_\omega))} \leq c2^{k}, \\ \label{eq:condition_super_level_exterior_level_measure}
	(1+\varepsilon) \mu(\ell^q_\nu(\ell^r_\omega)(f)> c2^{k}) \geq \mu\big(F(c2^{k}, \varepsilon)\big),
	\end{gather}  
	where $c$ will be fixed later. For $B = \bigcup_{A \in \mathcal{A}} A$, we have
	\begin{align*}
	\mu\big(F(c2^{k})\big) & \geq 2^{-(k+1)q}  \norm{ f 1_{F(c 2^{k})} 1_{B}}^q_{L^q_\nu(\ell^r_\omega)} \\
	& \geq  C 2^{-(k+1)q} \sum_{A \in \mathcal{A}} \norm{f 1_{F(c 2^{k})} 1_{A}}^q_{L^q_\nu(\ell^r_\omega)} \\
	& \geq  C 2^{-(k+1)q} \sum_{A \in \mathcal{A}} (C_{\Delta}^{-1} \norm{f 1_{A}}_{L^q_\nu(\ell^r_\omega)} - \norm{f 1_{A \setminus F(c 2^{k})}}_{L^q_\nu(\ell^r_\omega)})^q \\
	& \geq  C 2^{-(k+1)q} \sum_{A \in \mathcal{A}} ( C_{\Delta}^{-1} 2^{k-N} - c 2^{k})^q \mu(A),
	\end{align*}
	where we used the monotonicity of $\mu$ and \eqref{eq:bound_L_infty} in the first inequality, Lemma \ref{thm:L_q_orthogonality} applied to the $\nu$-Carath\'{e}odory collection $\mathcal{A}$ in the second, the quasi-triangle inequality for the outer $L^p$ quasi-norm of two summands in the third, and \eqref{eq:lower_bound_condition} and \eqref{eq:condition_super_level_exterior_level_size} in the fourth. By choosing 
	\begin{equation*}
	c=(2C_{\Delta})^{-1},
	\end{equation*}
	and taking $\varepsilon$ arbitrarily small, the previous chain of inequalities together with \eqref{eq:condition_super_level_exterior_level_measure} yields the desired inequality in \eqref{eq:atomic_super_level_set_exterior_level}.
	
	{\textbf{Case II: $q \geq r$.}} We use \eqref{eq:L_q_suborthogonality} from Lemma \ref{thm:L_q_subsuperorthogonality} applied to every arbitrary collection $\mathcal{A}$ of pairwise disjoint subsets of $X$ in place of Lemma \ref{thm:L_q_orthogonality}.
\end{proof}
	
We are now ready to state and prove a series of decomposition results for functions in the outer $L^p$ space with respect to a size of the form $\ell^q_\nu(\ell^r_\omega)$. Although the statements, as well as the proofs, are similar, we provide them separately in order to highlight the differences. The proofs rely on the selection of disjoint subsets where the size achieves the levels $\Psi^k$, for a certain $\Psi > 1$. The key ingredient in order to perform such a selection exhaustively at each step is the $q$-suborthogonality of the $L^q_\nu(\ell^r_\omega)$ quasi-norms of functions supported on certain disjoint sets. Therefore, according to the relation between the exponents $q,r$, we require the canopy condition \ref{def:parent_collection_nesting}, and we allow the constants to depend on the parameters associated with it. 

We start with a decomposition result in the full range of exponents under the assumption of the canopy condition \ref{def:parent_collection_nesting} on the setting.
\begin{proposition} \label{thm:atomic_decomposition_exterior_level_1_no_partition}
	For every $p ,q, r \in (0, \infty)$, $\Phi,K \geq 1$, there exist constants $C=C(p,q,r,\Phi,K)$, $c=c(p,q,r,\Phi,K)$ such that the following property holds true.
	
	Let $X$ be a finite set, $\mu,\nu$ outer measures, $\omega$ a measure, and $\mathcal{C}$ a $\mu$-covering function such that  $(X,\mu,\nu,\mathcal{C})$ satisfies	the canopy condition \ref{def:parent_collection_nesting}. For every function $f \in L^p_\mu(\ell^q_\nu(\ell^r_\omega))$ on $X$, there exists a collection $\{ E_{k} \colon k \in \Z \}$ of pairwise disjoint subsets of $X$ such that, if we set 
	\begin{equation*}
	F_k = \mathbf{B}_{\mathcal{C}} \big( \bigcup_{l \geq k} E_{l} \big),
	\end{equation*}
	then, for every $k \in \Z$, we have
	\begin{gather} 
	\label{eq:superlevel_exterior_level_1_no_partition}
	\ell^q_\nu(\ell^r_\omega) (f 1_{F_{k+1}^c})(E_{k}) > c 2^{k},  \qquad \qquad \text{when $E_{k} \neq \varnothing$,} \\
	\label{eq:maximal_choice_exterior_level_1_no_partition}
	\norm{f 1_{F_k^c}}_{L^\infty_\mu(\ell^q_\nu(\ell^r_\omega))} \leq 2^{k}, \\
	\label{eq:covering_exterior_level_1_no_partition}
	\mu(\ell^q_\nu(\ell^r_\omega) (f) > 2^{k}) \leq \mu( F_k ) , \\
	\label{eq:optimal_covering_exterior_level_1_no_partition}
	\mu (E_{k}) \leq C \mu(\ell^q_\nu(\ell^r_\omega) (f) > c 2^{k}).
	\end{gather}
	In particular, we have
	\begin{equation} \label{eq:outer_norm_as_size_and_measure_1_no_partition}
	\norm{f}^p_{L^p_\mu(\ell^q_\nu(\ell^r_\omega))} \sim_{p,q,r,\Phi,K}  \sum_{k \in \Z} 2^{kp} \mu (E_{k}) \sim_{p,q,r,\Phi,K} \sum_{k \in \Z} 2^{kp} \sum_{l \geq k} \mu (E_{l}).
	\end{equation}
\end{proposition}

\begin{proof}
	By \eqref{eq:Holder_containment}, we have $f \in L^\infty_\mu(\ell^q_\nu(\ell^r_\omega))$. We define the collection $\{ E_k \colon k \in \Z \}$ by a backward recursion on $k \in \Z$. For $k$ large enough such that
	\begin{equation*}
	\norm{f}_{L^\infty_\mu(\ell^q_\nu(\ell^r_\omega))} \leq 2^k,
	\end{equation*}
	we set $E_k$ to be empty. Now, we fix $k$ and assume to have selected $E_{l}$ for every $l > k$. In particular, $F_{k+1}$ is already well-defined. If there exists no subset $A$ of $X$ disjoint from $F_{k+1}$ such that
	\begin{equation} \label{eq:selection_condition_exterior_level}
	\ell^q_\nu(\ell^r_\omega)(f)(A) > 2^k,
	\end{equation}
	then we set $E_k$ to be empty, and proceed the recursion with $k-1$. 
	
	If there exists a subset $A$ of $X$ disjoint from $F_{k+1}$ satisfying \eqref{eq:selection_condition_exterior_level},
	we define an auxiliary $\nu$-Carath\'{e}odory collection $\{ E_{k,n} \colon n \in \N_k \}$ of subsets of $X$ by a forward recursion on $n \in \N_k$. The existence of $A$ provides the starting point $E_{k,1}$ for the recursion. Now, we fix $n$, assume to have selected $E_{k,m}$ for every $m \in \N, m < n$ forming a $\nu$-Carath\'{e}odory collection, and set 
	\begin{equation*}
	F_{k,n-1} = F_{k+1} \cup \mathbf{B}_{\mathcal{C}} \big( \bigcup_{m < n} E_{k,m} \big).
	\end{equation*}
	If there exists a subset $A$ of $X$ disjoint from $F_{k,n-1}$ satisfying \eqref{eq:selection_condition_exterior_level}, then we choose such a set $A$ to be $E_{k,n}$. By the canopy condition \ref{def:parent_collection_nesting}, we have that the collection $ \{ E_{k,m} \colon m \leq n \}$ is still $\nu$-Carath\'{e}odory. If no $A$ satisfying \eqref{eq:selection_condition_exterior_level} exists, we set $\N_k$ to be $\{ 1, \dots, n-1 \}$, stop the forward recursion, set 
	\begin{equation*}
	E_k = \bigcup_{n \in \N_k} E_{k,n}, 
	\end{equation*}
	and proceed the backward recursion with $k-1$. 
	
	By construction, we have \eqref{eq:maximal_choice_exterior_level_1_no_partition} and \eqref{eq:covering_exterior_level_1_no_partition} for every $k \in \Z$. By construction and Lemma \ref{thm:L_q_orthogonality} applied to the $\nu$-Carath\'{e}odory collection $\{ E_{k,n} \colon n \in \N_k \}$, we have \eqref{eq:superlevel_exterior_level_1_no_partition} for every nonempty $E_{k}$. To prove \eqref{eq:optimal_covering_exterior_level_1_no_partition}, we observe that for every $k$ such that $2^k$ is greater than the $L^\infty_\mu(\ell^q_\nu(\ell^r_\omega))$ quasi-norm of $f$, the statement is true. For every other $k$, the proof follows by construction and Lemma \ref{thm:atomic_super_level_set_exterior_level}.
	
	The equivalences in \eqref{eq:outer_norm_as_size_and_measure_1_no_partition} follow by \eqref{eq:covering_exterior_level_1_no_partition}, the definition of $F_k$, \eqref{eq:optimal_covering_exterior_level_1_no_partition}, Fubini, and the bounds for the geometric series.
\end{proof}

Under the assumption $q \geq r$ on the exponents, we can drop the assumption of the canopy condition \ref{def:parent_collection_nesting} on the setting. Moreover, for every function $f$, the collection $\{ E_{k} \colon k \in \Z \}$ produced by the decomposition forms a partition of the support of $f$.
\begin{proposition} \label{thm:atomic_decomposition_exterior_level_1_q_geq_r}
	For every $p ,q \in (0, \infty)$, $r \in (0, q]$, there exist constants $C=C(p,q,r)$, $c=c(p,q,r)$ such that the following property holds true.
	
	Let $X$ be a finite set, $\mu,\nu$ outer measures, $\omega$ a measure. For every function $f \in L^p_\mu(\ell^q_\nu(\ell^r_\omega))$ on $X$, there exists a collection $\{ E_{k} \colon k \in \Z \}$ of pairwise disjoint subsets of $X$ forming a partition of the support of $f$ such that, if we set 
	\begin{equation*} 
	F_k = \bigcup_{l \geq k} E_{l}.
	\end{equation*}
	then we have the same properties stated in  \eqref{eq:superlevel_exterior_level_1_no_partition}--\eqref{eq:outer_norm_as_size_and_measure_1_no_partition}.
\end{proposition}

\begin{proof}
	The argument is analogous to that in the previous proof. The only difference is in the definition of $E_k$, for which we do not need a second forward recursion. 
	
	In fact, we fix $k$ and assume to have selected $E_{l}$ for every $l > k$. In particular, $F_{k+1}$ is already well-defined. If there exists a subset $A$ of $X$ disjoint from $F_{k+1}$ satisfying \eqref{eq:selection_condition_exterior_level}, we set it to be $E_k$ making sure that
	\begin{equation*}
	\norm{f 1_{(A \cup F_{k+1})^c}} \leq 2^k.
	\end{equation*}
	We can fulfil this condition. In fact, if there exists a subset $B$ of $X$ disjoint from $A \cup F_{k+1}$ satisfying \eqref{eq:selection_condition_exterior_level}, then, by \eqref{eq:L_q_suborthogonality} in Lemma \ref{thm:L_q_subsuperorthogonality} and the subadditivity of $\nu$, also $A \cup B$ satisfies \eqref{eq:selection_condition_exterior_level}. 
	
	Due to the definition of $F_k$, the collection $\{ E_k \colon k \in \Z \}$ forms a partition of the support of $f$.
\end{proof}

Under the assumption of the canopy condition \ref{def:parent_collection_nesting} on the setting, we can recover a partition of the support of the function $f$ in the full range of exponents by a slightly different decomposition.
\begin{proposition} \label{thm:atomic_decomposition_exterior_level_2}
	For every $p ,q,r \in (0, \infty)$, $\Phi,K \geq 1$, there exist constants $C=C(p,q,r,\Phi,K)$, $c=c(p,q,r,\Phi,K)$, $\Psi= \Psi(\Phi,p)$ such that the following property holds true.
	
	Let $X$ be a set, $\mu,\nu$ outer measures, $\omega$ a measure, and $\mathcal{C}$ a $\mu$-covering function such that  $(X,\mu,\nu,\mathcal{C})$ satisfies	the canopy condition \ref{def:parent_collection_nesting}. For every function $f \in L^p_\mu(\ell^q_\nu(\ell^r_\omega))$ on $X$, there exists a collection $\{ E_{k} \colon k \in \Z \}$ of pairwise disjoint subsets of $X$ such that, if we set 
	\begin{equation*}
	F_k = \mathbf{B}_{\mathcal{C}} \big( \mathbf{B}_{\mathcal{C}}(F_{k+1} \cup E_{k}) \big),
	\end{equation*}
	then we have the same properties stated in  \eqref{eq:superlevel_exterior_level_1_no_partition}--\eqref{eq:optimal_covering_exterior_level_1_no_partition} with $2^k$ replaced by $\Psi^k$. 
	
	In particular, the $\nu$-Carath\'{e}odory collections $\{ \widetilde{E}^1_{k} \colon k \in \Z \}, \{ \widetilde{E}^2_{k} \colon k \in \Z \}$ defined by
	\begin{equation} \label{eq:definition_E_tilde_k}
	\widetilde{E}^1_k = \mathbf{B}_{\mathcal{C}}(F_{k+1} \cup E_{k}) \setminus F_{k+1}, \qquad \qquad \widetilde{E}^2_k = F_{k} \setminus \mathbf{B}_{\mathcal{C}}(F_{k+1} \cup E_{k}),
	\end{equation}
	form a partition of the support of $f$, and we have
	\begin{equation} \label{eq:outer_norm_as_size_and_measure_2}
	\norm{f}^p_{L^p_\mu(\ell^q_\nu(\ell^r_\omega))} \sim_{p,q,r,\Phi,K}  \sum_{k \in \Z} \Psi^{kp} \mu(E_{k}) \sim_{p,q,r,\Phi,K} \sum_{k \in \Z} \Psi^{kp} \big( \mu(\widetilde{E}^1_{k}) + \mu(\widetilde{E}^2_{k}) \big) .
	\end{equation}
\end{proposition}

\begin{proof}
	The argument is analogous to that in the proof of Proposition \ref{thm:atomic_decomposition_exterior_level_1_no_partition}. The only difference is that we replace the levels $2^k$ with the levels $\Psi^k$, where 
	\begin{equation*}
	\Psi = \Phi^{\frac{3}{p}}.
	\end{equation*} 
			
	In fact, we define $E_k$ by a double recursion as before, and $\widetilde{E}^1_{k}, \widetilde{E}^2_{k}$ as in \eqref{eq:definition_E_tilde_k}. Due to their definition, the collections $\{ \widetilde{E}^1_{k} \colon k \in \Z \}, \{ \widetilde{E}^2_{k} \colon k \in \Z \}$ are $\nu$-Carath\'{e}odory and they form a partition of the support of $f$.
	
	We turn now to the proof of the desired equivalences in \eqref{eq:outer_norm_as_size_and_measure_2}. By the properties corresponding to \eqref{eq:optimal_covering_exterior_level_1_no_partition} and \eqref{eq:covering_exterior_level_1_no_partition} in this setting, and the definition of $F_k$, we have
	\begin{align*}
	\sum_{k \in \Z} \Psi^{kp}\mu(E_k) & \leq C \sum_{k \in \Z} \Psi^{kp}\mu(\ell^q_\nu(\ell^r_\omega) (f) > c \Psi^{k}) \\
	& \leq C \norm{f}^p_{L^p_\mu(\ell^q_\nu(\ell^r_\omega))} \\
	& \leq C \sum_{k \in \Z} \Psi^{kp}\mu(\ell^q_\nu(\ell^r_\omega) (f) > \Psi^{k}) \\
	& \leq C \sum_{k \in \Z} \Psi^{kp} \sum_{l \geq k} \big( \mu(\widetilde{E}^1_{l}) + \mu(\widetilde{E}^2_{l}) \big).
	\end{align*}
	Moreover, by \eqref{eq:definition_E_tilde_k}, $\mathcal{C}$ being a $\mu$-covering function, and the definition of $\Psi$, we have
	\begin{align*}
	\sum_{k \in \Z} \Psi^{kp} \sum_{l \geq k} \big( \mu(\widetilde{E}^1_{l}) + \mu(\widetilde{E}^2_{l}) \big) & \leq C \sum_{k \in \Z} \Psi^{kp} \sum_{l \geq k} \sum_{j \geq l} \Phi^{2(j-l)} \mu(E_{j}) \\
	& \leq C \sum_{k \in \Z} \Psi^{kp} \sum_{j \geq k} \Phi^{2(j-k)} \mu(E_{j}) \\
	& \leq C \sum_{k \in \Z} \sum_{j \geq k} \Phi^{k-j} \Psi^{jp} \mu(E_{j}) \\
	& \leq C \sum_{j \in \Z} \Psi^{jp} \mu(E_{j}).
	\end{align*}
\end{proof}

We are now ready to prove Theorem \ref{thm:collapsing_2_step_iteration_finite}.

\begin{proof} [Proof of Theorem \ref{thm:collapsing_2_step_iteration_finite}]
	The case $q=\infty$ follows by definition. Therefore, without loss of generality, we assume $q=1$. 
	
	{\textbf{Case I: arbitrary $r \in (0, \infty]$.}} For a function $f \in L^1_\mu(\ell^1_\nu(\ell^{r}_\omega))$, let $\{E_{k} \colon k \in \Z \}$, $\{\widetilde{E}^1_{k} \colon k \in \Z \}$, $\{\widetilde{E}^2_{k} \colon k \in \Z \}$ be the collections of pairwise disjoint subsets of $X$ as in Proposition \ref{thm:atomic_decomposition_exterior_level_2}. By \eqref{eq:outer_norm_as_size_and_measure_2}, the property corresponding to \eqref{eq:superlevel_exterior_level_1_no_partition}, and Lemma \ref{thm:L_q_orthogonality}, we have
	\begin{align*}
	\norm{f}_{L^1_\mu(\ell^1_\nu(\ell^{r}_\omega))} &\leq C \sum_{k \in \Z} \Psi^{k} \mu(E_{k}) 
	\leq C \sum_{k \in \Z}  \norm{f 1_{E_{k}}}_{L^1_\nu(\ell^r_\omega)} 
	\leq C \norm{\sum_{k \in \Z} f 1_{E_{k}}}_{L^1_\nu(\ell^r_\omega)} \\
	&\leq C \norm{f}_{L^1_\nu(\ell^r_\omega)}.
	\end{align*}
	Moreover, by the quasi-triangle inequality for the outer $L^p$ quasi-norm of two summands, Lemma \ref{thm:L_q_orthogonality}, the property corresponding to \eqref{eq:maximal_choice_exterior_level_1_no_partition}, and \eqref{eq:outer_norm_as_size_and_measure_2}, we have
	\begin{align*}
	\norm{f}_{L^1_\nu(\ell^r_\omega)} & \leq C ( \norm{ \sum_{k \in \Z} f 1_{\widetilde{E}^1_k} }_{L^1_\nu(\ell^r_\omega)} + \norm{ \sum_{k \in \Z} f 1_{\widetilde{E}^2_k} }_{L^1_\nu(\ell^r_\omega)} ) \\
	& \leq C ( \sum_{k \in \Z} \norm{ f 1_{\widetilde{E}^1_k} }_{L^1_\nu(\ell^r_\omega)} + \sum_{k \in \Z} \norm{ f 1_{\widetilde{E}^2_k} }_{L^1_\nu(\ell^r_\omega)} ) \\	
	& \leq C \sum_{k \in \Z} \Psi^{k} \big( \mu(\widetilde{E}^1_k) + \mu(\widetilde{E}^2_k) \big) \\
	& \leq C \norm{f}_{L^1_\mu(\ell^1_\nu(\ell^r_\omega))}.
	\end{align*}
	
	{\textbf{Case II: $q \geq r$.}} For a function $f \in L^1_\mu(\ell^1_\nu(\ell^{r}_\omega))$, let $\{E_{k} \colon k \in \Z \}$ be the collection of pairwise disjoint subsets of $X$ as in Proposition \ref{thm:atomic_decomposition_exterior_level_1_q_geq_r}. By the properties corresponding to \eqref{eq:outer_norm_as_size_and_measure_1_no_partition} and \eqref{eq:superlevel_exterior_level_1_no_partition}, and \eqref{eq:L_q_suborthogonality} in Lemma \ref{thm:L_q_subsuperorthogonality}, we have
	\begin{align*}
	\norm{f}_{L^1_\mu(\ell^1_\nu(\ell^{r}_\omega))} &\leq C \sum_{k \in \Z} 2^{k} \mu(E_{k}) 
	\leq C \sum_{k \in \Z}  \norm{f 1_{E_{k}}}_{L^1_\nu(\ell^r_\omega)} 
	\leq C \norm{\sum_{k \in \Z} f 1_{E_{k}}}_{L^1_\nu(\ell^r_\omega)} \\
	&\leq C \norm{f}_{L^1_\nu(\ell^r_\omega)}.
	\end{align*}

	{\textbf{Case III: $q \leq r$.}} For a function $f \in L^1_\mu(\ell^1_\nu(\ell^{r}_\omega))$, let $\{A_{k} \colon k \in \Z \}$ be the collection of optimal sets associated with the super level measures $\mu(\ell^1_\nu(\ell^r_\omega)(f) > 2^{k})$, namely
	\begin{gather} \label{eq:condition_super_level_exterior_level_special_size}
	\norm{ f 1_{A_{k}^c}}_{L^\infty_\mu(\ell^1_\nu(\ell^r_\omega))} \leq 2^{k}, \\ \label{eq:condition_super_level_exterior_level_special_measure}
	\mu(\ell^1_\nu(\ell^r_\omega)(f) > 2^{k}) = \mu(A_k).
	\end{gather}  
	By \eqref{eq:L_q_superorthogonality} in Lemma \ref{thm:L_q_subsuperorthogonality}, \eqref{eq:condition_super_level_exterior_level_special_size}, the monotonicity of $\mu$, and \eqref{eq:condition_super_level_exterior_level_special_measure}, we have
	\begin{align*}
	\norm{f}_{L^1_\nu(\ell^r_\omega)} & \leq C \sum_{k \in \Z} \norm{f 1_{A_k \setminus A_{k+1}}}_{L^1_\nu(\ell^r_\omega)}		
	\leq C \sum_{k \in \Z} 2^{k+1} \mu( A_k \setminus A_{k+1})		
	\leq C \sum_{k \in \Z} 2^k \mu( A_k ) \\
	& \leq C \norm{f}_{L^1_\mu(\ell^1_\nu(\ell^r_\omega))}.
	\end{align*}
\end{proof}

\subsection{Dualizing function candidate}

We start recalling the setting. Let $p, q,r \in (1,\infty)$, $\Phi, K \geq 1$. Let $X$ be a finite set, $\mu,\nu$ outer measures, $\omega$ a measure, and $\mathcal{C}$ a $\mu$-covering function. For $q < r$, we assume $(X,\mu,\nu,\mathcal{C})$ to satisfy the canopy condition \ref{def:parent_collection_nesting}. For $q > r$, we assume $(X,\mu,\nu,\mathcal{C})$ to satisfy the crop condition \ref{def:anti_stacking_compatibility}. 

When $q=r$, the double iterated outer $L^p$ quasi-norm is isomorphic to a single iterated one, and the results stated in Theorem \ref{thm:Holder_triangular_2_step_iteration_finite} correspond to properties $(ii)$, $(iii)$ of Theorem 1.1 in \cite{2020arXiv200105903F}. 

When $q \neq r$, for a function $f \in L^p_\mu(\ell^q_\nu(\ell^r_\omega))$ on $X$, we provide the candidate dualizing function $g$ on $X$. We distinguish two cases.

\textbf{Case 1: $q>r$.} Let $\{ E_{k} \colon k \in \Z\}$ be the collection of pairwise disjoint subsets of $X$ associated with the function $f$ and the size $\ell^q_\nu(\ell^r_\omega)$ as in Proposition \ref{thm:atomic_decomposition_exterior_level_1_q_geq_r}.

\textbf{Case 2: $q<r$.} Let $\{ E_{k} \colon k \in \Z\}$ be the collection of pairwise disjoint subsets of $X$ associated with the function $f$ and the size $\ell^q_\nu(\ell^r_\omega)$ as in Proposition \ref{thm:atomic_decomposition_exterior_level_1_no_partition}.

In both cases, let $\{ U^{k}_{j} \colon j \in \Z \}$ be the collection of pairwise disjoint subsets of $E_k$ associated with the function $f 1_{E_{k}}$ and the size $\ell^r_\omega$ as in Proposition \ref{thm:atomic_decomposition_interior_level}. We define
\begin{align*} 
f_{k,j}(x) =& f(x) 1_{U^{k}_{j}}(x), \\
f_{k}(x) =& \sum_{j \in \Z} f_{k,j} (x) = f(x) \sum_{j \in \Z} 1_{U^{k}_{j}}(x).
\end{align*}

When $q > r$, let
\begin{equation*}
M = 2 + \Big\lfloor \frac{\log_2 K}{r} \Big\rfloor,
\end{equation*}
where $\lfloor x \rfloor$ is the largest integer smaller or equal than $x$. For 
\begin{equation*}
\mathcal{F}^{k}_j = \{ F \in \mathcal{E} \colon \ell^r_\omega(f_{k,j})(F) \leq 2^{j-M} \},
\end{equation*}
let $\mathcal{G}^{k}_{j}$ be its $\nu$-Carath\'{e}odory subcollection as in the crop condition \ref{def:anti_stacking_compatibility}, and set
\begin{equation*}
\widetilde{U}^k_j = U^{k}_{j} \setminus \bigcup_{G \in \mathcal{G}^{k}_{j}} G.
\end{equation*}

We set
\begin{equation*}
W^k_j = 
\begin{dcases}
\widetilde{U}^k_j, \qquad  & \textrm{for $q > r$,} \\
U^k_j, \qquad  & \textrm{for $q < r$.}
\end{dcases}
\end{equation*}
and we define
\begin{equation} \label{eq:dual_function}
\begin{split}
g_{k,j}(x) =& f(x)^{r-1} 1_{W^{k}_{j}}(x), \\
g_{k}(x) =& \sum_{j \in \Z} 2^{j(q-r)} g_{k,j} (x) = f(x)^{r-1} \sum_{j \in \Z} 2^{j(q-r)} 1_{W^{k}_{j}}(x), \\
g(x) = & \sum_{k \in \Z} 2^{k(p-q)} g_{k}(x) = f(x)^{r-1} \sum_{k \in \Z} 2^{k(p-q)} \sum_{j \in \Z} 2^{j(q-r)} 1_{W^{k}_{j}}(x).
\end{split}
\end{equation}

\begin{lemma} \label{thm:lower_bound_size_outside_exceptional_set}
	Let $p,q,r \in (1,\infty)$, $q \neq r$, $\Phi,K \geq 1$. There exists a constant $c=c(r,K)$ such that, for every function $f \in L^p_\mu(\ell^q_\nu(\ell^r_\omega))$ on $X$, we have
	\begin{equation} \label{eq:lower_bound_size_outside_exceptional_set}
	\norm{f_{k,j}^r 1_{W^k_j}}_{L^1(X,\omega)} \geq c 2^{jr} \nu(U^{k}_{j}).	
	\end{equation}
\end{lemma}
\begin{proof}
	{\textbf{Case I: $q > r$.}} We have
	\begin{align*}
	\norm{f_{k,j}^r 1_{W^k_j}}_{L^1(X,\omega)} 
	& \geq \norm{f_{k,j}^r }_{L^1(X,\omega)} - \sum_{G \in \mathcal{G}^{k}_{j}} \norm{f_{k,j}^r 1_G}_{L^1(X,\omega)} \\
	& \geq  2^{jr} \nu(U^{k}_{j}) - \sum_{G \in \mathcal{G}^{k}_{j}} 2^{(j-M)r} \nu(U^{k}_{j} \cap G) \\
	& \geq 2^{jr} \nu(U^{k}_{j}) - K 2^{(j-M)r} \nu(U^{k}_{j}) \\
	& \geq c 2^{jr} \nu(U^{k}_{j}),
	\end{align*}
	where we used \eqref{eq:superlevel_interior_level} and the control on the size $\ell^r_\omega$ defining the elements of $\mathcal{F}^{k}_{j}$ in the second inequality, the $\nu$-Carath\'{e}odory condition \eqref{eq:K_Carath\'{e}odory} for the collection $\mathcal{G}^{k}_{j}$ in the third, and the definition of $M$ in the fourth. 
		
	{\textbf{Case II: $q < r$.}} The desired inequality follows by \eqref{eq:superlevel_interior_level}. 
\end{proof}

The definition of $g$ guarantees the following good lower bound on the classical $L^1$ norm of $fg$, and good upper bound on the outer $L^{p'}_\mu(\ell^{q'}_\nu(\ell^{r'}_\omega))$ quasi-norm of $g$.
\begin{lemma} \label{thm:lower_bound}
	Let $p,q,r \in (1,\infty)$, $q \neq r$, $\Phi,K \geq 1$. There exists a constant $c=c(p,q,r,\Phi,K)$ such that, for every function $f \in L^p_\mu(\ell^q_\nu(\ell^r_\omega))$ on $X$, for $g$ defined by \eqref{eq:dual_function}, then
	\begin{equation*}
	\norm{fg}_{L^1(X,\omega)} \geq c \norm{f}^p_{L^p_\mu(\ell^q_\nu(\ell^r_\omega))}. 
	\end{equation*}
\end{lemma}
\begin{proof}	
	By \eqref{eq:lower_bound_size_outside_exceptional_set} and \eqref{eq:outer_norm_as_size_and_measure_interior_level}, we have
	\begin{align*}
	\norm{fg}_{L^1(X,\omega)} & = \sum_{k \in \Z} 2^{k(p-q)} \sum_{j \in \Z}  2^{j(q-r)} \norm{f_{k,j}^r 1_{W^k_j}}_{L^1(X,\omega)} \geq  c \sum_{k \in \Z} 2^{k(p-q)} \sum_{j \in \Z}   2^{jq} \nu(U^{k}_{j}) \\
	& \geq  c \sum_{k \in \Z} 2^{k(p-q)} \norm{f_{k}}^q_{L^q_\nu(\ell^r_\omega)}.
	\end{align*}
	For $q<r$, by \eqref{eq:superlevel_exterior_level_1_no_partition} and \eqref{eq:outer_norm_as_size_and_measure_1_no_partition}, we have
	\begin{equation*}
	\sum_{k \in \Z} 2^{k(p-q)} \norm{f_{k}}^q_{L^q_\nu(\ell^r_\omega)} \geq  c \sum_{k \in \Z} 2^{kp} \mu(E_{k}) \geq  c \norm{f}^p_{L^p_\mu(\ell^q_\nu(\ell^r_\omega))}. 
	\end{equation*}
	For $q > r$, the properties in Proposition \ref{thm:atomic_decomposition_exterior_level_1_q_geq_r} corresponding to \eqref{eq:superlevel_exterior_level_1_no_partition} and \eqref{eq:outer_norm_as_size_and_measure_1_no_partition} yield the analogous chain of inequalities.
\end{proof}

\begin{lemma} \label{thm:upper_bound}
	Let $p,q,r \in (1,\infty)$, $q \neq r$, $\Phi,K \geq 1$. There exists a constant $C=C(p,q,r,\Phi,K)$ such that, for every function $f \in L^p_\mu(\ell^q_\nu(\ell^r_\omega))$ on $X$, for $g$ defined by \eqref{eq:dual_function}, then
	\begin{equation} \label{eq:upper_bound}
	\norm{g}^{p'}_{L^{p'}_\mu(\ell^{q'}_\nu(\ell^{r'}_\omega))} \leq C \norm{f}^p_{L^p_\mu(\ell^q_\nu(\ell^r_\omega))}. 
	\end{equation}
\end{lemma}

\begin{proof}
	{\textbf{Case I: q > r.}} Let $\widetilde{k},j$ be fixed. For every subset $F$ of $X$, for every subset $U$ of $F$, we have
	\begin{equation*} 
	\begin{split}
	\ell^{r'}_\omega (g_{\widetilde{k}} 1_F 1_{(V^{\widetilde{k}}_{j})^c}) (U) &\leq \sum_{\widetilde{j} < j}  2^{\widetilde{j} (q-r)} ( \nu(U)^{-1} \norm{g_{\widetilde{k},\widetilde{j}} 1_{U \setminus V^{\widetilde{k}}_{\widetilde{j}+1} } }^{r'}_{L^{r'}_\omega})^{\frac{1}{r'}} \\
	& \leq \sum_{\widetilde{j} < j} 2^{\widetilde{j} (q-r)}  ( \nu(U)^{-1} \norm{f_{\widetilde{k},\widetilde{j}} 1_{U \setminus V^{\widetilde{k}}_{\widetilde{j}+1} } }^{r}_{L^{r}_\omega})^{\frac{1}{r'}} \\
	&\leq c 2^{j (q-1)},
	\end{split}
	\end{equation*}
	where we used the triangle inequality for the classical $L^{r'}$ norm in the first inequality, and \eqref{eq:maximal_choice_interior_level} in the third. The previous chain of inequalities yields
	\begin{equation} \label{eq:level_set_norm_realizing_function_interior_level}
	\nu(\ell^{r'}_\omega (g_{\widetilde{k}} 1_F) > c 2^{j(q-1)}) \leq \sum_{ \widetilde{j} \geq j} \nu( W^{\widetilde{k}}_{\widetilde{j}} \cap F).
	\end{equation}
	Moreover, for every fixed $\widetilde{j} \in \Z$, for $E = \mathbf{B}_{\mathcal{C}}(F)$, we have
	\begin{equation} \label{eq:optimal_theta_picking_interior_level}
	\nu( W^{\widetilde{k}}_{\widetilde{j}}  \cap F) \leq C \nu(\ell^{r}_\omega (f_{\widetilde{k}} 1_{E} ) > \widetilde{c} 2^{\widetilde{j}}).
	\end{equation}
	In fact, we have two cases.
	\begin{enumerate} [(i)]
		\item If $W^{\widetilde{k}}_{\widetilde{j}} \cap F = \varnothing$, the left hand side in \eqref{eq:optimal_theta_picking_interior_level} is $0$, and the inequality holds true.
		\item If $W^{\widetilde{k}}_{\widetilde{j}} \cap F \neq \varnothing$, by the crop condition \ref{def:anti_stacking_compatibility}, we have that $E' = \mathbf{B}_{\mathcal{C}}(W^{\widetilde{k}}_{\widetilde{j}} \cap F) \subseteq E$ is covered by a collection of disjoint subsets that are not in $\mathcal{F}^{\widetilde{k}}_{\widetilde{j}}$, so that
		\begin{equation*}
		\ell^r_\omega(f_{\widetilde{k},\widetilde{j}} 1_{E})(U^{\widetilde{k}}_{\widetilde{j}} \cap E') \geq \widetilde{c} 2^{\widetilde{j}},
		\end{equation*}
		hence, by Lemma \ref{thm:atomic_super_level_set_interior_level}, we obtain \eqref{eq:optimal_theta_picking_interior_level}.
	\end{enumerate}
	Therefore, by \eqref{eq:level_set_norm_realizing_function_interior_level} and \eqref{eq:optimal_theta_picking_interior_level}, we have
	\begin{equation} \label{eq:upper_bound_norm_inner}
	\begin{split}
	\norm{g_{\widetilde{k}} 1_F}_{L^{q'}_\nu(\ell^{r'}_\omega)}^{q'} &\leq C \sum_{j \in \Z} 2^{jq} \nu(\ell^{r'}_\omega (g_{\widetilde{k}} 1_F) > c 2^{j(q-1)}) \\
	&\leq C \sum_{j \in \Z} 2^{jq} \sum_{\widetilde{j} \geq j} \nu(\ell^{r}_\omega (f_{\widetilde{k}} 1_{E}) > \widetilde{c} 2^{\widetilde{j}}) \\
	&\leq C \norm{f_{\widetilde{k}} 1_{E}}_{L^q_\nu(\ell^r_\omega)}^q.
	\end{split}
	\end{equation}
	Hence, we have
	\begin{equation*}
	\begin{split}
	\ell^{q'}_\nu(\ell^{r'}_\omega) (g 1_{F_k^c})(F) & \leq C \sum_{\widetilde{k}<k} 2^{\widetilde{k} (p-q)} ( \mu(F)^{-1} \norm{g_{\widetilde{k}} 1_{F}}^{q'}_{L^{q'}_\nu(\ell^{r'}_\omega)})^{\frac{1}{q'}} \\
	& \leq C \sum_{\widetilde{k}<k} 2^{\widetilde{k} (p-q)} ( \mu(F)^{-1} \norm{f_{\widetilde{k}} 1_{E}}_{L^q_\nu(\ell^r_\omega)})^{\frac{1}{q'}} \\
	&\leq C 2^{k(p-1)},
	\end{split}
	\end{equation*}
	where we used the quasi-triangle inequality for the outer $L^{q'}_\nu(\ell^{r'}_\omega)$ quasi-norm proved in \cite{2020arXiv200105903F} in the first inequality, \eqref{eq:upper_bound_norm_inner} in the second, the property in Proposition \ref{thm:atomic_decomposition_exterior_level_1_q_geq_r} corresponding to \eqref{eq:maximal_choice_exterior_level_1_no_partition} and \eqref{eq:parent_optimality} in the third. The previous chain of inequalities yields
	\begin{equation*}
	\mu( \ell^{q'}_\nu(\ell^{r'}_\omega) (g) > C 2^{k(p-1)}) \leq \mu(F_k) \leq \widetilde{C} \sum_{\widetilde{k} \geq k} \mu(E_{\widetilde{k}}).
	\end{equation*}
	Together with the property in Proposition \ref{thm:atomic_decomposition_exterior_level_1_q_geq_r} corresponding to \eqref{eq:outer_norm_as_size_and_measure_1_no_partition}, this yields
	\begin{align*}
	\norm{g}_{L^{p'}_\mu(\ell^{q'}_\nu(\ell^{r'}_\omega))}^{p'} &\leq \widetilde{C} \sum_{k \in \Z} 2^{kp} \mu( \ell^{q'}_\nu(\ell^{r'}_\omega) (g) > C 2^{k(p-1)}) \\
	&\leq \widetilde{C} \sum_{k \in \Z} 2^{kp} \sum_{\widetilde{k} \geq k} \mu(E_{\widetilde{k}}) \\
	&\leq \widetilde{C} \norm{f}_{L^p_\mu(\ell^q_\nu(\ell^r_\omega))}^p.
	\end{align*}
	
	{\textbf{Case II: $q < r$.}} Let $\widetilde{k}$ be fixed. It is enough to prove that, for every subset $F$ of $X$, we have
	\begin{equation} \label{eq:Lq_ellr_size_control} 
	\norm{g_{\widetilde{k}} 1_{F}}_{L^{q'}_\nu(\ell^{r'}_\omega)}^{q'} \leq C \norm{f_{\widetilde{k}} 1_{F}}_{L^q_\nu(\ell^r_\omega)}^q .
	\end{equation}
	The desired inequality in \eqref{eq:upper_bound} then follows as in the previous case.
	
	Let $j$ be fixed. Let $V(2^{j})$ be an optimal set associated with the super level measure $\nu(\ell^r_\omega(f_{\widetilde{k}} 1_{F})> 2^{j})$, namely
	\begin{gather} \label{eq:condition_super_level_interior_level_special_size}
	\norm{ f_{\widetilde{k}} 1_{F} 1_{V(2^{j})^c}}_{L^\infty_\nu(\ell^r_\omega)} \leq 2^{j}, \\ \label{eq:condition_super_level_interior_level_special_measure}
	\nu(\ell^r_\omega( f_{\widetilde{k}} 1_{F} )> 2^{j}) = \nu\big(V(2^{j})\big).
	\end{gather}  
	For every subset $U$ of $F$, we have
	\begin{align*}
	\ell^{r'}_\omega (g_{\widetilde{k}} 1_F 1_{V(2^{j})^c} ) (U) & \leq \sum_{\widetilde{j} < j}  2^{\widetilde{j} (q-r)} ( \nu(U)^{-1} \norm{g_{\widetilde{k},\widetilde{j}} 1_{U  \setminus V^{\widetilde{k}}_{\widetilde{j}+1}}  }^{r'}_{L^{r'}_\omega})^{\frac{1}{r'}} + \\
	& \qquad \qquad + ( \nu(U)^{-1} \norm{ \sum_{\widetilde{j} \geq j}  2^{\widetilde{j} (q-r)} g_{\widetilde{k},\widetilde{j}} 1_{F} 1_{U  \setminus V(2^{j})} }^{r'}_{L^{r'}_\omega})^{\frac{1}{r'}} \\
	& \leq \sum_{\widetilde{j} < j} 2^{\widetilde{j} (q-r)}  ( \nu(U)^{-1} \norm{f_{\widetilde{k},\widetilde{j}} 1_{U \setminus V^{\widetilde{k}}_{\widetilde{j}+1}} }^{r}_{L^{r}_\omega})^{\frac{1}{r'}} + \\
	& \qquad \qquad + 2^{j (q-r)} ( \nu(U)^{-1} \norm{ \sum_{\widetilde{j} \geq j} f_{\widetilde{k},\widetilde{j}} 1_{F} 1_{U  \setminus V(2^{j})} }^{r'}_{L^{r'}_\omega})^{\frac{1}{r'}} \\
	& \leq c 2^{j (q-1)} ,
	\end{align*}
	where we used the triangle inequality for the classical $L^{r'}$ norm in the first inequality, the condition $q<r$ in the second, \eqref{eq:maximal_choice_exterior_level_1_no_partition} and \eqref{eq:condition_super_level_interior_level_special_size} in the third. Together with \eqref{eq:condition_super_level_interior_level_special_measure}, the previous chain of inequalities yields, for every $j \in \Z$,
	\begin{equation*}
	\nu(\ell^{r'}_\omega( g_{\widetilde{k}} 1_{F} )> c2^{j(q-1)}) \leq \nu(\ell^r_\omega( f_{\widetilde{k}} 1_{F} )> 2^{j}).
	\end{equation*}
	The inequality in \eqref{eq:Lq_ellr_size_control} follows multiplying by $2^{jq}$ and summing in $j \in \Z$ on both sides.
\end{proof}

We are now ready to prove Theorem \ref{thm:Holder_triangular_2_step_iteration_finite}.

\begin{proof}[Proof of Theorem \ref{thm:Holder_triangular_2_step_iteration_finite}]
	When $q=r$, the double iterated outer $L^p$ quasi-norm is isomorphic to a single iterated one, and the proof corresponds to the one of properties $(ii)$, $(iii)$ of Theorem 1.1 in \cite{2020arXiv200105903F}.
	
	When $q \neq r$, we proceed as follows.
	
	{\textbf{Property (i).}} By \eqref{eq:collapsing_pairing}, the $L^1(X,\omega)$-pairing of two functions $f,g$ is equivalent to the outer $L^1_\mu(\ell^1_\nu(\ell^1_\omega))$ quasi-norm of the product $fg$. The first inequality in \eqref{eq:Holder_2_iterated} is then given by outer H\"{o}lder's inequality, Proposition 3.4 in \cite{MR3312633}. The second inequality in \eqref{eq:Holder_2_iterated} is a corollary of Lemma \ref{thm:lower_bound} and Lemma \ref{thm:upper_bound} for $f \in L^p_\mu(\ell^q_\nu(\ell^r_\omega))$.
	
	{\textbf{Property (ii).}} The inequality in \eqref{eq:triangle_2_iterated} is a corollary of the triangle inequality for the $L^1(X,\omega)$ norm and property $(i)$.
\end{proof}

\subsection{Counterexamples} \label{subsec:counterexamples}

For every $m \in \N$, we introduce the finite setting 
\begin{align*}
X_m & = \{ x_i \colon 1 \leq i \leq m \}, \\
\omega_m (A) & = \mu_m(A) = \abs{A}, \qquad && \textrm{for every $A \subseteq X_m$,}\\
\nu_m(A) & = 1, \qquad && \textrm{for every $\varnothing \neq A \subseteq X_m$,} \\
f_i & = 1_{x_{i}}, \qquad && \textrm{for every $1 \leq i \leq m$,} \\
f & = 1_{X_m}.
\end{align*}
In particular, the collection of singletons $\{ \{ x_i \} \colon 1 \leq i \leq m \}$ satisfies the $\nu_m$-Carath\'{e}odory condition with parameter $K_m \geq m$. 

First, we observe that, for every exponent $r \in (0,\infty]$, for every function $g$, for every nonempty subset $A$ of $X_m$, we have
\begin{equation*}
\ell^r_{\omega_m}(g)(A) = \norm{g 1_A}_{L^r(X_m, \omega_m)}.
\end{equation*}
Therefore, for every exponent $r \in (0,\infty]$, for every function $g$, we have
\begin{equation*}
\nu_m(\ell^r_{\omega_m}(g) > \lambda) = \begin{dcases}
\nu_m(X_m) = 1, \qquad & \textrm{for $\lambda \in [0, \norm{g}_{L^\infty_{\nu_m}(\ell^r_{\omega_m})})$,} \\
\nu_m(\varnothing) = 0, \qquad & \textrm{for $\lambda \in [ \norm{g}_{L^\infty_{\nu_m}(\ell^r_{\omega_m})}, \infty)$,}
\end{dcases}
\end{equation*}
where, here and later as well, for every level $\lambda$, we provide a subset of $X_m$ realizing the infimum in the definition of the super level measure in \eqref{eq:super_level_measure}. 

Hence, for every exponents $q,r \in (0,\infty]$, we have
\begin{equation*}
\norm{g}_{L^q_{\nu_m}(\ell^r_{\omega_m})} = \norm{g}_{L^\infty_{\nu_m}(\ell^r_{\omega_m})} = \norm{g}_{L^r(X_m,\omega_m)}.
\end{equation*}

In particular, for every exponent $r \in (0,\infty]$, we have
\begin{gather*}
\sum_{i=1}^m \norm{f_i}_{L^1_{\nu_m}(\ell^r_{\omega_m})} = \sum_{i=1}^m 1 = m, \\ 
\norm{\sum_{i=1}^m f_i}_{L^1_{\nu_m}(\ell^r_{\omega_m})} = \norm{f}_{L^1_{\nu_m}(\ell^r_{\omega_m})} = m^{\frac{1}{r}}.
\end{gather*}
When $r \in (0, \infty]$, $r \neq 1$, one of the constants $C_1,C_2$ of super- or $q$-suborthogonality in \eqref{eq:L_q_orthogonality} blows up as $m$ grows to infinity.

Next, we observe that, for every exponents $q,r \in (0, \infty]$, for every function $g$, for every nonempty subset $A$ of $X_m$, we have
\begin{equation*}
\ell^q_{\nu_m}(\ell^r_{\omega_m})(g)( A ) = \mu_{m}(A)^{-\frac{1}{q}} \norm{g 1_A}_{L^q_{\nu_m}(\ell^r_{\omega_m})} = \abs{A}^{-\frac{1}{q}} \norm{g 1_A}_{L^r(X_m,\omega_m)}, 
\end{equation*}
hence, for every exponent $r \in [1, \infty]$, for every strict subset $B$ of $X_m$, we have
\begin{equation*}
\norm{f 1_{B^c}}_{L^\infty_{\mu_m}(\ell^1_{\nu_m}(\ell^r_{\omega_m}))} = 1 = \ell^1_{\nu_m}(\ell^r_{\omega_m})(f 1_{B^c})(\{ x_i \}), \qquad \textrm{for every $x_i \notin B$.}
\end{equation*}
Therefore, for every exponent $r \in [1, \infty]$, we have
\begin{equation*}
\mu_m(\ell^1_{\nu_m}(\ell^r_{\omega_m})(f) > \lambda) = \begin{dcases}
\mu_m(X_m) = m, \qquad & \textrm{for $\lambda \in [0, 1)$,} \\
\mu_m(\varnothing) = 0, \qquad & \textrm{for $\lambda \in [ 1, \infty)$.}
\end{dcases}
\end{equation*}

In particular, for every exponent $r \in [1,\infty]$, we have
\begin{equation*}
\norm{f}_{L^1_{\mu_m}(\ell^1_{\nu_m}(\ell^r_{\omega_m}))} = m.
\end{equation*}
When $r \in (1,\infty]$, the constant $C_2$ of the "collapsing effect" in \eqref{eq:collapsing_2_iterated} blows up as $m$ grows to infinity.

Finally, we observe that, for every exponents $q \in (1, \infty)$, $r \in (1, q]$, for every strict subset $B$ of $X_m$, we have
\begin{equation*}
\norm{f 1_{B^c}}_{L^\infty_{\mu_m}(\ell^q_{\nu_m}(\ell^r_{\omega_m}))} = \abs{X_m \setminus B}^{\alpha} = \ell^q_{\nu_m}(\ell^r_{\omega_m})(f 1_{B^c})(B^c),
\end{equation*}
where $\alpha= \alpha(r,q)= \frac{1}{r}-\frac{1}{q}$. Therefore, for every exponents $q \in (1, \infty), r \in (1, q]$, we have, for $1 \leq i \leq m$,
\begin{equation*}
\mu_m(\ell^q_{\nu_m}(\ell^r_{\omega_m})(f) > \lambda) = \begin{dcases}
\mu_m(X^{m-i+1}_{m}) = m-i+1, \qquad & \textrm{for $\lambda \in [(i-1)^{\alpha}, i^{\alpha})$,} \\
\mu_m(\varnothing) = 0, \qquad & \textrm{for $\lambda \in [ m^{\alpha}, \infty)$,}
\end{dcases}
\end{equation*}
where $X^{j}_m$ is any arbitrary subset of $X_m$ of cardinality $j$. 

In particular, for every exponents $p,q \in (1, \infty)$, $r \in (1,q]$, there exists a constant $c=c(p,q,r)$ such that, for every $m \in \N$ big enough, we have
\begin{gather*}
\sum_{i=1}^m \norm{f_i}_{L^p_{\mu_m}(\ell^q_{\nu_m}(\ell^r_{\omega_m}))} = \sum_{i=1}^m 1 = m, \\ 
\norm{\sum_{i=1}^m f_i}_{L^p_{\mu_m}(\ell^q_{\nu_m}(\ell^r_{\omega_m}))} = \norm{f}_{L^p_{\mu_m}(\ell^q_{\nu_m}(\ell^r_{\omega_m}))} \geq c m^{\frac{1}{p}-\frac{1}{q}+\frac{1}{r}}.
\end{gather*}
Therefore, the constants of the sharpness of outer H\"{o}lder's inequality in \eqref{eq:Holder_2_iterated} and the triangle inequality in \eqref{eq:triangle_2_iterated} blow up as $m$ grows to infinity when
\begin{equation*}
p,q,r \in (1, \infty), \qquad \frac{1}{p} - \frac{1}{q} +\frac{1}{r} > 1.
\end{equation*}

Now, for every $m \in \N$, we slightly modify the previous finite setting
\begin{align*}
X_m & = \{ x_i \colon 1 \leq i \leq m \}, \\
\omega_m (A) & = \abs{A}, \qquad && \textrm{for every $A \subseteq X_m$,}\\
\nu_m(A) & = 1, \qquad && \textrm{for every $A \subseteq X_m$,} \\
\sigma_m (\{x_i\}) & = 2^{\beta(i-1)}, \qquad && \textrm{for every $1 \leq i \leq m$,}\\
f & = 1_{X_m},
\end{align*}
where $\beta = \beta(r) = \frac{2}{r}$, and let $\mu_m$ be the measure generated via \eqref{eq:outer_measure_from_pre_measure} from $\sigma_m$. As in the previous setting, the collection of singletons $\{ \{ x_i \} \colon 1 \leq i \leq m \}$ satisfies the $\nu_m$-Carath\'{e}odory condition with parameter $K_m \geq m$. 

As in the previous setting, for every exponents $q,r \in (0, \infty]$, for every function $g$, for every nonempty subset $A$ of $X_m$, we have
\begin{equation*}
\ell^q_{\nu_m}(\ell^r_{\omega_m})(g)( A ) = \mu_{m}(A)^{-\frac{1}{q}} \norm{g 1_A}_{L^q_{\nu_m}(\ell^r_{\omega_m})} = \mu_{m}(A)^{-\frac{1}{q}} \norm{g 1_A}_{L^r(X_m,\omega_m)}, 
\end{equation*}
hence, for every exponent $r \in (0, 1]$, for every strict subset $B$ of $X_m$, we have
\begin{equation*}
\norm{f 1_{B^c}}_{L^\infty_{\mu_m}(\ell^1_{\nu_m}(\ell^r_{\omega_m}))} = 2^{-\beta(j-1)} = \ell^1_{\nu_m}(\ell^r_{\omega_m})(f 1_{B^c})(\{ x_j \}), 
\end{equation*}
where $j = \min \{ i \colon 1 \leq i \leq m, x_i \notin B \}$. Therefore, for every exponent $r \in (0,1]$, we have, for $1 \leq j < m$,
\begin{equation*}
\mu_m(\ell^q_{\nu_m}(\ell^r_{\omega_m})(f) > \lambda) = \begin{dcases}
\mu_m(X_{m}) = \sum_{i=1}^{m} 2^{\beta(i-1)}, \qquad & \textrm{for $\lambda \in [0, 2^{-\beta(m-1)})$,} \\
\mu_m(X^{j}_{m}) = \sum_{i=1}^{j} 2^{\beta(i-1)}, \qquad & \textrm{for $\lambda \in [2^{-\beta j}, 2^{-\beta(j-1)})$,} \\
\mu_m(\varnothing) = 0, \qquad & \textrm{for $\lambda \in [ 1, \infty)$,}
\end{dcases}
\end{equation*}
where $X^{j}_m = \{ x_i \colon 1 \leq i \leq j \} \subseteq X_m$.

In particular, for every exponent $r \in (0,1]$, there exists a constant $C=C(r)$ such that we have
\begin{gather*}
\norm{f}_{L^1_{\nu_m}(\ell^r_{\omega_m})} = m^{\frac{1}{r}}, \\
\norm{f}_{L^1_{\mu_m}(\ell^1_{\nu_m}(\ell^r_{\omega_m}))} \leq C m.
\end{gather*}
When $r \in (0,1)$, the constant $C_1$ of the "collapsing effect" in \eqref{eq:collapsing_2_iterated} blows up as $m$ grows to infinity.

\section{Examples} \label{sec:examples}

In this section we exhibit three settings in which we provide a $\mu$-covering function $\mathcal{C}$ satisfying the canopy condition \ref{def:parent_collection_nesting} and the crop condition \ref{def:anti_stacking_compatibility}.

\subsection{Finite set with three measures} Let $X$ be a finite set, $\mu, \nu, \omega$ be three measures on it. The function $\mathcal{C}$ defined by
\begin{equation*}
\mathcal{E} = \{ \{ x \} \colon x \in X \}, \qquad \qquad \mathcal{C}(A) = \{ \{ x \} \colon x \in A \},
\end{equation*}
is a $\mu$-covering function with parameter $\Phi = 1$. The canopy and the crop conditions with parameters $\Phi=K=1$ are satisfied because every collection of pairwise disjoint subsets of $X$ is $\nu$-Carath\'{e}odory with parameter $K = 1$, since $\nu$ is a measure, and the very definition of $\mathcal{C}$. The same conditions are satisfied by 
\begin{equation*}
\mathcal{E}' = \mathcal{P}(X), \qquad \qquad \mathcal{C}'(A) = A.
\end{equation*}

\subsection{Cartesian product of three finite sets with measures} Let $X_1,X_2,X_3$ be finite sets with measures $\omega_1,\omega_2,\omega_3$. Let $\mu,\nu,\omega$ be the outer measures $\mu_1,\mu_2,\mu_3$ defined on $X$ as in \eqref{eq:cartesian_product_outer_measure}. The function $\mathcal{C}$ defined by
\begin{equation*}
\mathcal{E} = \{ X_1 \times X_2 \times \{ z \} \colon z \in X_3 \}, \qquad \qquad \mathcal{C}(A) = \{ X_1 \times X_2 \times \{ z \} \colon z \in \pi_3(A) \},
\end{equation*}
where $\pi_3$ is the projection in $X_3$, is a $\mu$-covering function with parameter $\Phi = 1$. The canopy and the crop conditions with parameters $\Phi=K=1$ are satisfied because every collection of disjoint subsets of $X$ of the form $X_1 \times X_2 \times Z$ is $\nu$-Carath\'{e}odory with parameter $K = 1$, since on these sets $\nu$ behaves like the measure $\omega_2 \otimes \omega_3$, and the very definition of $\mathcal{C}$. The same conditions are satisfied by 
\begin{equation*}
\mathcal{E}' = \{ X_1 \times X_2 \times Z \colon Z \in \mathcal{P}(X_3) \}, \qquad \qquad \mathcal{C}'(A) = X_1 \times X_2 \times \pi_3(A).
\end{equation*}

\subsection{Upper half $3$-space with dyadic strips and trees} \label{subsec:3spacedyadic}

Let $X$ be the upper half $3$-space, together with the measure induced by the Lebesgue measure on $\R^3$,
\begin{equation} \label{eq:upper_3_space_dyadic_setting_weight}
\begin{split}
& X = \R^3_+ = \R^2_+ \times \R = \R \times (0,\infty) \times \R, \\
& \diff \omega(y,t,\eta) = \diff y \diff t \diff \eta.
\end{split}
\end{equation}
To define the outer measures, we start recalling the set $\mathcal{I}$ of dyadic intervals in $\R$,
\begin{equation*}
\begin{split}
& I(m,l) = ( 2^l m, 2^l (m+1) ], \\
& \mathcal{I} = \{ I(m,l) \colon m,l \in \Z \}.
\end{split}
\end{equation*}
Moreover, for every $m,l,n \in \Z$, we define the dyadic upper half tile $H(m,l,n)$ by
\begin{equation} \label{eq:dyadic_upper_half_tile}
H(m,l,n) = I(m,l) \times (2^{l-1}, 2^l] \times I(n,-l).
\end{equation}

Now, let $\mu$ be the outer measure generated by the pre-measure $\sigma$ on $\mathcal{D}$, the collection of dyadic strips, as in \eqref{eq:outer_measure_from_pre_measure}, namely
\begin{equation} \label{eq:upper_3_space_dyadic_setting_strip}
\begin{split}
& D(m,l) = D\big(I(m,l)\big) = \bigcup_{l' \leq l} \bigcup_{m' = 2^{l-l'} m}^{2^{l-l'} (m+1)-1} \bigcup_{n' \in \Z} H(m',l',n'), \\
& \mathcal{D} = \{ D(m,l) \colon m,l \in \Z \} = \{ D(I) \colon I \in \mathcal{I} \}, \\
& \sigma\big(D(m,l)\big) = \abs{I(m,l)} = 2^l, \qquad \qquad \text{for every $m,l \in \Z$.}
\end{split}
\end{equation}
Analogously, let $\nu$ be the outer measure generated by the pre-measure $\tau$ on $\mathcal{T}$, the collection of dyadic trees, as in \eqref{eq:outer_measure_from_pre_measure}, namely
\begin{equation} \label{eq:upper_3_space_dyadic_setting_tree}
\begin{split}
& T(m,l,n) = T\big(I(m,l), I(n,-l)\big) = \bigcup_{l' \leq l} \bigcup_{m' = 2^{l-l'} m}^{2^{l-l'} (m+1)-1} H\big(m',l',N(n,l')\big), \\
& \mathcal{T} = \{ T(m,l,n) \colon m,l,n \in \Z \} = \{ T(I,\widetilde{I}) \colon I,\widetilde{I} \in \mathcal{I}, \abs{I} \abs{\widetilde{I}} = 1 \}, \\
& \tau\big(T(m,l,n )\big) = \abs{I(m,l)} = 2^l, \qquad \qquad \text{for every $m,l,n \in \Z,$}
\end{split}
\end{equation}
where $N(n,l')$ is defined by the condition
\begin{equation} \label{eq:tree_condition}
I(n,-l) \subseteq I(N(n,l'),-l').
\end{equation} 
From now on, we assume all the strips and trees in this subsection to be dyadic, and we avoid repeating it. 

Next, for every $L \in \Z$, we define
\begin{equation} \label{eq:Y_l}
Y_L = \R \times (0,2^L] \times \R, 
\end{equation}
On $Y_L$, we have the measure $\omega_L$ and the outer measures $\mu_L, \nu_L$ induced by $\omega, \mu, \nu$. In particular, the outer measures $\mu_L, \nu_L$ are equivalently generated as in \eqref{eq:outer_measure_from_pre_measure} by the pre-measures $\sigma, \tau$ restricting the collections of dyadic strips and trees to those contained in $Y_L$, namely
\begin{align*}
& \mathcal{D}_L = \{ D(m,l) \colon m,l \in \Z, l \leq L  \}, \\
& \mathcal{T}_L = \{ T(m,l,n) \colon m,l,n \in \Z, l \leq L \}.
\end{align*}
Moreover, we drop the subscript $L$ in all the notation, as the definitions are consistent with the inclusion $Y_{L_1} \subseteq Y_{L_2}$ for $L_1 \leq L_2$.

To define the function $\mathcal{C}$ and check that it satisfies the conditions, we recall some properties of the geometry of dyadic strips and trees and introduce some auxiliary functions and state their properties. We postpone the proofs to Appendix \ref{sec:dyadic_geometry}. 

To make the notation more compact in the following definitions, we introduce a new symbol for the union of the elements of a collection of subsets of $X$,
\begin{gather*}
\mathcal{L} \colon \mathcal{P}(\mathcal{P}(X)) \to \mathcal{P}(X), \\
\mathcal{L}(\mathcal{A}) = \bigcup_{A \in \mathcal{A}} A.
\end{gather*}

We start with two observations about the geometry of the intersections between strips, and between a strip and a tree.
\begin{lemma} \label{thm:intersection_strip_strip}
	Given two strips $D_1,D_2$ in $\mathcal{D}$, their intersection is again a strip in $\mathcal{D}$, possibly empty. If it is nonempty, we have either $D_1 \subseteq D_2$ or $D_2 \subseteq D_1$.
\end{lemma}
\begin{lemma} \label{thm:intersection_strip_tree}
	Given a strip $D$ in $\mathcal{D}$ and a tree $T$ in $\mathcal{T}$, their intersection is again a tree $T'$ in $\mathcal{T}$, possibly empty.
\end{lemma}

After that, we follow up with some observations about the behaviour of the outer measures $\mu,\nu$ on strips, trees, their unions and their intersections.

\begin{lemma} \label{thm:measure_single}
For every strip $D$ in $\mathcal{D}$ and for every tree $T$ in $\mathcal{T}$, we have
\begin{gather} 
\label{eq:measure_strip}
\mu(D) = \sigma(D) = \abs{\pi(D)}, \\
\label{eq:measure_tree}
\nu(T) = \tau(T) = \abs{\pi(T)},
\end{gather}
where $\pi$ is the projection in the first coordinate.

Moreover, for every tree $T$ in $\mathcal{T}$, we have
\begin{equation} \label{eq:nu_as_mu}
\nu(T) =  \abs{\pi(T)} = \abs{\pi\big(D(T)\big)} = \mu\big(D(T)\big),
\end{equation}
where $D(T)$ is the strip in $\mathcal{D}$ containing $T$ defined by
\begin{equation*}
D(T) = \pi(T) \times (0, \abs{\pi(T)}] \times \R.
\end{equation*}
\end{lemma}

\begin{lemma} \label{thm:measure_disjoint_union}
For every collection $\mathcal{D}_1$ of pairwise disjoint strips in $\mathcal{D}$, we have
\begin{equation} \label{eq:additivity_mu}
\mu \big( \mathcal{L}(\mathcal{D}_1) \big) = \sum_{D_1 \in \mathcal{D}_1} \mu(D_1) = \sum_{D_1 \in \mathcal{D}_1} \abs{\pi(D_1)}.
\end{equation}
Analogously, for every collection $\mathcal{T}_1$ of pairwise disjoint trees in $\mathcal{T}$, we have
\begin{equation} \label{eq:additivity_nu}
\nu \big( \mathcal{L}(\mathcal{T}_1) \big) = \sum_{T_1 \in \mathcal{T}_1} \nu(T_1) = \sum_{T_1 \in \mathcal{T}_1} \abs{\pi(T_1)}.
\end{equation}
Moreover, for every collection $\mathcal{D}_1$ of pairwise disjoint strips in $\mathcal{D}$, for every tree $T$ in $\mathcal{T}$, we have
\begin{equation} \label{eq:additivity_nu_disjoint_strips}
\nu\big(T \cap \mathcal{L}(\mathcal{D}_1) \big) = \sum_{D_1 \in \mathcal{D}_1} \nu(T \cap D_1).
\end{equation}

\end{lemma}

Finally, we introduce the auxiliary functions. First, we define the function $\mathcal{Q}$ by
\begin{gather*}
\mathcal{Q} \colon \mathcal{P}(X) \to \mathcal{P}(\mathcal{D}), \\
\mathcal{Q}(A) = \{ E \colon E \in \mathcal{D}, E_{+} \cap A \neq \varnothing  \},
\end{gather*}
where $E_{+}$ is the upper half part of the strip $E$,
\begin{equation*}
E_{+} = \{ (x,s,\xi) \in E \colon s > \sigma(E)/2 \}.
\end{equation*}
It satisfies the following properties
\begin{gather} \label{eq:minimal_covering_covering}
A \subseteq \mathcal{L}(\mathcal{Q}(A)), 
\\
\label{eq:minimal_covering_monotone}
A_1 \subseteq A_2 \Rightarrow \mathcal{L}(\mathcal{Q}(A_1)) \subseteq \mathcal{L}(\mathcal{Q}(A_2)),
\\
\label{eq:minimal_covering_measure}
\mu(\mathcal{L}(\mathcal{Q}(A))) = \mu(A).
\end{gather}

After that, we define the function $\mathcal{N}$ by
\begin{gather*}
\mathcal{N} \colon \mathcal{P}(\mathcal{D}) \to \mathcal{P}(\mathcal{D}), \\
\mathcal{N}(\mathcal{D}_1) = \{ E \colon E \in \mathcal{D}, \abs{ \pi(E) \cap \pi (\mathcal{L}(\mathcal{D}_1)) } \geq \abs{\pi(E)}/2  \}.
\end{gather*}
It associates a collection of strips $\mathcal{D}_1$ to the collection of strips whose associated space interval is at least half covered by the space intervals associated with the elements of $\mathcal{D}_1$. It satisfies the following properties
\begin{gather} \label{eq:doubling_covering}
\mathcal{L}(\mathcal{D}_1) \subseteq \mathcal{L}(\mathcal{N}(\mathcal{D}_1)), \\
\label{eq:doubling_monotone}
\mathcal{L}(\mathcal{D}_1) \subseteq \mathcal{L}(\mathcal{D}_2) \Rightarrow \mathcal{L}(\mathcal{N}(\mathcal{D}_1)) \subseteq \mathcal{L}(\mathcal{N}(\mathcal{D}_2)), \\
\label{eq:doubling_measure}
\mu(\mathcal{L}(\mathcal{N}(\mathcal{D}_1))) \leq 2 \mu(\mathcal{L}(\mathcal{D}_1)).
\end{gather}

Finally, we define the function $\mathcal{M}$ by
\begin{gather*}
\mathcal{M} \colon \mathcal{P}(\mathcal{D}) \to \mathcal{P}(\mathcal{D}), \\
\mathcal{M}(\mathcal{D}_1) = \{ E \colon E \in \mathcal{D}_1, \forall D_1 \in \mathcal{D}_1 \setminus \{ E \} \textrm{ we have } E \not\subseteq D_1  \}.
\end{gather*}
It associates a collection of strips $\mathcal{D}_1$ to the subcollection of maximal elements with respect to inclusion. In particular, it is well-defined because, for every $L \in \Z$, the space $Y_L$ are bounded in the second variable. In fact, by Lemma \ref{thm:intersection_strip_strip}, the function $\mathcal{M}$ maps into the subset of collections of pairwise disjoint strips in $\mathcal{D}$. Moreover, it satisfies the following properties
\begin{gather} \label{eq:maximal_covering}
\mathcal{L}(\mathcal{D}_1) = \mathcal{L}(\mathcal{M}(\mathcal{D}_1)), \\
\label{eq:maximal_monotone}
\mathcal{L}(\mathcal{D}_1) \subseteq \mathcal{L}(\mathcal{D}_2) \Rightarrow \mathcal{L}(\mathcal{M}(\mathcal{D}_1)) \subseteq \mathcal{L}(\mathcal{M}(\mathcal{D}_2)), \\
\label{eq:maximal_measure}
\mu(\mathcal{L}(\mathcal{D}_1)) = \mu(\mathcal{L}(\mathcal{M}(\mathcal{D}_1))) = \sum_{E \in \mathcal{M}(\mathcal{D}_1)} \mu(E).
\end{gather}

We define the function $\mathcal{C} \colon \mathcal{P}(X) \to \mathcal{\dot{P}}(\mathcal{E})$ by
\begin{equation*}
\mathcal{E} = \mathcal{D}, \qquad \qquad \mathcal{C}(A) = \mathcal{M}(\mathcal{N}(\mathcal{Q}(A))),
\end{equation*}
where $\mathcal{\dot{P}}(\mathcal{E})$ stands for the set of subcollections of pairwise disjoint elements in $\mathcal{E}$.

We prove now that the function $\mathcal{C}$ is a $\mu$-covering function and that the setting $(X,\mu,\nu,\mathcal{C})$ satisfies the canopy condition \ref{def:parent_collection_nesting} and the crop condition \ref{def:anti_stacking_compatibility}.

\begin{lemma} \label{thm:covering_satisfied}
The function $\mathcal{C}$ is a $\mu$-covering function for every choice of the parameter $\Phi \geq 2$. 
\end{lemma}
\begin{proof}
We recall that
\begin{equation*}
\mathbf{B}_{\mathcal{C}}(A) = \mathcal{L}(\mathcal{M}(\mathcal{N}(\mathcal{Q}(A)))).
\end{equation*}
By \eqref{eq:minimal_covering_covering}, \eqref{eq:doubling_covering} and \eqref{eq:maximal_covering}, we have
\begin{equation*}
A \subseteq \mathbf{B}_{\mathcal{C}}(A).
\end{equation*}
By \eqref{eq:minimal_covering_monotone}, \eqref{eq:doubling_monotone} and \eqref{eq:maximal_monotone}, we have
\begin{equation*}
A_1 \subseteq A_2 \Rightarrow \mathbf{B}_{\mathcal{C}}(A_1) \subseteq \mathbf{B}_{\mathcal{C}}(A_2).
\end{equation*}
Moreover, by \eqref{eq:maximal_measure}, \eqref{eq:doubling_measure} and \eqref{eq:minimal_covering_measure}, we have
\begin{equation*}
\mu\big(\mathbf{B}_{\mathcal{C}}(A)\big) \leq 2 \mu(A).
\end{equation*}
\end{proof}

\begin{lemma} \label{thm:nesting_satisfied}
	The setting $(X,\mu,\nu,\mathcal{C})$ satisfies
	the canopy condition \ref{def:parent_collection_nesting} for every choice of parameters $\Phi, K \geq 2$.
\end{lemma}

\begin{proof}
Let $\mathcal{A}$ be a $\nu$-Carath\'{e}odory collection of subsets of $X$ with parameter $K$, and $\widetilde{D}$ a subset of $X$ disjoint from $\mathbf{B}_{\mathcal{C}}\big( \mathcal{L}(\mathcal{A})\big)$. We claim that the collection $\mathcal{A} \cup \{ \widetilde{D} \}$ is still $\nu$-Carath\'{e}odory with the same parameter $K$. In particular, we want to prove that for every subset $U$ of $X$, we have
\begin{equation} \label{eq:general_set_general_set}
\sum_{A \in \mathcal{A}} \nu(U \cap A) + \nu(U \cap \widetilde{D}) \leq K \nu(U).
\end{equation}
Without loss of generality, we assume $U \cap \widetilde{D} \neq \varnothing$, otherwise the inequality follows by the $\nu$-Carath\'{e}odory property for the collection $\mathcal{A}$. In particular, we have $\widetilde{D} \neq \varnothing$.

First, we prove \eqref{eq:general_set_general_set} under some additional assumptions on $\widetilde{D}$ and $U$. After that, we obtain the general case in a series of generalization steps.
 
{\textbf{Step 1.}} Let $\widetilde{D}$ be a nonempty set of the form 
\begin{equation}
D \setminus \mathbf{B}_{\mathcal{C}}\big(\mathcal{L}(\mathcal{A})\big),
\end{equation}
where $D$ is a strip in $\mathcal{D}$, and $ \mathbf{B}_{\mathcal{C}}\big(\mathcal{L}(\mathcal{A})\big) \subsetneq D$. We claim that, for every tree $T$ in $\mathcal{T}$, we have
\begin{equation} \label{eq:single_tree_single_strip}
\sum_{A \in \mathcal{A}} \nu(T \cap A) + \nu(T \cap D) \leq K \nu(T).
\end{equation}
The version of \eqref{eq:general_set_general_set} for the particular choices of $T$ and $\widetilde{D}$ follows by the monotonicity of $\nu$.

Without loss of generality, we assume $T$ to be contained in $D$. The result for an arbitrary tree $T$ follows by that for $T \cap D$, which by Lemma \ref{thm:intersection_strip_tree} is a tree as well, and the monotonicity of $\nu$. 

For every tree $T$ contained in $D$ with nonempty intersection with $\widetilde{D}$, we have 
\begin{equation*} 
D(T) \notin \mathcal{N}(\mathcal{Q}(\mathcal{L}(\mathcal{A}))).
\end{equation*}
Together with \eqref{eq:nu_as_mu}, this yields
\begin{equation*}
\nu(T) = \abs{\pi\big(D(T)\big)} \geq 2 \abs{\pi(D(T) \cap \mathcal{L}(\mathcal{Q}(\mathcal{L}(\mathcal{A}))))}.
\end{equation*}
By \eqref{eq:maximal_covering} and the disjointness of the elements of a collection $\mathcal{M}(\mathcal{D}_1)$ for every $\mathcal{D}_1 \subseteq \mathcal{D}$, we have
\begin{align*}
\abs{\pi(D(T) \cap \mathcal{L}(\mathcal{Q}(\mathcal{L}(\mathcal{A}))))} & = \abs{\pi(D(T) \cap \mathcal{L}(\mathcal{M}(\mathcal{Q}(\mathcal{L}(\mathcal{A})))))} \\
& = \sum_{E \in \mathcal{M}(\mathcal{Q}(\mathcal{L}(\mathcal{A})))} \abs{\pi(D(T) \cap E)}.
\end{align*}
By the monotonicity of the Lebesgue measure, Lemma \ref{thm:intersection_strip_tree}, and \eqref{eq:nu_as_mu}, we have
\begin{align*}
\sum_{E \in \mathcal{M}(\mathcal{Q}(\mathcal{L}(\mathcal{A})))} \abs{\pi(D(T) \cap E)} & \geq \sum_{E \in \mathcal{M}(\mathcal{Q}(\mathcal{L}(\mathcal{A})))} \abs{\pi(T \cap E)} \\
& \geq \sum_{E \in \mathcal{M}(\mathcal{Q}(\mathcal{L}(\mathcal{A})))} \nu( T \cap E).
\end{align*}
By \eqref{eq:additivity_nu_disjoint_strips} and the monotonicity of $\nu$, we have
\begin{equation*}
\sum_{E \in \mathcal{M}(\mathcal{Q}(\mathcal{L}(\mathcal{A})))} \nu( T \cap E) \geq \nu( T \cap \mathcal{L}(\mathcal{M}(\mathcal{Q}(\mathcal{L}(\mathcal{A}))))) \geq \nu(T \cap \mathcal{L}(\mathcal{A})).
\end{equation*}
Together with the condition $K \geq 2$ and the $\nu$-Carath\'{e}odory property for the collection $\mathcal{A}$, the previous chains of inequalities yield
\begin{align*}
K \nu(T) & \geq \nu(T \cap D) + 2(K-1) \nu(T \cap \mathcal{L}(\mathcal{A})) \\
& \geq \nu(T \cap D) + K \nu(T \cap \mathcal{L}(\mathcal{A})) \\
& \geq \nu(T \cap D) + \sum_{A \in \mathcal{A}} \nu(T \cap A ).
\end{align*}

{\textbf{Step 2.}} Let $\widetilde{D}$ be a nonempty set of the form 
\begin{equation*}
\widetilde{D} = \bigcup_{D' \in \mathcal{D}'} \widetilde{D}' = \bigcup_{D' \in \mathcal{D}'} \big( D' \setminus \mathbf{B}_{\mathcal{C}} ( \mathcal{L}(\mathcal{A})) \big),
\end{equation*}
where $\mathcal{D}'$ is a collection of pairwise disjoint strips. We claim that, for every tree $T$ in $\mathcal{T}$, we have \eqref{eq:general_set_general_set} for the particular choices of $T$ and $\widetilde{D}$. 

By definition, for every strip $\mathcal{D}'$, we have
\begin{equation*}
D' \not\subseteq \mathbf{B}_{\mathcal{C}}(\mathcal{L}(\mathcal{A})).
\end{equation*}
Therefore, by Lemma \ref{thm:intersection_strip_strip}, we have
\begin{equation*}
\mathcal{C}(\mathcal{L}(\mathcal{A})) = \mathcal{C}_1 \cup \bigcup_{D' \in \mathcal{D}'} \mathcal{C}_{D'},
\end{equation*}
where the elements of $\mathcal{C}_1$ are disjoint from $\mathcal{L}(\mathcal{D}')$, while, for every $D'$ in $\mathcal{D}'$, the elements of $\mathcal{C}_{D'}$ are contained in $D'$. In particular, we have
\begin{align*}
\mathcal{A} & = \mathcal{A}_1 \cup \bigcup_{D' \in \mathcal{D}'} \mathcal{A}_{D'} \\
& = \{ A \colon A \in \mathcal{A}, A \subseteq \mathcal{L}(\mathcal{C}_1) \} \cup \bigcup_{D' \in \mathcal{D}'} \{ A \colon A \in \mathcal{A}, A \subseteq \mathcal{L}(\mathcal{C}_{D'}) \}.
\end{align*}
Then
\begin{equation} \label{eq:single_tree_multiple_strip}
\begin{split}
K \nu(T) & \geq K \nu \big(T \cap (\mathcal{C}(\mathcal{L}(\mathcal{A})) \cup \bigcup_{D' \in \mathcal{D}'} D' ) \big) \\
& \geq K \nu \big(T \cap \mathcal{L}(\mathcal{C}_1) \big)  + K \sum_{D' \in \mathcal{D}'} \nu(T \cap D') \\
& \geq \sum_{A \in \mathcal{A}_1} \nu(T \cap A ) + \sum_{D' \in \mathcal{D}'} \big( \sum_{A \in \mathcal{A}_{D'}} \nu(T \cap A ) + \nu(T \cap D') \big) \\
& \geq \sum_{A \in \mathcal{A} } \nu(T \cap A ) + \nu(T \cap \mathcal{L}( \mathcal{D}')) \\
& \geq \sum_{A \in \mathcal{A} } \nu(T \cap A ) + \nu(T \cap \widetilde{D}).
\end{split}
\end{equation}
where we used the monotonicity of $\nu$ in the first and in the fifth inequality, \eqref{eq:additivity_nu_disjoint_strips} in the second, the $\nu$-Carath\'{e}odory property for the collection $\{ A \colon A \in \mathcal{A}, A \subseteq \mathcal{L}(\mathcal{C}_1)  \}$ and \eqref{eq:single_tree_single_strip} for each $D'$ in $\mathcal{D}'$ in the third, Fubini and \eqref{eq:additivity_nu_disjoint_strips} in the fourth.

{\textbf{Step 3.}} Let $\widetilde{D}$ be an arbitrary nonempty set disjoint from $\mathbf{B}_{\mathcal{C}}\big(\mathcal{L}(\mathcal{A})\big)$. We claim that, for every tree $T$ in $\mathcal{T}$, we have \eqref{eq:general_set_general_set} for the particular choices of $T$ and $\widetilde{D}$. 

For $\mathcal{D}' = \mathcal{M}(\mathcal{Q}(\widetilde{D}))$, we define
\begin{equation*}
\widetilde{D}_1 = \bigcup_{D' \in \mathcal{D}'} \big( D' \setminus \mathbf{B}_{\mathcal{C}} (\mathcal{L}(\mathcal{A}))\big).
\end{equation*}
By \eqref{eq:single_tree_multiple_strip} and the monotonicity of $\nu$, we have
\begin{equation} \label{eq:single_tree_general_set}
K \nu(T) \geq \sum_{A \in \mathcal{A}} \nu(T \cap A) + \nu(T \cap \widetilde{D}_1) \geq \sum_{A \in \mathcal{A}} \nu(T \cap A) + \nu(T \cap \widetilde{D}).
\end{equation} 

{\textbf{Step 4.}} Let $\widetilde{D}$ be an arbitrary nonempty set disjoint from $\mathbf{B}_{\mathcal{C}}\big(\mathcal{L}(\mathcal{A})\big)$. We claim that, for every subset $U$ of $X$, we have \eqref{eq:general_set_general_set}. 

In fact, there exists a collection $\mathcal{T}' \subseteq \mathcal{T}$ covering $U$ $\nu$-optimally, namely
\begin{gather} \label{eq:epsilon_covering_set}
U \subseteq \bigcup_{T \in \mathcal{T}'} T, \\ \label{eq:epsilon_covering_measure}
\sum_{T \in \mathcal{T}'} \tau(T) = \nu(U).
\end{gather}
By \eqref{eq:single_tree_general_set} for every tree $T$ in $\mathcal{T}'$, the subadditivity of $\nu$, and \eqref{eq:epsilon_covering_set}, we have
\begin{align*}
K \sum_{T \in \mathcal{T}'} \nu(T) & \geq \sum_{T \in \mathcal{T}'} \big( \sum_{A \in \mathcal{A}} \nu(T \cap A) + \nu( T \cap \widetilde{D} ) \big) \\
& \geq \sum_{A \in \mathcal{A}} \sum_{T \in \mathcal{T}'} \nu(T \cap A) + \sum_{T \in \mathcal{T}'} \nu( T \cap \widetilde{D} ) \\
& \geq \sum_{A \in \mathcal{A}} \nu(U \cap A) + \nu(U \cap \widetilde{D} ).
\end{align*}
Together with \eqref{eq:epsilon_covering_measure}, this yields the desired inequality in \eqref{eq:general_set_general_set}.
\end{proof}

\begin{lemma} \label{thm:anti_stacking_compatibility_satisfied}
	The setting $(X,\mu,\nu,\mathcal{C})$ satisfies
	the crop condition \ref{def:anti_stacking_compatibility} for every choice of parameters $\Phi \geq 2, K \geq 1$.
\end{lemma}

\begin{proof}
	For every collection $\mathcal{A}$ of strips in $\mathcal{D}$, let $\mathcal{B} = \mathcal{M}(\mathcal{A})$. The subcollection $\mathcal{B}$ is $\nu$-Carath\'{e}odory with parameter $K = 1$. Moreover, for every subset $F$ of $X$ disjoint from $\mathcal{L}(\mathcal{B}) = \mathcal{L}(\mathcal{A})$, we have
\begin{equation*}
\mathcal{C}(F) \cap \mathcal{A} = \mathcal{Q}(F) \cap \mathcal{A} = \varnothing,
\end{equation*}
and this yields
\begin{equation*}
\mathbf{B}_{\mathcal{C}} (F) = \mathbf{B}_{\widetilde{\mathcal{C}}} (F).
\end{equation*}
\end{proof}

\section{Double iterated outer $L^p$ spaces on the upper half $3$-space} \label{sec:dyadic_upper_space}

In this section we prove Theorem \ref{thm:collapsing_Holder_triangular_upper_3_space} in the dyadic upper half $3$-space setting described in \eqref{eq:upper_3_space_dyadic_setting_weight}, \eqref{eq:upper_3_space_dyadic_setting_strip} and \eqref{eq:upper_3_space_dyadic_setting_tree}, reducing the problem to an equivalent one in a finite setting via an approximation argument.

We start stating some auxiliary results about the approximation of functions in outer $L^p$ spaces. We use them to prove the approximation of functions in outer $L^p$ spaces on the upper half $3$-space $X$ by functions with support in $X_J$ for a certain $J \in \N$, where
\begin{equation} \label{eq:monotone_subset}
X_J = (-2^J J, 2^J J] \times (2^{-J}, 2^J] \times (-2^J J, 2^J J].
\end{equation}
On $X_J$, we have the measure $\omega_J$ and the outer measures $\mu_J, \nu_J$ induced by $\omega, \mu, \nu$. In particular, this setting inherits the definition of the function $\mathcal{C}$ on $Y_J$, for $Y_J$ defined in \eqref{eq:Y_l}, and its properties (Lemma \ref{thm:covering_satisfied}, Lemma \ref{thm:nesting_satisfied}, Lemma \ref{thm:anti_stacking_compatibility_satisfied}). 

Next, for any $J \in \N$, we introduce a finite setting $X'_J$ and exhibit a map between functions on $X_J$ and on $X'_J$ preserving the double iterated outer $L^p$ quasi-norms. We use Theorem \ref{thm:collapsing_2_step_iteration_finite}, Theorem \ref{thm:Holder_triangular_2_step_iteration_finite} in the finite settings to prove Theorem \ref{thm:collapsing_Holder_triangular_upper_3_space}. 

Finally, we conclude the section with some observations about the result analogous to Theorem \ref{thm:collapsing_Holder_triangular_upper_3_space} for double iterated outer $L^p$ spaces in the upper half $3$-space setting where the outer measures are defined by arbitrary strips and trees originally considered in \cite{2016arXiv161007657U}.

\subsection{Approximation results}

First, we state a result about the approximation of functions in $L^p_\mu(S)$ by functions in $L^p_\mu(S) \cap L^\infty_\mu(S)$, for a size $S$ of the form $\ell^r_\omega$ or $ \ell^q_\nu(\ell^r_\omega)$, and more generally an arbitrary size in the definition in \cite{MR3312633}.

\begin{lemma} \label{thm:approximation_max_size}
	For every $p \in (0, \infty)$, there exists a constant $C=C(p)$ such that the following property holds true.
	
	Let $X$ be a set, $\mu$ an outer measure, and $S$ a size. For every $f \in L^p_\mu(S)$, there exists a subset $A$ of $X$ such that $f 1_A$ is in $L^p_\mu(S) \cap L^\infty_\mu(S)$ and we have
	\begin{equation*}
	\norm{f}_{L^p_\mu(S)} \leq C \norm{f 1_{A}}_{L^p_\mu(S)}.
	\end{equation*}
\end{lemma}

Next, we state a result about the behaviour of the super level measures for single iterated outer $L^p$ spaces for monotonically increasing cut offs of a function in a general setting.
\begin{lemma} [Monotonic convergence I] \label{thm:monotone_approximation_size}
	For every $r \in (0, \infty)$, there exist constants $C=C(r)$, $c=c(r)$ such that the following property holds true. 
	
	Let $X$ be a set, $\nu$ an outer measure, and $\omega$ a measure. Let $\{ X_J \colon J \in \N \}$ be a monotonically increasing sequence of subsets of $X$ such that
	\begin{equation*}
	X = \bigcup_{J \in \N} X_J,
	\end{equation*} 
	and let $f \in L^\infty_\nu(\ell^r_\omega)$ be a function on $X$. Then, for every $k \in \Z$, there exists $J=J(r,f,k) \in \Z$ such that 
	\begin{equation*} 
	\nu(\ell^r_\omega(f) > 2^k) \leq C \sum_{l \geq k} \nu(\ell^r_\omega(f 1_{X_J}) > c 2^{k} ).
	\end{equation*}
\end{lemma}

Finally, we state a result about the behaviour of the super level measures for double iterated outer $L^p$ spaces for monotonically increasing cut offs of a function in the dyadic upper half $3$-space setting.
\begin{lemma} [Monotonic convergence II] \label{thm:monotone_approximation_size_exterior}
	For every $q,r \in (0, \infty)$, there exist constants $C=C(q,r)$, $c=c(q,r)$ such that the following property holds true. 
	
	Let $f \in L^\infty_\mu(\ell^q_\nu(\ell^r_\omega)) $ be a function on $X= \R \times (0,\infty) \times \R$, and let $\{ X_J \colon J \in \N \}$ be the monotonically increasing sequence of subsets of $X$ defined in \eqref{eq:monotone_subset}. Then, for every $k \in \Z$, there exists $J=J(q,r,f,k) \in \Z$ such that 
	\begin{equation*}
	\mu(\ell^q_\nu(\ell^r_\omega) (f) > 2^k) \leq C \sum_{l \geq k} \mu(\ell^q_\nu(\ell^r_\omega) (f 1_{X_J}) > c 2^{k} ).
	\end{equation*}
\end{lemma}

We postpone the proofs of the previous three results to Appendix \ref{sec:approximation}. We use them to prove the following results about the approximation of functions in $L^q_\nu(\ell^r_\omega)$ and $L^p_\mu(\ell^q_\nu(\ell^r_\omega))$ by functions with support in $X_j$ for a certain $j \in \N$.

\begin{lemma} \label{thm:approximation_tree}
	For every $q,r \in (0, \infty)$, there exists a constant $C=C(q,r)$ such that the following property holds true. 
	
	For every function $f \in L^q_\nu(\ell^r_\omega)$, there exists $J = J(q,r,f) \in \N$ such that
	\begin{equation*} 
	\norm{f 1_{X_J} }_{L^q_\nu(\ell^r_\omega)} \leq \norm{f}_{L^q_\nu(\ell^r_\omega)} \leq C \norm{f 1_{X_J}}_{L^q_\nu(\ell^r_\omega)}.
	\end{equation*}
\end{lemma}
\begin{proof}
	The first inequality follows by the monotonicity of the outer $L^p$ quasi-norms. 

To prove the second inequality, by Lemma \ref{thm:approximation_max_size}, we assume $f$ to be in $L^q_\nu(\ell^r_\omega) \cap L^\infty_\nu(\ell^r_\omega)$. Next, we observe that there exists $K = K(q,r,f) \in \N$ such that
\begin{equation*}
\norm{f}_{L^q_\nu(\ell^r_\omega)}^q \leq C \sum_{k \in \Z} 2^{kq} \nu( \ell^r_\omega (f) > 2^{k}) \leq C \sum_{k \in [-K,K]} 2^{kq} \nu(\ell^r_\omega (f) > 2^{k}).
\end{equation*}
By Lemma \ref{thm:monotone_approximation_size}, for every $k \in [-K,K]$, there exists a $\widetilde{J}=\widetilde{J}(r,f,k) \in \N$ such that
\begin{equation*}
\nu( \ell^r_\omega (f) > 2^{k}) \leq C \sum_{l \geq k} \nu( \ell^r_\omega (f 1_{X_{\widetilde{J}}}) > c 2^{l}).
\end{equation*}
By taking $J = \max_{k \in [-K,K]} \widetilde{J}(k,f,r)$, the previous inequalities yield
\begin{equation*}
\norm{f}_{L^q_\nu(\ell^r_\omega)}^q \leq C \sum_{k \in [-K,K]} 2^{kq} \sum_{l \geq k} \nu( \ell^r_\omega (f 1_{X_J}) > c 2^{l}) \leq C \norm{f 1_{X_J}}_{L^q_\nu(\ell^r_\omega)}^q.
\end{equation*}
\end{proof}

\begin{lemma} \label{thm:approximation}
	For every $p,q,r \in (0, \infty)$. There exists a constant $C=C(p,q,r)$ such that the following property holds true.
	
	For every function $f \in L^p_\mu(\ell^q_\nu(\ell^r_\omega))$, there exists $J = J(p,q,r,f) \in \N$ such that
\begin{equation*}
\norm{f 1_{X_J} }_{L^p_\mu(\ell^q_\nu(\ell^r_\omega))} \leq \norm{f}_{L^p_\mu(\ell^q_\nu(\ell^r_\omega))} \leq C \norm{f 1_{X_J}}_{L^p_\mu(\ell^q_\nu(\ell^r_\omega))}.
\end{equation*}
\end{lemma}
\begin{proof}
	The inequalities follow via the same argument used in the previous proof, with Lemma \ref{thm:monotone_approximation_size} replaced by Lemma \ref{thm:monotone_approximation_size_exterior}.
\end{proof}

\subsection{Equivalence with finite settings}
We introduce the following finite setting,
\begin{equation*}
\begin{split}
& X' = \Z^3, \\
& \omega'(m,l,n) = 1, \\
& D'(m,l) = \{ (m',l',n') \colon m' \in [2^{l-l'} m,2^{l-l'} (m+1)), l' \leq l, n' \in \Z \}, \\
& \mathcal{D}' = \{ D'(m,l) \colon m,l \in \Z \} , \\
& \sigma'\big(D(m,l)\big) = 2^l, \qquad \qquad \ \ \ \text{for every $m,l \in \Z$,} \\
& T'(m,l,n) = \{ (m',l',n') \colon m' \in [2^{l-l'} m,2^{l-l'} (m+1)), l' \leq l, n'=N(n,l') \}, \\
& \mathcal{T}' = \{ T'(m,l,n) \colon m,l,n \in \Z \}, \\
& \tau'\big(T'(m,l,n)\big) = 2^l, \qquad \qquad \text{for every $m,l,n \in \Z$,}
\end{split}
\end{equation*}
where $N(n,l')$ is defined by the condition \eqref{eq:tree_condition}, and $\mu',\nu'$ are defined by $\sigma',\tau'$ as in \eqref{eq:outer_measure_from_pre_measure}. Moreover, for every $J \in \N$, we define
\begin{equation*}
X'_J = \{ (m,l,n) \in X \colon l \in (-J,J], m \in [- J 2^{J-l}, J 2^{J-l}), n \in [- J 2^{J+l}, J 2^{J+l}) \},
\end{equation*}
On $X_J$, we have the measure $\omega'_J$ and the outer measures $\mu'_J, \nu'_J$ induced by $\omega', \mu', \nu'$. In fact, the outer measure $\mu'_J$ is equivalently generated by the pre-measure $\sigma'_J$ on $\mathcal{D}'_J$ as in \eqref{eq:outer_measure_from_pre_measure}, namely
\begin{align*}
& D'_J(m,l) = D'(m,l) \cap X'_J, \\
& \mathcal{D}'_J = \{ D'_J(m,l) \colon m,l \in \Z, D'_J(m,l) \neq \varnothing \}, \\
& \sigma_J\big(D'_J(m,l)\big) = 2^l, \qquad \qquad \text{for every $m,l \in \Z, D'_J(m,l) \neq \varnothing$,}
\end{align*}
and the outer measure $\nu'_J$ by the pre-measure $\tau'_J$ on $\mathcal{T}'_J$ as in \eqref{eq:outer_measure_from_pre_measure}, namely
\begin{align*}
& T'_J(m,l,n) = T'(m,l,n) \cap X'_J, \\
& \mathcal{T}'_J = \{ T'_J(m,l,n) \colon m,l,n \in \Z, T'_J(m,l,n) \neq \varnothing \}, \\
& \tau'_J\big(T'_J(m,l,n)\big) = 2^l, \qquad \qquad \text{for every $m,l,n \in \Z, T'_J(m,l,n) \neq \varnothing$.}
\end{align*}
The setting on $X'_J$ inherits the definition of the function $\mathcal{C}$ on $X_J$ and its properties (Lemma \ref{thm:covering_satisfied}, Lemma \ref{thm:nesting_satisfied}, Lemma \ref{thm:anti_stacking_compatibility_satisfied}) via the map associating every triple $(m,l,n) \in X'$ to $H(m,l,n)$, the pairwise disjoint subsets of $X$ defined in \eqref{eq:dyadic_upper_half_tile}. 

Moreover, every function $f$ on $X$ that is in $L^r_{\loc}(X,\omega)$ for some $r \in (0,\infty]$ defines a function $F(f,r)$ on $X'$ by
\begin{equation*}
F(f,r)(m,l,n) = \norm{f 1_{H(m,l,n)}}_{L^r(X,\omega)}.
\end{equation*}

For every fixed $r \in (0,\infty]$, the map between functions on $X$ and on $X'$ just described preserves the iterated outer $L^p$ quasi-norms.
\begin{lemma} \label{thm:starting}
	Let $p,q,r \in (0,\infty)$. For every $f$ supported in $X_J$ for any $J \in \N$, we have
	\begin{gather*}
	\norm{f}_{L^q_\nu(\ell^r_\omega)} = \norm{F(f,r)}_{L^q_{\nu'}(\ell^r_{\omega'})}, \\
	\norm{f}_{L^p_\mu(\ell^q_\nu(\ell^r_\omega))} = \norm{F(f,r)}_{L^p_{\mu'}(\ell^q_{\nu'}(\ell^r_{\omega'}))}.
	\end{gather*}
\end{lemma}
\begin{proof}
	Let $J \in \N$ be fixed, and assume that $f$ is supported in $X_J$. 
	
	We start observing that $F(f,r)$ is supported in $X'_J$. Moreover, in both cases, we can restrict to consider only the elements of $\mathcal{D}_J, \mathcal{T}_J$ and $\mathcal{D}'_J, \mathcal{T}'_J$, since we have
	\begin{gather*}
	\norm{f}_{L^p_\mu(\ell^q_\nu(\ell^r_\omega))} = \norm{f}_{L^p_{\mu_J}(\ell^q_{\nu_J}(\ell^r_{\omega_J}))}, \\
	\norm{F(f,r)}_{L^p_{\mu'}(\ell^q_{\nu'}(\ell^r_{\omega'}))} = \norm{F(f,r)}_{L^p_{\mu'_J}(\ell^q_{\nu'_J}(\ell^r_{\omega'_J}))}. 
	\end{gather*}
	
	In particular, for any $U \in \mathcal{T}_J$, we have $U= T_J(m,l,n)$, and we define $U' \in \mathcal{T}'_J$ by $U' = T'_J(m,l,n)$, hence satisfying
	\begin{equation} \label{eq:same_measure}
	\nu_J(U) = \tau_J(U) = \tau'_J(U') = \nu'_J(U'). 
	\end{equation}
	
	Now, for any two collections $\mathcal{U}_1, \mathcal{U}_2$ of elements in $\mathcal{T}_J$, we define, for $i = 1,2$,
	\begin{equation*}
	U_i = \mathcal{L}(\mathcal{U}_i), U'_i = \mathcal{L}(\mathcal{U}'_i),
	\end{equation*}
	and we have
	\begin{equation} \label{eq:commutativity}
	F(f 1_{U_1 \setminus U_2}, r) =	F(f,r) 1_{U'_1 \setminus U'_2}.
	\end{equation}
	
	Next, by the definition of $F(f,r)$, we have
	\begin{equation} \label{eq:same_norm_base}
	\norm{f}_{L^r(X_J, \omega_J)} = \norm{F(f,r)}_{L^r(X'_J, \omega'_J)}.
	\end{equation}
	
	Therefore, for any element $U$ in $\mathcal{T}_J$, we have
	\begin{equation} \label{eq:same_norm}
	\norm{f 1_{U}}_{L^r(X_J, \omega_J)} = \norm{F(f 1_{U},r)}_{L^r(X'_J,\omega'_J)} = \norm{F(f,r) 1_{U'}}_{L^r(X'_J, \omega'_J)},
	\end{equation}
	where we used \eqref{eq:same_norm_base} in the first equality, and \eqref{eq:commutativity} in the second. Moreover, for any $A \subseteq X_J$, there exists a finite subcollection $\mathcal{U}$ of $\mathcal{T}_J$ such that $A \subseteq \mathcal{L}(\mathcal{U})$ and
	\begin{equation} \label{eq:measure_A}
	\nu_J(A) = \sum_{U \in \mathcal{U}} \tau_J(U) = \sum_{U \in \mathcal{U}} \nu_J(U).
	\end{equation}
	In particular, we have
	\begin{equation} \label{eq:size_A}
	\begin{split}
	\nu_J(A)^{-1} \norm{f 1_{A}}^r_{L^r(X_J, \omega_J)}	& \leq \nu_J(A)^{-1} \sum_{U \in \mathcal{U}} \norm{f 1_U}^r_{L^r(X_J, \omega_J)} \\
	& \leq \nu_J(A)^{-1} \max_{V \in \mathcal{U}} \nu_J(V)^{-1} \norm{f 1_V}^r_{L^r(X_J, \omega_J)} \sum_{U \in \mathcal{U}} \nu_J(U) \\
	& \leq \max_{V \in \mathcal{U}} \nu_J(V)^{-1} \norm{f 1_V}^r_{L^r(X_J, \omega_J)},
	\end{split}
	\end{equation}
	where we used the monotonicity of the $L^r$ quasi-norm in the first inequality, H\"older's inequality in the third, and \eqref{eq:measure_A} in the fourth. The analogous properties hold true for any $F$ supported in $X'_J$.
	
	Therefore, for any $\lambda > 0$, we have, for $F=F(f,r)$,
	\begin{align*}
	& \nu_J (\ell^r_{\omega_J}(f) > \lambda) = \\
	& = \inf \{ \nu_J(A) \colon A \subseteq X_J, \sup \{ \nu_J(B)^{-1/r} \norm{f 1_{B} 1_{A^c}}_{L^r(X_J, \omega_J)} \colon B \subseteq X_J \} \leq \lambda \} \\
	& = \inf \{ \nu_J(\mathcal{L}(\mathcal{U})) \colon \mathcal{U} \subseteq \mathcal{T}_J, \sup \{ \nu_J(V)^{-1/r} \norm{f 1_{V} 1_{\mathcal{L}(\mathcal{U})^c}}_{L^r(X_J, \omega_J)} \colon V \in \mathcal{T}_J \} \leq \lambda \} \\
	& = \inf \{ \nu'_J(\mathcal{L}(\mathcal{U}')) \colon \mathcal{U}' \subseteq \mathcal{T}'_J, \sup \{ \nu'_J(V')^{-1/r} \norm{F 1_{V'} 1_{\mathcal{L}(\mathcal{U}')^c}}_{L^r(X'_J, \omega'_J)} \colon V' \in \mathcal{T}'_J \} \leq \lambda \} \\
	& = \inf \{ \nu'_J(A') \colon A' \subseteq X'_J, \sup \{ \nu'_J(B')^{-1/r} \norm{F 1_{B'} 1_{(A')^c}}_{L^r(X'_J, \omega'_J)} \colon B' \subseteq X'_J \} \leq \lambda \} \\
	& = \nu'_J(\ell^r_{\omega'_J}(F) > \lambda),
	\end{align*}
	where we used \eqref{eq:measure_A} and \eqref{eq:size_A} in the second equality, \eqref{eq:same_measure} and \eqref{eq:same_norm} in the third, the analogous of \eqref{eq:measure_A} and \eqref{eq:size_A} in the fourth. Hence
	\begin{equation*}
	\norm{f}_{L^q_{\nu_J}(\ell^r_{\omega_J})} = \norm{F(f,r)}_{L^q_{\nu'_J}(\ell^r_{\omega'_J})}.
	\end{equation*} 
	
	Applying an analogous argument to the "exterior" level of definition of the double iterated outer $L^p$ space, we obtain	
	\begin{equation*}
	\norm{f}_{L^p_{\mu_J}(\ell^q_{\nu_J}(\ell^r_{\omega_J}))} = \norm{F(f,r)}_{L^p_{\mu'_J}(\ell^q_{\nu'_J}(\ell^r_{\omega'_J}))}.
	\end{equation*} 
\end{proof}

We are now ready to prove Theorem \ref{thm:collapsing_Holder_triangular_upper_3_space}.

\begin{proof} [Proof of Theorem \ref{thm:collapsing_Holder_triangular_upper_3_space}]
	
	Let $p,q,r \in (0, \infty]$. By Lemma \ref{thm:approximation_tree} and Lemma \ref{thm:approximation}, for every $f \in L^{p}_\mu(\ell^{q}_\nu(\ell^{r}_\omega))$, there exists $J = J(f,p,q,r) \in \N$ such that
	\begin{equation} \label{eq:first}
	\begin{gathered}
	\norm{f 1_{X_J}}_{L^{q}_\nu(\ell^{r}_\omega)} \leq \norm{f}_{L^{q}_\nu(\ell^{r}_\omega)} \leq C \norm{f 1_{X_J}}_{L^{q}_\nu(\ell^{r}_\omega)}, \\
	\norm{f 1_{X_J}}_{L^{p}_\mu(\ell^{q}_\nu(\ell^{r}_\omega))} \leq \norm{f}_{L^{p}_\mu(\ell^{q}_\nu(\ell^{r}_\omega))} \leq C \norm{f 1_{X_J}}_{L^{p}_\mu(\ell^{q}_\nu(\ell^{r}_\omega))},
	\end{gathered}
	\end{equation}
	where $C$ is independent of $f$ and $J$. By Lemma \ref{thm:starting}, we have
	\begin{equation} \label{eq:second}
	\begin{gathered}
	\norm{f 1_{X_J}}_{L^{q}_\nu(\ell^{r}_\omega)} =  \norm{F(f 1_{X_J},r)}_{L^{q}_{\nu'}(\ell^{r}_{\omega'})}
	=  \norm{F(f,r) 1_{X'_J}}_{L^{q}_{\nu'_J}(\ell^{r}_{\omega'_J})}, \\
	\norm{f 1_{X_J}}_{L^{p}_\mu(\ell^{q}_\nu(\ell^{r}_\omega))} = \norm{F(f 1_{X_J},r)}_{L^{p}_{\mu'}(\ell^{q}_{\nu'}(\ell^{r}_{\omega'}))} 
	= \norm{F(f,r) 1_{X'_J}}_{L^{p}_{\mu'_J}(\ell^{q}_{\nu'_J}(\ell^{r}_{\omega'_J}))}.
	\end{gathered}
	\end{equation}
	
	{\textbf{Property (i).}} Let $q,r \in (0, \infty)$. By Theorem \ref{thm:collapsing_2_step_iteration_finite}, we have
	\begin{equation*}
	C^{-1} \norm{F(f,r) 1_{X'_J}}_{L^{q}_{\nu'_J}(\ell^{r}_{\omega'_J})} \leq \norm{F(f,r) 1_{X'_J}}_{L^{q}_{\mu'_J}(\ell^{q}_{\nu'_J}(\ell^{r}_{\omega'_J}))} \leq C \norm{F(f,r) 1_{X'_J}}_{L^{q}_{\nu'_J}(\ell^{r}_{\omega'_J})},
	\end{equation*}
	where $C$ is independent of $f$ and $J$. Together with \eqref{eq:first} and \eqref{eq:second}, the previous chain of inequalities yields the desired equivalence in \eqref{eq:collapsing_2_iterated_upper_3_space}.
	
	{\textbf{Property (ii).}} Let $p,q,r \in (1,\infty)$. By Theorem \ref{thm:Holder_triangular_2_step_iteration_finite}, for every $f \in L^p_\mu(\ell^q_\nu(\ell^r_\omega))$, there exists a function $G$ on $X'_J$ with unitary outer $L^{p'}_{\mu'_J}(\ell^{q'}_{\nu'_J}(\ell^{r'}_{\omega'_J}))$ quasi-norm such that
	\begin{equation} \label{eq:dual_realization}
	\begin{split}
	C^{-1} \norm{F(f,r) 1_{X'_J}}_{L^{p}_{\mu'_J}(\ell^{p}_{\nu'_J}(\ell^{r}_{\omega'_J}))} & \leq \norm{F(f,r) 1_{X'_J} G}_{L^1(X'_J, \omega'_J)} \\
	& \leq C \norm{F(f,r) 1_{X'_J}}_{L^{p}_{\mu'_J}(\ell^{p}_{\nu'_J}(\ell^{r}_{\omega'_J}))},
	\end{split}
	\end{equation}
	where $C$ is independent of $f$ and $J$. We define a function $g$ on $X$ by
	\begin{equation*}
	g(x,s,\xi) = \abs{f(x,s,\xi)}^{r-1} \sum_{m,l,n \in \Z} F(m,l,n)^{1-r} G(m,l,n) 1_{H(m,l,n)}(x,s,\xi).
	\end{equation*}
	By construction, we have
	\begin{equation*}
	F(g,r') = G.
	\end{equation*}
	Together with Lemma \eqref{thm:starting}, this yields 
	\begin{equation*} 
	\norm{g}_{L^{p'}_{\mu}(\ell^{q'}_{\nu}(\ell^{r'}_{\omega}))} = \norm{G}_{L^{p'}_{\mu'}(\ell^{q'}_{\nu'}(\ell^{r'}_{\omega'}))} = \norm{G}_{L^{p'}_{\mu'_J}(\ell^{q'}_{\nu'_J}(\ell^{r'}_{\omega'_J}))} = 1.
	\end{equation*}
	Moreover, by construction we have
	\begin{equation*} 
	\norm{fg}_{L^{1}_\omega} = \norm{F(f,r) G}_{L^1(X'_J, \omega'_J)} = \norm{F(f,r) G}_{L^1(X'_J, \omega'_J)} = \norm{F(f,r) 1_{X'_J} G}_{L^1(X'_J, \omega'_J)}.
	\end{equation*}
	Together with \eqref{eq:first}, \eqref{eq:second}, and \eqref{eq:dual_realization}, the last two chains of equalities yield the desired equivalence in \eqref{eq:Holder_2_iterated_upper_3_space}.
	
	{\textbf{Property (iii).}} The inequality in \eqref{eq:triangle_2_iterated_upper_3_space} is a corollary of the triangle inequality for the $L^1(X,\omega)$ norm and property $(ii)$.
\end{proof}

\subsection{Upper half 3-space with arbitrary strips and trees} \label{subsec:3spacearbitrary}

We turn to the case of double iterated outer $L^p$ spaces on the upper half $3$-space setting where the outer measures are defined by arbitrary strips and trees. In particular, let
\begin{equation} \label{eq:upper_3_space_arbitrary_setting}
\begin{split}
& X = \R^3_+ = \R^2_+ \times \R = \R \times (0,\infty) \times \R, \\
& \diff \omega(y,t,\eta) = \diff y \diff t \diff \eta, \\
& \widetilde{D} (x,s) = \{ (y,t,\eta) \colon y \in (x,x+s], t \in (0,s], \eta \in \R \}, \\
& \mathcal{\widetilde{D}} = \{ \widetilde{D}(x,s) \colon x \in \R , s \in (0,\infty) \} , \\
& \widetilde{\sigma}\big(\widetilde{D}(x,s)\big) = s, \qquad \qquad \ \ \text{for every $x \in \R , s \in (0,\infty)$,} \\
& \widetilde{T}(x,s,\xi) = \{ (y,t,\eta) \colon y \in (x,x+s], t \in (0,s], \eta \in ( \xi - t^{-1}, \xi + t^{-1} ] \}, \\
& \mathcal{\widetilde{T}} = \{ \widetilde{T}(x,s,\xi) \colon x \in \R, s \in (0,\infty), \xi \in \R \}, \\
& \widetilde{\tau}\big(\widetilde{T}(x,s,\xi )\big) = s, \qquad \qquad \text{for every $x \in \R , s \in (0,\infty), \xi \in \R$,}
\end{split}
\end{equation}
where $\widetilde{\mu},\widetilde{\nu}$ are defined by $\widetilde{\sigma},\widetilde{\tau}$ as in \eqref{eq:outer_measure_from_pre_measure}.

On one hand, the outer measures generated by dyadic strips and arbitrary ones are equivalent and we can substitute the outer measure $\widetilde{\mu}$ with $\mu$. In particular, we have $\mathcal{D} \subseteq \mathcal{\widetilde{D}}$, and every element of $\mathcal{\widetilde{D}}$ is covered by at most two elements of $\mathcal{D}$ with comparable pre-measure.

On the other hand, the outer measures generated by dyadic trees and arbitrary ones are not equivalent. In fact, while for every dyadic tree $T$ in $\mathcal{T}$ we have
\begin{equation*}
\widetilde{\nu}(T) \leq \nu(T),
\end{equation*}
instead for every arbitrary tree $\widetilde{T}$ in $\mathcal{\widetilde{T}}$ we have
\begin{equation} \label{eq:arbitrary_dyadic}
\nu(\widetilde{T}) = \infty,
\end{equation}
and we postpone the proof to Appendix \ref{sec:dyadic_geometry}. Therefore, we can not trivially deduce the same result stated in Theorem \ref{thm:collapsing_Holder_triangular_upper_3_space} in the setting described in \eqref{eq:upper_3_space_arbitrary_setting} from Theorem \ref{thm:collapsing_Holder_triangular_upper_3_space} itself.

However, a reduction of the problem to an equivalent one in a finite setting via an approximation argument analogous to that described in the previous subsections still yields the desired result. We briefly comment on some additional observations, providing guidance to the readers interested in a complete proof.

First, we observe that the outer measure $\widetilde{\nu}$ is equivalent to $\widetilde{\nu}_d$, the outer measure defined as in \eqref{eq:outer_measure_from_pre_measure} by the pre-measure $\widetilde{\tau}$ restricting the collection $\mathcal{\widetilde{T}}$ of trees to those associated with dyadic intervals, namely
\begin{equation*}
\mathcal{\widetilde{T}}_d = \{ \widetilde{T}(2^l m ,2^l, 2^{-l} n) \colon m,l,n \in \Z \} \subseteq \mathcal{\widetilde{T}}.
\end{equation*}
The geometry of the elements of $\mathcal{D}, \mathcal{\widetilde{T}}_d$ and their intersections is analogous to that of the elements of $\mathcal{D}, \mathcal{T}$. Therefore, for every function $f$ in a double iterated outer $L^p$ space in the setting $(X,\mu,\widetilde{\nu}_d,\omega)$, we can pass to a cut off $f 1_{X_J}$ approximating the double iterated outer $L^p$ quasi-norm of $f$, for $X_J$ defined in \eqref{eq:monotone_subset}.

Next, for every fixed $J \in \N$, we consider the outer measure $\widetilde{\nu}_{d,J}$ induced on $Y_J$ by $\widetilde{\nu}_d$, where $Y_J$ is defined in \eqref{eq:Y_l}. We observe that $\widetilde{\nu}_{d,J}$ is equivalent to the outer measure generated as in \eqref{eq:outer_measure_from_pre_measure} by the pre-measure $ \widetilde{\tau}$ restricting the collection $\widetilde{\mathcal{T}}_d$ of trees to those contained in $Y_J$, namely
\begin{equation*}
\mathcal{\widetilde{T}}_{d,J} = \{ \widetilde{T}(2^l m ,2^l, 2^{-l} n) \colon m,l,n \in \Z, l \leq J \} \subseteq \widetilde{\mathcal{T}}_d.
\end{equation*}
In the setting $(Y_J,\mu_J,\widetilde{T}_{d,J}, \omega_J)$, we can state definitions and prove results based on the geometry of the elements of $\mathcal{D}_{J}, \mathcal{\widetilde{T}}_{d,J}$ analogous to those in Section \ref{sec:examples}. Therefore, for every $J \in \N$, we can define a $\mu_J$-covering function $\mathcal{\widetilde{C}}$ satisfying the canopy condition \ref{def:parent_collection_nesting} and the crop condition \ref{def:anti_stacking_compatibility}. In particular, this definition is inherited by $X_J \subseteq Y_J$.

After that, for every fixed $J \in \N$, we observe that the elements of $\mathcal{D}_{J}, \mathcal{\widetilde{T}}_{d,J}$ with nonempty intersection with $X_J$ are finitely many. Therefore, we can introduce a finite setting with a point for every intersection and the induced measure and outer measures. In particular, we conclude the result corresponding to that stated in Theorem \ref{thm:collapsing_Holder_triangular_upper_3_space} via an argument analogous to that of the previous subsection. 

\appendix
\section{Geometry of the dyadic upper half $3$-space setting} \label{sec:dyadic_geometry}

In this appendix, we exhibit the postponed proofs of the results involving the geometry of the dyadic strips and trees in the upper half $3$-space stated in Section \ref{sec:examples}, and in \eqref{eq:arbitrary_dyadic} in Section \ref{sec:dyadic_upper_space}.

We start recalling that every dyadic strip $D$ in $\mathcal{D}$ is determined by a dyadic interval $I_D$ in $\mathcal{I}$, and has the form
\begin{equation} \label{eq:strip}
D = I_D \times (0,\abs{I_D}] \times \R = \pi(D) \times (0,\abs{\pi(D)}] \times \R,
\end{equation}
and every dyadic tree $T$ in $\mathcal{T}$ is determined by two dyadic intervals $I_T,\widetilde{I}_T$ in $\mathcal{I}$ such that $\abs{I_T} \abs{\widetilde{I}_T} = 1$ and has the form
\begin{equation} \label{eq:tree}
T = \bigcup_{J \in \mathcal{I}, J \subseteq I_T} J \times (0, \abs{J} ] \times \widetilde{J}(T,J) = \bigcup_{J \in \mathcal{I}, J \subseteq \pi(T)} J \times (0, \abs{J} ] \times \widetilde{J}(T,J),
\end{equation}
where the dyadic interval $\widetilde{J}(T,J)$ in $\mathcal{I}$ is defined by the conditions
\begin{gather*}
\abs{\widetilde{J}(T,J)} = \abs{J}^{-1}, \\
\widetilde{I}_T = \widetilde{J}\big(T,\pi(T)\big) \subseteq \widetilde{J}(T,J).
\end{gather*}

\begin{proof}[Proof of Lemma \ref{thm:intersection_strip_strip}]
	If $D_1 \cap D_2$ is empty, the statement is trivially verified. Therefore, we assume that the strips $D_1,D_2$ have a nonempty intersection. Hence the dyadic intervals $\pi(D_1),\pi(D_2)$ have a nonempty intersection as well. Therefore, we have either $\pi(D_1) \subseteq \pi(D_2)$ or $\pi(D_2) \subseteq \pi(D_1)$. Without loss of generality, we can restrict to the first case, the second being analogous. We have $\abs{\pi(D_1)} \leq \abs{\pi(D_2)}$, hence by \eqref{eq:strip}
	\begin{equation*}
	D_1 \subseteq D_2.
	\end{equation*}
\end{proof}

\begin{proof}[Proof of Lemma \ref{thm:intersection_strip_tree}]
	If $D \cap T$ is empty, the statement is trivially verified. Therefore, we assume that the strip $D$ and the tree $T$ have a nonempty intersection. Hence the dyadic intervals $\pi(D),\pi(T)$ have a nonempty intersection as well. Therefore, we have either $\pi(D) \subseteq \pi(T)$ or $\pi(T) \subseteq \pi(D)$. In the first case, we have $\abs{\pi(D)} \leq \abs{\pi(T)}$, hence by \eqref{eq:strip} and \eqref{eq:tree}
	\begin{equation*}
	D \cap T = T\Big(\pi(D),\widetilde{J}\big(T,\pi(D)\big)\Big).
	\end{equation*}
	In the second case, we have $\abs{\pi(T)} \leq \abs{\pi(D)}$, hence by \eqref{eq:strip} and \eqref{eq:tree}
	\begin{equation*}
	D \cap T = T.
	\end{equation*}
\end{proof}

\begin{proof}[Proof of Lemma \ref{thm:measure_single}]
	Let $D$ be a strip in $\mathcal{D}$. Then
	\begin{equation*}
	\mu(D) = \inf \{ \sum_{D_1 \in \mathcal{D}_1} \sigma(D_1) \colon \mathcal{D}_1 \subseteq \mathcal{D}, D \subseteq \mathcal{L}(\mathcal{D}_1) \}.
	\end{equation*}
	Therefore, the inequality
	\begin{equation*}
	\mu(D) \leq \sigma(D) 
	\end{equation*}
	follows trivially. To prove the opposite inequality, we observe that for every covering $\mathcal{D}_1$ of $D$ by means of strips in $\mathcal{D}$, there exists a strip $E$ in $\mathcal{D}_1$ such that
	\begin{equation*}
	(x_D, \abs{\pi(D)}, 0) \in E,
	\end{equation*} 
	where $x_D$ is the middle point of the dyadic interval $\pi(D)$. In particular, this implies 
	\begin{equation*}
	\sigma(E) \geq \abs{\pi(D)}.
	\end{equation*} 
	Therefore, we have
	\begin{equation*}
	\sum_{D_1 \in \mathcal{D}_1} \sigma(D_1) \geq \sigma(D),
	\end{equation*}
	By taking the infimum among all the possible coverings of $D$, we obtain the desired equality in \eqref{eq:measure_strip}.
	
	The statement for a tree $T$ in $\mathcal{T}$ in \eqref{eq:measure_tree} follows by an analogous argument considering the point
	\begin{equation*}
	(x_T, \abs{\pi(T)}, \xi_T),
	\end{equation*} 
	where $x_T$ is the middle point of the dyadic interval $\pi(T)$, and $\xi_T$ is the middle point of the dyadic interval $\widetilde{J}\big(T, \pi(T)\big)$.
	
	The statement in \eqref{eq:nu_as_mu} follows by the definition of $D(T)$,  \eqref{eq:measure_strip}, and \eqref{eq:measure_tree}.
\end{proof}

\begin{proof}[Proof of Lemma \ref{thm:measure_disjoint_union}]
	Let $\mathcal{D}_1$ be a collection of pairwise disjoint strips in $\mathcal{D}$. The inequality
	\begin{equation*}
	\mu(\mathcal{L}(\mathcal{D}_1)) \leq \sum_{D_1 \in \mathcal{D}_1} \mu(D_1),
	\end{equation*}
	follows by the subadditivity of $\mu$. To prove the opposite inequality, we consider a covering $\mathcal{D}_2$ of $\mathcal{L}(\mathcal{D}_1)$. Without loss of generality, we assume that every $E$ in $\mathcal{D}_2$ is not strictly contained in any element of $D_1$, otherwise it would be useless to the purpose of covering. Therefore, we have $E \nsubset \mathcal{L}(\mathcal{D}_1)$, and, by Lemma \ref{thm:intersection_strip_strip}, we have
	\begin{equation*}
	\mathcal{D}_1 = \mathcal{D}_{1,E} \cup \mathcal{\widetilde{D}}_1,
	\end{equation*} 
	where every element of $\mathcal{D}_{1,E}$ is contained in $E$, and every element of the other collection is disjoint from $E$. In particular,
	\begin{equation} \label{eq:containment}
	\mathcal{L}(\mathcal{D}_{1,E}) \subseteq E.
	\end{equation}
	As a consequence, we have
	\begin{equation*}
	\sigma(E) = \abs{\pi(E)} \geq \abs{\pi\big(\mathcal{L}(\mathcal{D}_{1,E})\big)} = \sum_{D_1 \in \mathcal{D}_{1,E}} \abs{\pi(D_1)} = \sum_{D_1 \in \mathcal{D}_{1,E}} \mu(D_1),
	\end{equation*}
	where we used \eqref{eq:measure_strip} in the first and in the third equality, \eqref{eq:containment} and the monotonicity of $\pi$ and the Lebesgue measure in the inequality, the distributivity of the projection over set union and the additivity of the Lebesgue measure on the disjoint intervals in $\pi(\mathcal{D}_1)$ in the second equality. Together with the observation that for every element $D_1$ of $\mathcal{D}_1$ there exists at least one $E$ in $ \mathcal{D}_2$ such that $D_1 \in \mathcal{D}_{1,E}$, we obtain 
	\begin{equation*}
	\sum_{E \in \mathcal{D}_2} \sigma(E) \geq \sum_{E \in \mathcal{D}_2} \sum_{D_1 \in \mathcal{D}_{1,E}} \mu(D_1) \geq \sum_{D_1 \in \mathcal{D}_1} \mu(D_1).
	\end{equation*}
	By taking the infimum among all the possible coverings of $\mathcal{L}(\mathcal{D}_1)$, we obtain the desired equality in \eqref{eq:additivity_mu}.
	
	The statement for a collection $\mathcal{T}_1$ of pairwise disjoint trees in \eqref{eq:additivity_nu} follows by an analogous argument. The additional observation is that the collection of trees $\mathcal{T}$ splits into two families
	\begin{equation*}
	\mathcal{T} = \mathcal{T}_{+} \cup \mathcal{T}_{-},
	\end{equation*}
	where the elements of $\mathcal{T}_{+}$ are all contained in $\R \times (0,\infty) \times (0,\infty)$, while the elements of $\mathcal{T}_{-}$ are all contained in $\R \times (0,\infty) \times (-\infty, 0]$. In particular, every element of the first family is disjoint from every element of the second one.
	
	The statement in \eqref{eq:additivity_nu_disjoint_strips} follows by Lemma \ref{thm:intersection_strip_tree} and \eqref{eq:additivity_nu}.
\end{proof}

\begin{proof}[Proof of \eqref{eq:minimal_covering_covering}, \eqref{eq:minimal_covering_monotone}]
	Let $A$ be a subset of $X$. For every point $(x,s,\xi)$ in $A$, there exist $l \in \Z$ such that $s \in (2^{l-1},2^l]$, and $m \in \Z$ such that $x \in I(m,l) $. Hence, we have
	\begin{equation*}
	(x,s,\xi) \in D(m,l)_{+},
	\end{equation*}
	proving \eqref{eq:minimal_covering_covering}.
	
	Next, let $A_1,A_2$ be two subsets of $X$ such that $A_1 \subseteq A_2$. By the definition of $\mathcal{Q}$, we have $\mathcal{Q}(A_1) \subseteq \mathcal{Q}(A_2)$. Taking the union of the elements of the collection in both cases, we obtained the desired inclusion, proving \eqref{eq:minimal_covering_monotone}.
\end{proof}

\begin{proof}[Proof of \eqref{eq:doubling_covering}, \eqref{eq:doubling_monotone}]
	Let $\mathcal{D}_1$ be a collection of strips. By the definition of $\mathcal{N}$, we have $\mathcal{D}_1 \subseteq \mathcal{N}(\mathcal{D}_1)$. Taking the union of the elements of the collection in both cases, we obtained the desired inclusion, proving \eqref{eq:doubling_covering}.
	
	Next, let $\mathcal{D}_1,\mathcal{D}_2$ be two collections of strips such that $\mathcal{L}(\mathcal{D}_1) \subseteq \mathcal{L}(\mathcal{D}_2)$. In particular, $\pi(\mathcal{L}(\mathcal{D}_1)) \subseteq \pi(\mathcal{L}(\mathcal{D}_2))$. By the definition of $\mathcal{N}$, we have $\mathcal{N}(\mathcal{D}_1) \subseteq \mathcal{N}(\mathcal{D}_1)$. Taking the union of the elements of the collection in both cases, we obtained the desired inclusion, proving \eqref{eq:doubling_monotone}.
\end{proof}

\begin{proof}[Proof of \eqref{eq:maximal_covering}, \eqref{eq:maximal_monotone}]
	Let $\mathcal{D}_1$ be a collection of strips. Since $\mathcal{M}(\mathcal{D}_1) \subseteq \mathcal{D}_1$, we have the inclusion $\mathcal{L}(\mathcal{M}(\mathcal{D}_1)) \subseteq \mathcal{L}(\mathcal{D}_1)$.
	
	To prove the inclusion in the opposite direction, we observe that for every strip $D'$ in $\mathcal{D}_1 \setminus \mathcal{M}(\mathcal{D}_1)$, there exists a finite collection of strips in $\mathcal{D}$ strictly containing $D'$. In particular, there exists a maximal one in $\mathcal{D}_1$, which then belongs to $\mathcal{M}(\mathcal{D}_1)$ and is unique by definition. Taking the union of the elements of the collection in both cases, we obtained the desired inclusion, proving \eqref{eq:maximal_covering}. 
	
	The monotonicity property in \eqref{eq:maximal_monotone} follows trivially.
\end{proof}

\begin{proof}[Proof of \eqref{eq:maximal_measure}, \eqref{eq:doubling_measure}, \eqref{eq:minimal_covering_measure}]
	The equalities in \eqref{eq:maximal_measure} follow by \eqref{eq:maximal_covering} and \eqref{eq:additivity_mu}.
	
	Now, we turn to the proof of the inequality in \eqref{eq:doubling_measure}. By \eqref{eq:maximal_covering}, we have
	\begin{equation*}
	\mathcal{N} \circ \mathcal{M} = \mathcal{N},
	\end{equation*}
	hence
	\begin{equation*}
	\mu(\mathcal{L}(\mathcal{N}(\mathcal{D}_1))) = \mu(\mathcal{L}(\mathcal{N}(\mathcal{M}(\mathcal{D}_1)))).
	\end{equation*}
	By \eqref{eq:maximal_covering} and \eqref{eq:additivity_mu}, we have
	\begin{gather*}
	\mu(\mathcal{L}(\mathcal{N}(\mathcal{M}(\mathcal{D}_1)))) = \mu(\mathcal{L}(\mathcal{M}(\mathcal{N}(\mathcal{M}(\mathcal{D}_1))))) =  \sum_{E \in \mathcal{M}(\mathcal{N}(\mathcal{M}(\mathcal{D}_1)))} \abs{\pi(E)}, \\
	\mu(\mathcal{L}(\mathcal{D}_1)) = \mu(\mathcal{L}(\mathcal{M}(\mathcal{D}_1))) = \sum_{E \in \mathcal{M}(\mathcal{D}_1)} \abs{\pi(E)}. 
	\end{gather*}
	By the disjointness of the elements in $\mathcal{M}(\mathcal{D}_1)$ and Lemma \ref{thm:intersection_strip_strip}, we can partition the collection $\mathcal{M}(\mathcal{D}_1)$ into pairwise disjoint subcollections $\mathcal{M}(\mathcal{D}_1)_{E}$, one for each element $E \in \mathcal{M}(\mathcal{N}(\mathcal{M}(\mathcal{D}_1)))$, so that
	\begin{equation*}
	\mathcal{L}(\mathcal{M}(\mathcal{D}_1)_{E}) \subseteq E.
	\end{equation*} 
	By the definition of $\mathcal{N}$, we have 
	\begin{equation*}
	\sum_{E \in \mathcal{M}(\mathcal{N}(\mathcal{M}(\mathcal{D}_1)))} \abs{\pi(E)} \leq 2 \sum_{E \in \mathcal{M}(\mathcal{N}(\mathcal{M}(\mathcal{D}_1)))} \sum_{F \in \mathcal{M}(\mathcal{D}_1)_E} \abs{\pi(F)} \leq  2 \sum_{F \in \mathcal{M}(\mathcal{D}_1)} \abs{\pi(F)}.
	\end{equation*}
	Together with the previous chains of equalities, this yields the desired inequality in \eqref{eq:doubling_measure}.
	
	Finally, we turn to the proof of the equality in \eqref{eq:minimal_covering_measure}. The inequality
	\begin{equation*}
	\mu(A) \leq \mu(\mathcal{L}(\mathcal{Q}(A))),
	\end{equation*}
	follows by \eqref{eq:minimal_covering_covering} and the monotonicity of $\mu$. The inequality 
	\begin{equation*}
	\mu(\mathcal{L}(\mathcal{Q}(A))) = \mu(\mathcal{L}(\mathcal{M}(\mathcal{Q}(A)))) \leq \mu(A),
	\end{equation*}
	follows by an argument analogous to the one used to prove \eqref{eq:additivity_mu} upon observing that for every $E$ in $\mathcal{M}(\mathcal{Q}(A))$, the intersection between $E_{+}$ and $A$ is nonempty.
\end{proof}

\begin{proof} [Proof of \eqref{eq:arbitrary_dyadic}]
	Without loss of generality, we assume the arbitrary tree $\widetilde{T} \in \mathcal{\widetilde{T}}$ to be of the form $\widetilde{T}(0,1,1)$, namely
	\begin{equation*}
	\widetilde{T}(0,1,1) = \{ (y,t,\eta) \colon y \in (0,1], t \in (0,1], \eta \in ( 1 - t^{-1}, 1 + t^{-1} ] \}.
	\end{equation*}
	Next, let $\widetilde{T}_0$ be the subset of $\widetilde{T}$ defined by
	\begin{equation*}
	\widetilde{T}_0 = \widetilde{T}(0,1,1) \cap (0,1] \times (0,1] \times (0,\infty).
	\end{equation*}
	Due to the monotonicity of $\nu$, it is enough to show that
	\begin{equation*}
	\nu(\widetilde{T}_0) = \infty.
	\end{equation*}
	Now, let $\mathcal{U}_0 \subseteq \mathcal{T}$ be a covering of $\widetilde{T}_0$ by dyadic trees. For every $l \in \N$, let $V_l$ be the subset of $\widetilde{T}_0$ defined by
	\begin{equation*}
	V_l = (0,1] \times (2^{-l-1},2^{-l}] \times (2^{l},2^{l}+1],
	\end{equation*}
	and let $\mathcal{U}_0(l)$ be the subcollection of $\mathcal{U}_0$ defined by its dyadic tree with nonempty intersection with $V_l$. In particular, we have
	\begin{equation*}
	V_l \subseteq \mathcal{L}(\mathcal{U}_0(l)),
	\end{equation*}
	and, for every $l' \in \N, l' \neq l$, for every $U \in \mathcal{U}_0(l)$, we claim that
	\begin{equation*}
	U \cap V_{l'} = \varnothing.
	\end{equation*}
	In particular, the dyadic tree $U$ has the form $T\big(m,-j,n(l,j)\big)$, where $j \in \Z, j \leq l$, $m \in \Z, 0 \leq m < 2^j$, and $n(l,j) \in \Z$ is defined by the condition
	\begin{equation*}
	I(n(l,j),j) \subseteq I(1,l).
	\end{equation*}
	If $j > l'$, we have
	\begin{gather*}
	U \subseteq \R \times (0,2^{l'-1}] \times \R, \\
	V_{l'} \subseteq \R \times (2^{l'-1},2^{l'}] \times \R,
	\end{gather*}
	yielding the desired disjointness. 
	
	If $j < l'$, we distinguish two cases.
	
	{\textbf{Case I: $l < l'$.}} We have 
	\begin{gather*}
	I(n(l,j),j) \subseteq I(1,l) \subseteq I(0,l'), \\
	(2^{l'},2^{l'}+1] \subseteq I(1,l'),
	\end{gather*} 
	yielding the desired disjointness. 
	
	{\textbf{Case II: $l > l'$.}} We have 
	\begin{gather*}
	I(n(l,j),j) \subseteq I(1,l), \\
	(2^{l'},2^{l'}+1] \subseteq I(1,l') \subseteq I(0,l),
	\end{gather*} 
	yielding the desired disjointness. 
		
	Therefore, the subcollections $\mathcal{U}_0(l)$ are pairwise disjoint, and we have
	\begin{equation*}
	\sum_{T \in \mathcal{U}_0} \tau(T) \geq \sum_{l \in \N} \sum_{T \in \mathcal{U}_0(l)} \tau(T) \geq \sum_{l \in \N} \nu(V_l).
	\end{equation*}
	It is enough to observe that, for every $l \in \N$, we have
	\begin{equation*}
	\nu(V_l) = 1.
	\end{equation*}
	In fact, for every covering $\mathcal{V}_l$ of $V_l$ by dyadic trees in $\mathcal{T}$, we have
	\begin{equation*}
	\pi(V_l) \subseteq \pi \big(\bigcup_{V \in \mathcal{V}_l} V \big) \subseteq \bigcup_{V \in \mathcal{V}_l} \pi(V),
	\end{equation*}
	hence
	\begin{equation*}
	1 = \abs{\pi(V_l)} \leq \sum_{V \in \mathcal{V}_l} \abs{\pi(V)} = \sum_{V \in \mathcal{V}_l} \tau(V).
	\end{equation*}
\end{proof}

\section{Approximation for outer $L^p$ spaces} \label{sec:approximation}

In this appendix, we exhibit the postponed proofs of the approximation results stated in Section \ref{sec:dyadic_upper_space}.

\begin{proof} [Proof of \ref{thm:approximation_max_size}]
	We have
	\begin{equation*}
	\norm{f}_{L^p_\mu(S)}^p \leq C \sum_{k \in \Z} 2^{kp} \mu(S(f) > 2^{k}).
	\end{equation*}
	In particular, there exists $k_0 \in \N$ such that, for every $\widetilde{k} \in \N, \widetilde{k} \geq k_0$, we have
	\begin{equation} \label{eq:up_finite_discrete}
	\norm{f}_{L^p_\mu(S)}^p \leq C \sum_{k \leq \widetilde{k}} 2^{kp} \mu(S(f) > 2^{k}).
	\end{equation}
	If $\mu(S(f) > 2^{k_0}) = 0$, we have that $f \in L^\infty_\mu(S)$, and we can take $A = X$.
	
	Otherwise, we claim that there exists $k_1 \in \N, k_1 > k_0$ such that
	\begin{equation} \label{eq:jump_level}
	\mu(S(f) > 2^{k_1-1}) > 2^p \mu(S(f) > 2^{k_1}).
	\end{equation}
	If not, for every $k \in \N, k > k_0$, we would have
	\begin{equation*}
	2^{kp} \mu(S(f) > 2^{k}) \geq 2^{k_0 p} \mu(S(f) > 2^{k_0}) > 0,
	\end{equation*}
	yielding the contradiction
	\begin{equation*}
	\norm{f}_{L^p_\mu(S)}^p \geq C \sum_{k = k_0 +1}^\infty 2^{kp} \mu(S(f) > 2^{k}) \geq C \sum_{k = k_0 +1}^\infty 2^{k_0 p} \mu(S(f) > 2^{k_0}) = \infty.
	\end{equation*}
	
	Now, let $B$ be an optimal set associated with $\mu(\ell^r(f) > 2^{k_1})$ up to a factor $2^{-1}(1+2^p)$, namely
	\begin{gather} 
	\label{eq:size_B}
	\norm{f 1_{B^c}}_{L^\infty_\mu(S)} \leq 2^{k_1}, \\
	\label{eq:measure_B}
	\mu(S(f) > 2^{k_1}) \leq \mu(B) \leq \frac{1+2^p}{2} \mu(S(f) > 2^{k_1}),
	\end{gather}
	and define $A = B^c$, so that $f 1_A \in L^\infty_\mu(S)$.
	
	We claim that for every $k \in \N, k < k_1$, we have
	\begin{equation} \label{eq:claim}
	\mu(S(f 1_{A}) > 2^{k}) \geq \frac{1-2^{-p}}{2} \mu(S(f) > 2^{k}). 
	\end{equation}
	If not, there would exist $\widetilde{k} \in \N, \widetilde{k} < k_1$ such that
	\begin{equation*} 
	\mu(S(f 1_{A}) > 2^{\widetilde{k}}) < \frac{1-2^{-p}}{2} \mu(S(f) > 2^{\widetilde{k}}). 
	\end{equation*}
	yielding the contradiction
	\begin{align*}
	\mu(S(f) > 2^{\widetilde{k}}) & \leq \mu(S(f 1_{A}) > 2^{\widetilde{k}}) + \mu(B) \\
	& < \frac{1-2^{-p}}{2} \mu(S(f) > 2^{\widetilde{k}}) + \frac{1+2^{p}}{2} 2^{-p} \mu(S(f) > 2^{k_1-1}) \\
	& \leq \mu(S(f) > 2^{\widetilde{k}}),
	\end{align*}
	where we used \eqref{eq:size_B} and the subadditivity of $\mu$ in the first inequality, \eqref{eq:measure_B} and \eqref{eq:jump_level} in the second, and the monotonicity of the super level measure $\mu(S(f)> \lambda)$ in $\lambda$ in the third. 
	
	Therefore, by \eqref{eq:up_finite_discrete} and \eqref{eq:claim}, we have
	\begin{equation*}
	\norm{f}_{L^p_\mu(S)}^p \leq C \sum_{k < k_1} 2^{kp} \mu(S(f) > 2^{k}) \leq C \sum_{k < k_1} 2^{kp} \mu(S(f 1_{A}) > 2^{k}) \leq C \norm{f 1_{A}}_{L^p_\mu(S)}^p.
	\end{equation*}
\end{proof}

\begin{proof} [Proof of Lemma \ref{thm:monotone_approximation_size}]
	Without loss of generality, upon normalization of $f$, we assume that 
	\begin{equation*}
	1 < \norm{f}_{L^\infty_\nu(\ell^r_\omega)} \leq 2.
	\end{equation*}
	
	For every $k \in \Z, k > 0$, the super level measure of $f$ associated with the level $2^k$ is zero, and the desired inequality is trivially satisfied.
	
	For the remaining $k \in \Z, k \leq 0$, we prove the desired inequality by induction. In particular, we prove that there exist constants $C=C(r)$, $c=c(r)$, and a bounded sequence $\{ C_k \colon C_k < C, k \in \Z, k \leq 0 \}$ such that
	\begin{equation*}
	\nu(\ell^r_\omega(f) > 2^k) \leq C_k \sum_{l \geq k} \nu(\ell^r_\omega(f 1_{X_j}) > c 2^{l} ).
	\end{equation*}
	
	{\textbf{Case I: $k=0$.}} By the $r$-orthogonality of the classical $L^r$ quasi-norm on sets with disjoint supports, there exists a set $B_{0}$ such that
	\begin{gather} \label{eq:measure_B_0_size}
	\ell^r_\omega(f)(B_{0}) > 1, \\ \label{eq:measure_B_0_measure}
	\nu(\ell^r_\omega(f) > 1) \leq \nu(B_{0}).
	\end{gather}
	By the monotonicity of the classical $L^r$ quasi-norm and \eqref{eq:measure_B_0_size}, there exists $j \in \N$ such that
	\begin{equation*}
	\ell^r_\omega(f 1_{X_j})(B_{0}) > 1.
	\end{equation*}
	Since we have
	\begin{equation*}
	\norm{f 1_{X_j}}_{L^\infty_\nu(\ell^r_\omega)} \leq \norm{f}_{L^\infty_\nu(\ell^r_\omega)} \leq 2,
	\end{equation*}
	we obtain, by Lemma \ref{thm:atomic_super_level_set_interior_level}, 
	\begin{equation*}
	\nu(B_{0}) \leq C_0 \nu(\ell^r_\omega(f 1_{X_j}) > c ).
	\end{equation*}
	Together with \eqref{eq:measure_B_0_measure}, this yields the desired inequality.
	
	{\textbf{Case II: $k < 0$.}} We assume that there exists $j=j(r,f,k+1) \in \N$ such that
	\begin{equation} \label{eq:induction_hypothesis}
	\nu(\ell^r_\omega(f) > 2^{k+1}) \leq C_{k+1} \sum_{l \geq k+1} \nu(\ell^r_\omega(f 1_{X_j}) > c 2^{l} ).
	\end{equation}
	Now, for every $\varepsilon > 0$, there exists a set $A_{k+1}$ such that
	\begin{gather} \label{eq:measure_A_k_plus_1_size}
	\norm{f 1_{A_{k+1}^c}}_{L^\infty_\nu(\ell^r_\omega)} \leq 2^{k+1},	\\ \label{eq:measure_A_k_plus_1_measure}
	\nu(\ell^r_\omega(f) > 2^{k+1}) \leq \nu(A_{k+1}) \leq (1+\varepsilon) \nu(\ell^r_\omega(f) > 2^{k+1}).
	\end{gather}
	We will fix $\varepsilon$ later. In particular, we have
	\begin{equation} \label{eq:induction_step}
	\nu(\ell^r_\omega(f) > 2^{k}) \leq \nu(A_{k+1}) + \nu(\ell^r_\omega(f 1_{A_{k+1}^c}) > 2^{k}).
	\end{equation}
	If we have 
	\begin{equation*} 
	\norm{f 1_{A_{k+1}^c}}_{L^\infty_\nu(\ell^r_\omega)} \leq 2^{k},
	\end{equation*}
	we obtain 
	\begin{equation*}
	\nu(\ell^r_\omega(f) > 2^{k}) \leq \nu(A_{k+1}) \leq (1+\varepsilon) C_{k+1} \sum_{l \geq k+1} \nu(\ell^r_\omega(f 1_{X_j}) > c 2^{l} ) .
	\end{equation*}
	Otherwise, we have
	\begin{equation*}
	2^{k} < \norm{f 1_{A_{k+1}^c}}_{L^\infty_\nu(\ell^r_\omega)} \leq 2^{k +1}.
	\end{equation*}
	Applying to the function $f 1_{A_{k+1}^c}$ an argument analogous to that of the previous case, we obtain $j=j(r,f,{k}) \in \N$, without loss of generality greater than $j(r,f,{k+1})$, such that
	\begin{equation*}
	\nu(\ell^r_\omega(f 1_{A_{k+1}^c}) > 2^k) \leq C_0 \nu(\ell^r_\omega(f 1_{A_{k+1}^c} 1_{X_j}) > c 2^k ) \leq C_0 \nu(\ell^r_\omega(f 1_{X_j}) > c 2^k).
	\end{equation*}
	Together with \eqref{eq:induction_step}, \eqref{eq:measure_A_k_plus_1_measure}, and \eqref{eq:induction_hypothesis}, the previous chain of inequalities yields 
	\begin{equation*}
	\nu(\ell^r_\omega(f) > 2^{k}) \leq (1+\varepsilon) C_{k+1} \sum_{l \geq k+1} \nu(\ell^r_\omega(f 1_{X_j}) > c 2^{l} ) + C_0 \nu(\ell^r_\omega(f 1_{X_j}) > c 2^k).
	\end{equation*}
	By choosing $\varepsilon = \varepsilon(k) = 2^{2^{k}} -1$ and defining $C_k = 2^{1 - 2^{k}} C_0, C = 2 C_0$, we obtain the desired inequality.
\end{proof}

\begin{proof} [Proof of Lemma \ref{thm:monotone_approximation_size_exterior}]
		
	The proof is analogous to that of Lemma \ref{thm:monotone_approximation_size} upon the following observation. Without loss of generality, it is enough to comment in the case
	\begin{equation*}
	1 < \norm{f}_{L^\infty_\mu(\ell^q_\nu(\ell^r_\omega))} \leq 2.
	\end{equation*}
	Therefore, for every dyadic strip $E \in \mathcal{D}$, we have $f 1_E \in L^q_\nu(\ell^r_\omega)$. Moreover, there exists a collection of maximal dyadic strips $\{ E_{n} \colon E_{n} \in \mathcal{D}, n \in \N \}$ such that
	\begin{gather*}
	\ell^q_\nu(\ell^r_\omega)(f)(E_n) > 1, \\
	\mu(\ell^q_\nu(\ell^r_\omega)(f) > 1) \leq \sum_{n \in \N} \mu(E_n).
	\end{gather*}
	In particular, there exists a finite subcollection such that
	\begin{equation*}
	\mu(\ell^q_\nu(\ell^r_\omega)(f) > 1) \leq 2 \sum_{n=1}^N \mu(E_n).
	\end{equation*}
	Since the dyadic strips are maximal, then they are disjoint, hence, by Lemma \ref{thm:measure_disjoint_union}, they are $\nu$-Carath\'{e}odory with parameter $1$. 
	
	Now we apply an argument analogous to that used to prove Lemma \ref{thm:monotone_approximation_size} with the monotonicity of the classical $L^r$ quasi-norms replaced by Lemma \ref{thm:approximation_tree}, and Lemma \ref{thm:atomic_super_level_set_interior_level} replaced by Lemma \ref{thm:atomic_super_level_set_exterior_level}.
\end{proof}

\bibliographystyle{amsplain}
\bibliography{mybibliography}

\end{document}